\documentclass[11pt,reqno]{amsart}
\usepackage{a4wide}
\usepackage{amsfonts,amsmath,amssymb,amscd}
\usepackage{amsthm}
\usepackage{amsmath,todonotes}
\usepackage{amscd}
\usepackage{cancel}
\usepackage{verbatim}
\usepackage{color}
\usepackage[all]{xy}
\usepackage[mathscr]{eucal}
\usepackage{mathrsfs}
\usepackage{fullpage}
\usepackage{enumitem}
\usepackage{hyperref}
\usepackage{bbm}
\usepackage{qtree}
\usepackage{mathrsfs}
\usepackage{bm}
\usepackage{bigints}
\usepackage{comment}
\usepackage{mathtools}
\allowdisplaybreaks
\let\dotlessi=\i
\newcommand{\pFq}[5]{\ensuremath{{}_{#1}F_{#2} \left( \genfrac{}{}{0pt}{}{#3}
		{#4} \bigg| {#5} \right)}}
\renewcommand{\a}{\alpha}
\renewcommand{\b}{\beta}

\newcommand{\G}{\Gamma}

\newcommand{\z}{\zeta}
\usepackage[headheight=110pt,top=0.6in, bottom=0.8in, left=0.6in, right=0.6in]{geometry}
\numberwithin{equation}{section}

\newtheorem{remark}[]{Remark}

\newtheorem{theorem}{Theorem}

\newtheorem{lemma}[theorem]{Lemma}

\newtheorem{corollary}[theorem]{Corollary}

\newtheorem{proposition}[theorem]{Proposition}

\theoremstyle{definition}

\numberwithin{theorem}{section} 
\numberwithin{equation}{section}
\numberwithin{table}{section}

\newcommand{\C}{\mathbb{C}}

\newcommand{\N}{\mathbb{N}}

\newcommand{\re}{\textnormal{Re}}
\newcommand{\im}{\textnormal{Im}}

\renewcommand{\(}{\left(}
\renewcommand{\)}{\right)}

\makeatletter
\def\proof{\@ifnextchar[{\@oproof}{\@nproof}}
\def\@oproof[#1][#2]{\trivlist\item[\hskip\labelsep\textit{#2 Proof of\
		#1.}~]\ignorespaces}
\def\@nproof{\trivlist\item[\hskip\labelsep\textit{Proof.}~]\ignorespaces}
\makeatother

\begin{document}
	\title[]{Explicit transformations for generalized Lambert series associated with the divisor function $\sigma_{a}^{(N)}(n)$ and their applications}
\author{Soumyarup Banerjee}
\address{Department of Mathematics and Statistics, Indian Institute of Science Education and Research Kolkata, Mohanpur, Nadia 741246, West Bengal, India}

\email{soumya.tatan@gmail.com}

\author{Atul Dixit}
\address{Discipline of Mathematics, Indian Institute of Technology Gandhinagar, Palaj, Gandhinagar 382355, Gujarat, India} 
\email{adixit@iitgn.ac.in}

\author{Shivajee Gupta}
\address{Discipline of Mathematics, Indian Institute of Technology Gandhinagar, Palaj, Gandhinagar 382355, Gujarat, India} 
\email{shivajee.o@iitgn.ac.in}

	\medskip
	\begin{abstract}
	Let $\sigma_a^{(N)}(n)=\sum_{d^{N}|n}d^a$. An explicit transformation is obtained for the generalized Lambert series  $\sum_{n=1}^{\infty}\sigma_{a}^{(N)}(n)e^{-ny}$ for $\re(a)>-1$ using the recently established Vorono\"i summation formula for $\sigma_a^{(N)}(n)$, and is extended to a wider region by analytic continuation. For $N=1$, this Lambert series plays an important role in string theory scattering amplitudes as can be seen in the recent work of Dorigoni and Kleinschmidt. These transformations exhibit several identities - a new generalization of Ramanujan's formula for $\zeta(2m+1)$, an identity associated with extended higher Herglotz functions, generalized Dedekind eta-transformation, Wigert's transformation etc., all of which are derived in this paper, thus leading to their uniform proofs. A special case of one of these explicit transformations naturally leads us to consider generalized power partitions with ``$n^{2N-1}$ copies of $n^{N}$''. Asymptotic expansion of their generating function as $q\to1^{-}$ is also derived which generalizes Wright's result on the plane partition generating function. In order to obtain these transformations, several new intermediate results are required, for example, a new reduction formula for Meijer $G$-function and an almost closed-form evaluation of $\left.\frac{\partial E_{2N, \b}(z^{2N})}{\partial\beta}\right|_{\beta=1}$, where $E_{\alpha, \beta}(z)$ is a two-variable Mittag-Leffler function.
	\end{abstract}
	\thanks{}
\maketitle
	\vspace{-1cm}
	\tableofcontents
	\vspace{-1.3cm}

	\section{Introduction}\label{intro}
	
	Lambert series are series of the form
	%\begin{equation}
	$\displaystyle\sum_{n=1}^{\infty}a(n)\frac{q^{n}}{1-q^{n}}$,
	%\end{equation}
	where $a(n)$ is an arithmetic function and $q$ is a complex number such that $|q|<1$. They can be written in the equivalent form
	%\begin{equation}
	$\displaystyle\sum_{n=1}^{\infty}\frac{a(n)}{e^{ny}-1}$
	%\end{equation}
	by letting $q=e^{-y}$, where Re$(y)>0$.
	Lambert series are central objects in number theory. For example, when $a(n)=n^{k}$, where $k\geq3$ is an odd integer, these Lambert series $\displaystyle\sum_{n=1}^{\infty}\frac{n^k}{e^{ny}-1}=\sum_{n=1}^{\infty}\sigma_{k}(n)e^{-ny}$, where $\sigma_{a}(n):=\sum_{d|n}d^a$, are \emph{essentially} Eisenstein series on $\textup{SL}_{2}(\mathbb{Z})$. It is well-known that Eisenstein series satisfy elegant modular transformations and  are prototypical examples of modular forms on the full modular group. 
	
	Ramanujan \cite[p.~173, Ch. 14, Entry 21(i)]{ramnote}, \cite[p.~319-320, formula (28)]{lnb}, \cite[p.~275-276]{bcbramsecnote} derived a famous transformation which involves the value of the Riemann zeta function at an odd positive integer, that is, $\zeta(2m+1)$, and the Lambert series given above, namely, for Re$(\a), \textup{Re}(\b)>0$ with $\a\b=\pi^2$ and $m\in\mathbb{Z}\backslash\{0\}$,
	\begin{align}\label{zetaodd}
		\a^{-m}\left\{\frac{1}{2}\zeta(2m+1)+\sum_{n=1}^{\infty}\frac{n^{-2m-1}}{e^{2\a n}-1}\right\}&=(-\b)^{-m}\left\{\frac{1}{2}\zeta(2m+1)+\sum_{n=1}^{\infty}\frac{n^{-2m-1}}{e^{2\b n}-1}\right\}\nonumber\\
		&\quad-2^{2m}\sum_{j=0}^{m+1}\frac{(-1)^jB_{2j}B_{2m+2-2j}}{(2j)!(2m+2-2j)!}\a^{m+1-j}\b^j,
	\end{align}
	where $B_j$ are the Bernoulli numbers and $\zeta(s)=\sum_{n=1}^{\infty}n^{-s}$ for Re$(s)>1$.
	Ramanujan's formula not only encapsulates the aforementioned transformations of Eisenstein series on $\textup{SL}_{2}(\mathbb{Z})$ but also those of their associated Eichler integrals. For more information on this, we refer the reader to \cite{berndtstraubzeta}.
	
	The arithmetic nature of the odd zeta values is quite mysterious. Ap\'{e}ry showed that $\zeta(3)$ is irrational which is the only concrete result we have to this day as far as a single odd zeta value is concerned. Zudilin's result that at least one of $\zeta(5), \zeta(7), \zeta(9)$ or $\zeta(11)$ is irrational is the next best result we have for other odd zeta values. Compare this with the representation of the even zeta values given by Euler's formula, namely,
	\begin{equation}\label{ef}
		\zeta (2m) =(-1)^{m+1}\frac{(2\pi)^{2m}B_{2m}}{2(2m)!},
	\end{equation}
which at once establishes their transcendence as $\pi$ is transcendental.

	At the bottom of page 332 of \cite{lnb}, Ramanujan considered the generalized Lambert series
	\begin{equation}\label{gls}
		\mathfrak{F}_{k, N}(y):=\sum_{n=1}^{\infty}\frac{n^k}{e^{n^{N}y}-1},
	\end{equation}
	where $N$ is a positive integer and $k-N$ is any even integer.  However, he did not give any transformation for $\mathfrak{F}_{k, N}(y)$. Obviously, $\mathfrak{F}_{k, 1}(y)$ is the aforementioned Lambert series associated with Eisenstein series on $\textup{SL}_{2}(\mathbb{Z})$. 
	
	Meanwhile, Wigert \cite[p.~8-9, Equation (1.5)]{wig} derived an exact transformation for $\mathfrak{F}_{0, N}(y)$ when $N$ is even, but only an asymptotic formula when $N$ is odd. Kanemitsu, Tanigawa and Yoshimoto \cite[Theorem 1]{ktyhr} obtained a transformation for $\mathfrak{F}_{N-2h, N}(y)$ when $N\in\mathbb{N}$ and $0<h\leq N/2$, and another one in a subsequent work \cite[Theorem 2.1]{ktyacta} when $h\geq N/2$ and $N$ is even.
	
	The second author and Maji \cite[Theorem 1.1]{dixitmaji1} extended the aforementioned results of Kanemitsu, Tanigawa and Yoshimoto to any $h\in\mathbb{Z}$ and $N\in\mathbb{N}$, and as a result, gave the following one-parameter generalization of Ramanujan's formula \eqref{zetaodd} (see \cite[Theorem 1.2]{dixitmaji1}):
	
	{\textit Let $N$ be an odd positive integer and $\a,\b>0$ such that $\a\b^{N}=\pi^{N+1}$. Then for any non-zero integer $m$,
		\begin{align}\label{zetageneqn}
			&\a^{-\frac{2Nm}{N+1}}\left(\frac{1}{2}\zeta(2Nm+1)+\sum_{n=1}^{\infty}\frac{n^{-2Nm-1}}{\textup{exp}\left((2n)^{N}\a\right)-1}\right)\nonumber\\
			&=\left(-\b^{\frac{2N}{N+1}}\right)^{-m}\frac{2^{2m(N-1)}}{N}\Bigg(\frac{1}{2}\zeta(2m+1)+(-1)^{\frac{N+3}{2}}\sum_{j=\frac{-(N-1)}{2}}^{\frac{N-1}{2}}(-1)^{j}\sum_{n=1}^{\infty}\frac{n^{-2m-1}}{\textup{exp}\left((2n)^{\frac{1}{N}}\b e^{\frac{i\pi j}{N}}\right)-1}\Bigg)\nonumber\\
			&\quad+(-1)^{m+\frac{N+3}{2}}2^{2Nm}\sum_{j=0}^{\left\lfloor\frac{N+1}{2N}+m\right\rfloor}\frac{(-1)^jB_{2j}B_{N+1+2N(m-j)}}{(2j)!(N+1+2N(m-j))!}\a^{\frac{2j}{N+1}}\b^{N+\frac{2N^2(m-j)}{N+1}}.
	\end{align}}
	It is easy to see that one recovers \eqref{zetaodd} upon letting $N=1$ in the above identity. Recently, a new method based on convolution of ``Dirichlet series'' was given in \cite{ccvw} to prove \eqref{zetageneqn} and other such identities.
	
	Observe that the function $\mathfrak{F}_{k, N}(y)$ in \eqref{gls} can be written in the form
	\begin{equation}\label{fkny}
		\mathfrak{F}_{k, N}(y)=\sum_{n=1}^{\infty}\sigma_{k}^{(N)}(n)e^{-ny},
	\end{equation}
	where, for any $a\in\mathbb{C}$ and $N\in\mathbb{N}$,
	\begin{equation}\label{defaf}
		\sigma_a^{(N)}(n) := \sum_{d^N\mid n} d^a
	\end{equation}
	is a generalization of $\sigma_a(n)$, which seems to first appear explicitly in the work of Crum \cite{crum} although its special case $a=0$, and denoted by $d^{(N)}(n)$, was studied long before by Wigert \cite{wig}. This function also briefly appeared in the work of Berndt, Roy, Zaharescu and the second author \cite[Section 10.2]{bdrz1}.
	
	The generalized divisor function $\sigma_a^{(N)}(n)$ turns up in various areas of mathematics. For example, consider the generalized Ramanujan sum introduced by Cohen \cite{cohen}, namely,
	\begin{equation*}
		c_{\ell, N}(n):=\sum_{b=1\atop{(b,\ell^N)_N=1}}^{\ell^N} \exp{\left(\frac{2\pi ibn}{\ell^{N}}\right)},
	\end{equation*} 
	where, by $(b,\ell^N)_N=1$, we mean there is no prime $p$ such that $p|\ell$ and $p^{N}|b$. The Dirichlet series associated with $c_{\ell, N}(n)$ then satisfies \cite[Theorem 4]{cohen} (see also \cite[p.~163]{mccarthy})
	\begin{equation*}
		\zeta(s)\sum_{\ell=1}^{\infty}\frac{c_{\ell, N}(n)}{\ell^s}=\sigma_{1-\frac{s}{N}}^{(N)}(n)
	\end{equation*}
	for $s>N$. The special case $N=1$ of this result is due to Ramanujan \cite{ram1918}.
	
	Furthermore, if $p_{k}(n)$ denotes the number of partitions of an integer $n$ into $k$-th powers, commonly known as \emph{power partitions of order $k$}, one encounters the generalized Lambert series $\mathfrak{F}_{N, N}(y)$, defined in \eqref{fkny}, in the new proof of the asymptotic expansion of $p_k(n)$ given by Tenenbaum, Wu and Li \cite[Equation (2.4)]{tenenbaum} using the saddle-point method. 
	
	Very recently, Zagier \cite{zagierhrj}, in his study of power partitions, generalized the modular transformation satisfied by the Dedekind eta function $\eta(z):=e^{\frac{i\pi z}{12}}\prod\limits_{n=1}^{\infty}(1-e^{2\pi inz}), \textup{Im}(z)>0$, namely $\eta(-1/z)=\sqrt{-iz}\eta(z)$, by showing that the generalized eta-function $\eta_s(z)$
	\begin{equation}\label{Zagier eta}
		\eta_s(z):=\exp{\left(-\pi i\zeta(-s)z\right)}\prod_{n=1}^{\infty}(1-\exp{(2\pi in^{s}z)})\hspace{4mm}(z\in\mathbb{H}, s\in\mathbb{R}^{+})
	\end{equation}
	satisfies
	\begin{equation}\label{Zagier}
		\eta_N(-1/z)=(2\pi)^{(N-1)/2}\sqrt{-iz}\prod_{w\in\mathbb{H}\atop{w^N=\pm z}}\eta_{1/N}(w),
	\end{equation}
	for $N\in\mathbb{N}$. Letting $s=N$ in \eqref{Zagier eta} and then taking logarithm on both sides of the resulting identity leads to \eqref{log eta} below, which again involves the function $\sigma_{N}^{(N)}(n)$.
	
	Zagier's generalized eta function in \eqref{Zagier eta} as well as his transformation \eqref{Zagier} both occur\footnote{Ramanujan gave equivalent definition of the generalized eta function for $x>0$, but it can be easily extended to Re$(x)>0$.} in an equivalent form on page $330$ of Ramanujan's Lost Notebook \cite{lnb} although he proves two further results involving the function \cite[Theorems 1, 3]{zagierhrj}. See \cite[p.~234-238]{RLNII} for the proof of the equivalent form of \eqref{Zagier} given by Ramanujan. The case $N=2$ of \eqref{Zagier} is given without proof by Hardy and Ramanujan in \cite{ramhar}. Kr\"{a}tzel \cite{kratzel2} further generalized \eqref{Zagier} by means of his generalization of Ramanujan's equivalent form of \eqref{Zagier}.
	In \S \ref{zagierjrh}, we will show that the logarithmic derivative of \eqref{Zagier} is nothing but  the special case $m=-1$ of our Theorem \ref{zetagen3} given in Corollary \ref{zagiereqvt}.
	
 In addition to working with $\mathfrak{F}_{0, N}(y)$, Wigert \cite{Wig} also worked with the generalized partial theta function $\sum_{n=1}^{\infty}\exp{(-n^{N}x)}$ and obtained a transformation for it.
 
	Koshliakov \cite{koshwigleningrad} obtained the Vorono\"{\dotlessi} summation formula for $\sigma_{0}^{(N)}(n)$, or, equivalently $d^{(N)}(n)$, although he does not give any indication of how he proved one of the crucial intermediate results \cite[Equation (6)]{koshwigleningrad} used in deriving his formula. A generalization of this intermediate result of Koshliakov was recently obtained in \cite[Theorem 2.1]{DMV}. Its proof is deep, employing the uniqueness theorem in the theory of linear differential equations along with the properties of Stirling numbers of the second kind and elementary symmetric polynomials. Obviously, it also rigorously proves Equation (6) in Koshliakov's paper \cite{koshwigleningrad}. Koshliakov's Vorono\"{\dotlessi} summation formula for $\sigma_{0}^{(N)}(n)$ is stated next.\\
	
	\textit{ Let $0 < \alpha < \beta$ and $ \alpha,  \beta \not\in \mathbb{Z}$.  Let $N>1$ be a natural number. Let $f(x)$ be an analytic function defined inside a closed contour containing $[\alpha,  \beta]$. Then
		\begin{align}\label{koshvsf}
			\sum_{\a<n<\b} \sigma_{0}^{(N)}(n)f(n) & = \int_{\a}^{\b} \left( \zeta(N) +\frac{1}{N}\zeta\left(\frac{1}{N}\right) y^{1/N-1}\right) f(y) dy \nonumber\\
			& + 4 (2\pi)^{1/N-1} \sum_{n=1}^\infty S_{0}^{(N)}(n) \int_{\a}^\b H^{(N)}\left( (2\pi)^{1+ 1/N} (ny)^{1/N} \right)y^{\frac{1}{N}-1} f(y) dy,
		\end{align}
		%where $S^{(k)}(n):= \sum_{d_{1}^k d_2=n} d_{2}^{\frac{1}{k}-1}$ and  $H^{(k)}(x):= \int_{0}^\infty \cos(1/t^k) \cos(x t) t^{-k} dt.$
		where
		\begin{align*}
			S^{(N)}(n) := \sum_{d_1^N d_2 = n} d_2^{\frac{1}{N}-1},\hspace{5mm}
			H^{(N)}(x):= \int_{0}^\infty t^{-N}\cos(t^{-N}) \cos(x t)\, dt.
	\end{align*}}
	Before Koshliakov, Wigert \cite[Equation (B)]{wigannalen} obtained a Vorono\"{\dotlessi} summation formula for the Riesz sum $\sum_{n\leq x}\sigma_{0}^{(N)}(n)(x-n)^p$ for $p>1$. Recently, the second author, Maji and Vatwani \cite{DMV} obtained the Vorono\"{\dotlessi} summation formula for $\sigma_{a}^{(N)}(n)$ and containing the test function $f$, thereby generalizing Koshliakov's \eqref{koshvsf}. This result is now given.
	
	\textit{Let $0 < \alpha < \beta$ and $ \alpha,  \beta \not\in \mathbb{Z}$.  Let $N\in\mathbb{N}$ and $a\in\mathbb{C}$ with $-1 < \re(a) <N$ and $a \neq N-1$.  Define
		\begin{equation}
			S_{a}^{(N)}(n):= \sum\limits_{d_{1}^N d_2=n} d_{2}^{\frac{1+a}{N}-1}\label{defbf}
			\end{equation}
			and let $f(x)$ be analytic inside a closed contour containing $[\alpha,  \beta]$.  Then
		\begin{align*}
			\sum_{ \alpha < n < \beta }  \sigma_a^{(N)}(n)f(n) & =  \int_{\alpha}^{\beta}  f(t) \left( \zeta(N-a) + \frac{1}{N} t^{\frac{1+a}{N}-1} \zeta\left( \frac{1+a}{N} \right)\right) \, dt \nonumber  \\
			&\quad+2 (2\pi)^{\frac{1+a}{N}-1}  \sum_{n=1}^{\infty}  S_{a}^{(N)}(n)  \int_{\alpha}^{\beta }  f(t)  t^{\frac{1+a}{N}-1} H_a^{(N)}\left( ( 2\pi)^{\frac{1}{N}+1} (n t)^{\frac{1}{N}} \right)\, dt,
		\end{align*} 
		where\footnote{We emphasize here that the notation $H_a^{(N)}(x)$ \emph{does not} mean $N$-th derivative of some function $H_a(x)$. This notation is used so as to be consistent with that used by Wigert \cite{wig} and Koshliakov \cite{koshwigleningrad} for the associated arithmetic as well as special functions. We retain it throughout the paper for other functions as well.}
			\begin{equation}\label{HaNx}
			H_a^{(N)}(x) := \int_0^\infty t^{a-N}\cos(xt)\cos\left(\frac{1}{t^N}\right)\, dt.
		\end{equation}
		Moreover, if $a=N-1$, then
		\begin{align*}
			\sum_{ \alpha < n < \beta }  \sigma_{N-1}^{(N)}(n)f(n) & =  \int_{\alpha}^{\beta}  f(t) \left(\frac{(N+1)\gamma+\log(t)}{N}\right)\, dt \nonumber  \\
			&\quad+2 \sum_{n=1}^{\infty}  S_{N-1}^{(N)}(n)  \int_{\alpha}^{\beta }  f(t) H_{N-1}^{(N)}\left( ( 2\pi)^{\frac{1}{N}+1} (n t)^{\frac{1}{N}} \right)\, dt,
		\end{align*}
	where $\gamma$ is Euler's constant. }
	
%		Let $0 < \alpha < \beta$ and $ \alpha,  \beta \not\in \mathbb{Z}$.  Let $N>1$ be a natural number and $a$ be a complex number with $-1 < \re(a) <N$ and $a \neq N-1$.  Let $f(x)$ be an analytic function defined inside a closed contour containing $[\alpha,  \beta]$.  Then
%		\begin{align}
%			\sum_{ \alpha < n < \beta } \sigma_a^{(N)}(n)f(n) & =  \int_{\alpha}^{\beta}  f(t) \left( \zeta(N-a) + \frac{1}{N} t^{\frac{1+a}{N}-1} \zeta\left( \frac{1+a}{N} \right)\right) \mathrm{d}t +  2 (2\pi)^{\frac{1+a}{N}-1}  \sum_{n=1}^{\infty}  S_{a}^{(N)}(n) \nonumber  \\
%			& \times   \int_{\alpha}^{\beta }  f(t)  t^{\frac{1+a}{N}-1} H_a^{(N)}\left( ( 2\pi)^{\frac{1}{N}+1} (n t)^{\frac{1}{N}} \right)dt,
%		\end{align}
%		where
%		\begin{align}\label{saNn}
%			S_a^{(N)}(n) &:= \sum_{d_1^N d_2 = n} d_2^{\frac{1+a}{N}-1},\\
%			H_a^{(N)}(x)&:= \int_{0}^\infty t^{a-N}\cos(1/t^N) \cos(x t) dt.
%	\end{align}}
	When $a=0$, this gives \eqref{koshvsf}. The following ``infinite'' version of the above result was also obtained in  \cite[Theorem 2.4]{DMV}.
	
	\textit{Let $N \in \mathbb{N}, a \in \mathbb{C}$ be such that $ -1 < \re(a) < N$. 
		Let $f \in \mathscr S (\mathbb R)$, where $\mathscr S (\mathbb R)$ denotes the space of Schwartz functions on $\mathbb R$, and let $F(s)=\mathcal{M}(f)(s)$ be the Mellin transform of $f$. 
		Then, for $a\neq N-1$,
		\begin{align}\label{ainfinite}
			\sum_{n=1}^\infty \sigma_{a}^{(N)}(n) f(n) & = - \frac{1}{2}\zeta(-a)f(0^{+}) + \zeta(N-a) \int_{0}^\infty f(y) dy + \frac{1}{N} F\left( \frac{1+a}{N} \right)\zeta\left( \frac{1+a}{N} \right) \nonumber\\ 
			& + \frac{(2 \pi)^{(N+1)\left(\frac{1+a}{N}\right)-a}}{\pi^2} \sum_{n=1}^\infty S_{a}^{(N)}(n) \int_{0}^\infty H_a^{(N)}\left( (2\pi)^{1+ \frac{1}{N}} (ny)^{\frac{1}{N}} \right)y^{\frac{1+a}{N}-1} f(y)\, dy.
	\end{align}
Moreover, when $a=N-1$,
\begin{align*}
	\sum_{n=1}^\infty \sigma_{N-1}^{(N)}(n) f(n) & = -\frac{1}{2}\zeta(1-N)f(0^{+}) + 	\int_{0}^{\infty}  f(t) \left(\frac{(N+1)\gamma+\log(t)}{N}\right) dt \nonumber\\ 
	& + 4 \sum_{n=1}^\infty S_{N-1}^{(N)}(n) \int_{0}^\infty H_{N-1}^{(N)}\left( (2\pi)^{1+ \frac{1}{N}} (ny)^{\frac{1}{N}} \right)f(y)\, dy.
\end{align*}}

Several properties of the kernel $H_a^{(N)}(x)$ are obtained in \cite[Section 4]{DMV}. For  $N=1$, the kernel $H_a^{(N)}(x)$ reduces to a combination of the $K$-, $Y$- and $J$-Bessel functions and thus \eqref{ainfinite} gives the Vorono\"{\dotlessi} summation formula for $\sigma_a(n)$ \cite[Section 6]{bdrz1}, a symmetric version of which was given by Guinand \cite[Equation (1)]{guinand}.
%	and is as follows:
%	Let $-\frac{1}{2}<\textup{Re}(a)<\frac{1}{2}$. Let $f(x)$ and $f'(x)$ be integrals, $f$ tend to zero as $x\rightarrow\infty,f(x),xf'(x)$ and $x^2f''(x)$ belong to $L^2(0,\infty)$. Let
%	\begin{equation}\label{fgkt}
%		g(x)=2\pi\int_0^\infty f(t)\left(\cos\left(\frac{\pi a}{2}\right)\left(\frac{2}{\pi}K_{a}(4\pi\sqrt{xt})-Y_{a}(4\pi\sqrt{xt})\right)-\sin\left(\frac{\pi a}{2}\right)J_a(4\pi\sqrt{xt})\right)dt,
%	\end{equation}
%	where $J_a(x), Y_a(x)$ are Bessel functions of the first and second kinds respectively, and $K_a(x)$ is the modified Bessel function of the second kind. Then,
%	\begin{align}\label{guinandSummationFormulaeqn}
%		&\sum_{n=1}^\infty\sigma_{-a}(n)n^{\frac{a}{2}}f(n)-\zeta(1+a)\int_0^\infty x^{\frac{a}{2}}f(x)\, dx-\zeta(1-a)\int_0^\infty x^{-\frac{a}{2}}f(x)\, dx\nonumber\\
%		&=\sum_{n=1}^\infty\sigma_{-a}(n)n^{\frac{a}{2}}g(n)-\zeta(1+a)\int_0^\infty x^{\frac{a}{2}}g(x)\, dx-\zeta(1-a)\int_0^\infty x^{-\frac{a}{2}}g(x)\, dx.
%	\end{align} 
%	Here, `$f(x)$ is an integral' means $f$ is an integral of some function, that is, $f$ can be written in the form $f(x)=\int_{a}^{x}h(t)\, dt$ for some function $h$ and $-\infty\leq a<x$.
%	
	The second author, Kesarwani and Kumar \cite[Theorem 2.4]{dkk} recently used this version of Guinand and a special case of Koshliakov's integral evaluation \cite[Equation (15)]{kosh1938} to obtain the following transformation for $\sum_{n=1}^{\infty}\sigma_{a}(n)e^{-ny}$ for \emph{any} $a\in\mathbb{C}$ such that Re$(a)>-1$ and $\textup{Re}(y)>0$:
	\begin{align}\label{maineqn}
		&\sum_{n=1}^\infty  \sigma_a(n)e^{-ny}+\frac{1}{2}\left(\left(\frac{2\pi}{y}\right)^{1+a}\mathrm{cosec}\left(\frac{\pi a}{2}\right)+1\right)\zeta(-a)-\frac{1}{y}\zeta(1-a)\nonumber\\
		&=\frac{2\pi}{y\sin\left(\frac{\pi a}{2}\right)}\sum_{n=1}^\infty \sigma_{a}(n)\Bigg(\frac{(2\pi n)^{-a}}{\Gamma(1-a)} {}_1F_2\left(1;\frac{1-a}{2},1-\frac{a}{2};\frac{4\pi^4n^2}{y^2} \right) -\left(\frac{2\pi}{y}\right)^{a}\cosh\left(\frac{4\pi^2n}{y}\right)\Bigg),
	\end{align}
where ${}_1F_2$-hypergeometric function is defined in \eqref{1fqq} below. Moreover, they also analytically continued the above result to obtain \cite[Theorem 2.5]{dkk} , for  $m\in\mathbb{N}\cup\{0\}$, $\mathrm{Re}(a)>-2m-3$, and $\textup{Re}(y)>0$,
	\begin{align}\label{extendedideqn}
		&\sum_{n=1}^\infty  \sigma_a(n)e^{-ny}+\frac{1}{2}\left(\left(\frac{2\pi}{y}\right)^{1+a}\mathrm{cosec}\left(\frac{\pi a}{2}\right)+1\right)\zeta(-a)-\frac{\zeta(1-a)}{y}\nonumber\\
		&=\frac{2\sqrt{2\pi}}{y^{1+\frac{a}{2}}}\sum_{n=1}^\infty\sigma_a(n)n^{-\frac{a}{2}}\left\{{}_{\frac{1}{2}}{K}_{\frac{a}{2}}\left(\frac{4\pi^2n}{y},0\right)-\frac{\pi2^{\frac{3}{2}+a}}{\sin\left(\frac{\pi a}{2}\right)}\left(\frac{4\pi^2n}{y}\right)^{-\frac{a}{2}-2}A_m\left(\frac{1}{2},\frac{a}{2},0;\frac{4\pi^2n}{y}\right)\right\}\nonumber\\
		&\qquad-\frac{y(2\pi)^{-a-3}}{\sin\left(\frac{\pi a}{2}\right)}\sum_{k=0}^m\frac{\zeta(a+2k+2)\zeta(2k+2)}{\Gamma(-a-1-2k)}\left(\frac{4\pi^2}{y}\right)^{-2k},
	\end{align}
	where
	\begin{align*}
		A_m(\mu,\nu,w;z):=\sum_{k=0}^m\frac{(-1)^{-\mu-w-\frac{1}{2}}\Gamma\left(\mu+w+\frac{1}{2}+k\right)}{k!\Gamma\left(-\nu-\mu-k\right)\Gamma\left(\frac{1}{2}-\nu-\mu-w-k\right)}\left(\frac{z}{2}\right)^{-2k},
	\end{align*}
	and ${}_{\mu}K_{\nu}(z, w)$ is a generalized Bessel function which they defined for $\nu\in\mathbb{C}\backslash\left(\mathbb{Z}\backslash\{0\}\right)$, and $z, \mu, w\in\mathbb{C}$ with $\mu+w\neq-\frac{1}{2}, -\frac{3}{2}, -\frac{5}{2},\cdots$ by
	\begin{align*}
		{}_{\mu}K_{\nu}(z, w)&:=\frac{\pi z^w 2^{\mu+\nu-1}}{\sin(\nu\pi)}\bigg\{\left(\frac{z}{2}\right)^{-\nu}\frac{\Gamma(\mu+w+\tfrac{1}{2})}{\Gamma(1-\nu)\Gamma(w+\tfrac{1}{2}-\nu)}\pFq12{\mu+w+\tfrac{1}{2}}{w+\tfrac{1}{2}-\nu,1-\nu}{\frac{z^2}{4}}\nonumber\\
		&\quad\quad\quad\quad\quad\quad-\left(\frac{z}{2}\right)^{\nu}\frac{\Gamma(\mu+\nu+w+\tfrac{1}{2})}{\Gamma(1+\nu)\Gamma(w+\tfrac{1}{2})}\pFq12{\mu+\nu+w+\tfrac{1}{2}}{w+\tfrac{1}{2},1+\nu}{\frac{z^2}{4}}\bigg\},
	\end{align*}
	and with ${}_{\mu}K_{0}(z, w)=\lim_{\nu\to0}{}_{\mu}K_{\nu}(z, w)$. (Observe that ${}_{-\nu}K_{\nu}(z, w)=z^{w}K_{\nu}(z)$.)
	
	It is interesting to note that in his second notebook \cite[p.~269]{ramnote} (see also \cite[p.~416]{V}), Ramanujan obtained a different identity for $\sum_{n=1}^\infty  \sigma_a(n)e^{-n\alpha}$ for any $a\in\mathbb{C}$ such that $\textup{Re}(a)>2$, and where $\a, \b>0$ with $\a\b=4\pi^2$, namely,
	\begin{align*}
		&\sqrt{\alpha^a}\bigg(\dfrac{\Gamma(a)\zeta(a)}{(2\pi)^a}+
		\cos\left(\frac12 \pi a\right)
		\sum_{n=1}^{\infty}\dfrac{n^{a-1}}{e^{\alpha n}-1}\bigg)\\
		&=\sqrt{\beta^a}\bigg(\cos\left(\frac12 \pi a\right)\dfrac{\Gamma(a)\zeta(a)}{(2\pi)^a}+
		\sum_{n=1}^{\infty}\dfrac{n^{a-1}}{e^{\beta n}-1}
		-\sin\left(\frac12 \pi a\right)\textup{PV}
		\int_0^{\infty}\dfrac{x^{a-1}}{e^{2\pi x}-1}
		\cot\left(\frac12 \beta x\right)dx\bigg),
	\end{align*}
	where $\textup{PV}$ denotes the Cauchy principal value.
	
	The result in \eqref{maineqn} gives the transformation formulas for Eisenstein series on $\textup{SL}_{2}(\mathbb{Z})$. Moreover, its generalization, that is \eqref{extendedideqn}, not only gives Ramanujan's formula \eqref{zetaodd} and the modular transformation for the logarithm of the Dedekind eta function as corollaries but also new transformations for $\sum_{n=1}^{\infty}\sigma_{2m}(n)e^{-ny}, m\in\mathbb{Z}\backslash\{0\}$ all of which are derived in \cite{dkk}. Before the work in \cite{dkk}, Wigert had obtained the corresponding transformation for $a=0$ which is also derived in \cite{dkk} from \eqref{extendedideqn}. Using the concept of transseries, Dorigoni and Kleinschmidt \cite[Equation (2.43)]{dorigoni} considered the special case of \eqref{extendedideqn} when $a$ is a negative even integer.
	
	The main goal of this paper is to generalize \eqref{maineqn} and \eqref{extendedideqn} in the setting of $\sum_{n=1}^\infty \sigma_{a}^{(N)}(n)e^{-ny}$. This is done in Theorems \ref{In terms of G} and \ref{Analytic continuation} respectively. Theorem \ref{Analytic continuation} involves a new generalization of the generalized Bessel function ${}_{\mu}K_{\nu}(z, w)$ whose asymptotic behavior as $z\to\infty$ is derived in Theorem \ref{Asymptotic expansion of muKnuN}. 
	
	Several important special cases of Theorems \ref{In terms of G} and \ref{Analytic continuation} are derived in the second part of the paper along with interesting applications of the some of them. These special cases include a new generalization of Ramanujan's formula for $\zeta(2m+1)$ (Theorem \ref{zetagen3}), an identity associated with extended higher Herglotz functions which was recently obtained in \cite{hhf1} (Corollary \ref{gencompanion}), an identity equivalent to the generalized Dedekind eta-transformation (Corollary \ref{zagiereqvt}) and a transformation of Wigert  (Corollary \eqref{neven}). The important special case $a=2Nm-1+N$ of Theorem \ref{In terms of G}, given in Theorem \ref{thm_resultt}, naturally leads us to the interesting new construct - \emph{generalized power partitions with ``$n^{2N-1}$ copies of $n^{N}$''}. The asymptotic expansion of the generating function of these partitions as $q\to1^{-}$ is derived in Corollary \ref{asym F_N(q)}. Wright's corresponding result for the plane partition generating function \cite[Lemma 1]{wright} is the special case $N=1$ of this result. 
	
	 The aforementioned special cases of Theorems \ref{In terms of G} and \ref{Analytic continuation} generalize almost all of the results in \cite{dkk} and reduce to the latter when $N=1$. Several of them employ Theorem \ref{meijergsim} which is a new result reducing the Meijer $G$-function involved in them to a sum of   ${}_1F_{q}$ and elementary functions.
	
	%\begin{proof}
	%We apply duplication and reflection formula respectively in the gamma and the cosine factors which evaluates the limit as 
	%\begin{align}
	%\lim_{a\to -(2\ell+1)} \Gamma(a)\cos\left(\frac{\pi a}{2}\right) &= \lim_{a\to -(2\ell+1)} \frac{\Gamma\(\frac{a}{2}\)\Gamma\(\frac{1}{2}+ \frac{a}{2}\)}{\sqrt{\pi}2^{1-a}} \times \frac{\pi}{\Gamma\(\frac{1}{2}+ \frac{a}{2}\)\Gamma\(\frac{1}{2}- \frac{a}{2}\)}\nonumber\\
	%&= \lim_{a\to -(2\ell+1)}\sqrt{\pi} 2^{a-1} \frac{\Gamma\(\frac{a}{2}\)}{\Gamma\(\frac{1}{2}- \frac{a}{2}\)} = \frac{\sqrt{\pi}}{2^{2\ell+2}}\frac{\Gamma\(-\ell-\frac{1}{2}\)}{\Gamma(\ell+1)}\nonumber\\
	%&= \frac{\sqrt{\pi}}{2^{2\ell+2} \ell!} \frac{\Gamma\(\frac{1}{2}\)}{(-1/2)(-3/2)\cdots (-\ell-1/2)}\nonumber\\
	%&= \frac{\sqrt{\pi}}{2^{2\ell+2} \ell!} \frac{(-1)^{\ell+1} 2^{\ell+1} \Gamma\(\frac{1}{2}\)}{1\cdot 3 \cdots (2\ell+1)}\nonumber\\
	%&=\frac{\sqrt{\pi}}{2^{2\ell+2} \ell!} \frac{(-1)^{\ell+1} 2^{\ell+1}(2\cdot 4 \cdots 2\ell) \Gamma\(\frac{1}{2}\)}{(2\ell+1)!}\nonumber\\
	%&=\frac{\sqrt{\pi}}{2^{2\ell+2} \ell!} \frac{(-1)^{\ell+1} 2^{2\ell+1}\ell! \Gamma\(\frac{1}{2}\)}{(2\ell+1)!} = \frac{(-1)^{\ell+1}\pi}{2(2\ell+1)!}
	%\end{align}
	%This completes the proof of the lemma.
	\section{Main results}\label{mr}
	
Before we state our results, we define certain special functions appearing in our work. We begin with the Meijer $G$-function \cite[p.~415, Definition 16.17]{NIST}.
	Let $m,n,p,q$ be integers such that $0\leq m \leq q$, $0\leq n \leq p$. Let $a_1, \cdots, a_p$ and $b_1, \cdots, b_q$ be complex numbers such that $a_i - b_j \not\in \mathbb{N}$ for $1 \leq i \leq n$ and $1 \leq j \leq m$.   Then the Meijer $G$-function is defined by 
	\begin{align}\label{MeijerG}
		G_{p,q}^{\,m,n} \!\left(  \,\begin{matrix} a_1,\cdots , a_p \\ b_1, \cdots b_m; b_{m+1}, \cdots, b_q \end{matrix} \; \Big| X   \right) := \frac{1}{2 \pi i} \int_L \frac{\prod_{j=1}^m \Gamma(b_j - w) \prod_{j=1}^n \Gamma(1 - a_j +w) X^w  } {\prod_{j=m+1}^q \Gamma(1 - b_j + w) \prod_{j=n+1}^p \Gamma(a_j - w)}\, dw,
	\end{align}
	where $L$  goes from $-i \infty$ to $+i \infty$ and separates the poles of $\Gamma(1-a_j+s)$  from the poles of $\Gamma(b_j-s)$.  
	The integral converges  absolutely if $p+q  < 2(m+n)$ and $|\arg(X)| < (m+n - \frac{p+q}{2}) \pi$.  In the case $p+q  = 2(m+n)$ and $\arg(X)=0$,  the integral converges absolutely if $ \left(  \re(w) + \frac{1}{2} \right) (q-p) > \re(\psi ) +1$,  where $\psi = \sum_{j=1}^q b_j - \sum_{j=1}^p a_j$.  
	
	A special case of the Meijer $G$-function which we will encounter in our work is the following:
	\begin{equation}\label{muknug}
		{}_{\mu}K_{\nu}^{(N)}(z, w) := 2^{\mu + \frac{2}{N}-1}\pi^{(1-N)\nu}z^{w+\nu-\frac{2}{N}} G_{1, \, \, 2N+1}^{N+1, \, \, 1}\bigg( \begin{matrix}
			1+\frac{1}{2N}-\mu-\nu-w \\
			\frac{1}{2}+\frac{1}{2N}-\nu, \langle \frac{i}{N} \rangle_{i=1}^N; 1+\frac{1}{2N}-w, \langle 1+\frac{3}{2N} - \frac{i}{N} \rangle_{i=2}^N 
		\end{matrix} \bigg | \frac{z^2}{4} \bigg).
	\end{equation}
	It can be easily checked that ${}_{\mu}K_{\nu}^{(1)}(z, w)$ reduces to the generalized modified Bessel function of the second kind recently studied in \cite[Equation (1.17)]{dkk}.
	
	We are now ready to state our first result.
	\begin{theorem}\label{In terms of G}
		Let $N\in\mathbb{N}$. Let $\sigma_a^{(N)}(n)$ be defined in \eqref{defaf}. For any complex number $a$ with $\re(a)>-1$ and $\re(y)>0$,
		\begin{align}\label{In terms of G_eqn}
			\sum_{n=1}^\infty \sigma_a^{(N)}(n) e^{-ny}  &= -\frac{\zeta(-a)}{2} + \frac{\zeta(N-a)}{y} + \frac{1}{N} \frac{\Gamma \left(\frac{1+a}{N}\right) \zeta \left(\frac{1+a}{N}\right)}{y^{\frac{1+a}{N}}}\nonumber\\
			&+ \frac{N^{3/2}}{\pi^{5/2}} \left( \frac{2\pi}{y} \right)^{\frac{a-1}{N}} \sum_{n=1}^\infty \frac{S_a^{(N)}(n)}{n^{2/N}} G_{1, \, \, 2N+1}^{N+1, \, \, 1} \left(\begin{matrix}
				\frac{1}{2} + \frac{1-a}{2N}\\
				\frac{1}{2} + \frac{1-a}{2N}, \langle \frac{i}{N}\rangle; \langle 1+\frac{3}{2N}-\frac{i}{N} \rangle
			\end{matrix} \Bigg | \frac{4 \pi^{2N+2} n^2}{y^2 N^{2N}} \right),
		\end{align}
		where, here and in the sequel\footnote{Whenever the number of terms in the sequence $\langle \cdot \rangle$ is $M$, $M\neq N$, we explicitly write $\langle \cdot \rangle|_{i=1}^{M}$. }, $\langle \cdot \rangle$ denotes the sequence of $N$ terms running over $i$ with $1 \leq i \leq N$. 
		
		Equivalently, using \eqref{muknug},
			\begin{align}\label{Using muKnuN}
			\sum_{n=1}^\infty \sigma_a^{(N)}(n) e^{-ny} +\frac{\zeta(-a)}{2} - \frac{\zeta(N-a)}{y} - \frac{1}{N} \frac{\Gamma \left(\frac{1+a}{N}\right) \zeta \left(\frac{1+a}{N}\right)}{y^{\frac{1+a}{N}}} = \frac{2 (2\pi)^{\frac{1}{N}-\frac{1}{2}} N^{\frac{a-1}{2}}}{y^{\frac{1}{N}+\frac{a}{2N}}} \sum_{n=1}^\infty \frac{S_a^{(N)}(n)}{n^{\frac{a}{2N}}}{}_{\frac{1}{2}} K_{\frac{a}{2N}}^{(N)}\left(\frac{4\pi^{N+1}n}{yN^N}, 0\right).
		\end{align}
	\end{theorem}
	When $N=1$, the above result gives Theorem 2.4 from \cite{dkk} as a special case. The above theorem can be analytically continued in a wider region of the $a$-complex plane. This requires the following asymptotic formula for $	{}_{\mu}K_{\nu}^{(N)}(z, w)$ for large values of $z$.
	\begin{theorem}\label{Asymptotic expansion of muKnuN}
Let ${}_{\mu}K_{\nu}^{(N)}(z, w)$ be defined in \eqref{muknug}, where $|\arg(z)|\leq\rho\pi-\delta$ with  $\rho>0,\ \delta\geq0$.	Let $m\in\mathbb{N}\cup\{0\}$. As $z\rightarrow \infty$,
		\begin{align}\label{estk}
			{}_\mu K_{\nu}^{(N)}(z, w) &= \frac{2^{3\mu+2\nu+2w+\frac{1}{N}-1} \pi^{(1-N)\nu-N} }{z^{2\mu+\nu+w+{1\over N}}}\sin (\pi(\mu+\nu))\prod_{i=2}^{N}\sin \left(\pi \left( \mu+\nu+w-{i-1\over N}\right) \right) \nonumber  \\
			&\times\sum_{k=0}^{m}{(-1)^{k(N+1)+N} \over k!}\prod_{i=1}^{2N-1}\Gamma\left( {i\over 2N }+\mu+\nu+w+k\right)  \Gamma \left({1\over 2}+\mu+w +k\right) \Gamma\left(1+\mu+\nu +k\right)\left({z \over 2} \right)^{-2k} \nonumber\\
			&+ \mathcal{O}\left({z^{-2m-2-{1 \over N}-2\mu-\nu-w}} \right).
		\end{align}
	\end{theorem}
	Using the above result, we obtain the analytic continuation of Theorem \ref{In terms of G} for any complex value of $a$. This is stated in the following theorem of which Theorem 2.5 of \cite{dkk} is the special case $N=1$.
	\begin{theorem}\label{Analytic continuation}
		Let $m \in \N \cup \lbrace 0 \rbrace$ and $\re(y)>0$.  For $\re(a)>-(2m+2)N-1$, $N\in\mathbb{N}$, the following identity holds :
		\begin{align}\label{Eqn:Analytic continuation}
			&\sum_{n=1}^\infty \sigma_a^{(N)}(n) e^{-ny}  + \frac{\zeta(-a)}{2} - \frac{\zeta(N-a)}{y} - \frac{1}{N} \frac{\Gamma \left(\frac{1+a}{N}\right) \zeta \left(\frac{1+a}{N}\right)}{y^{\frac{1+a}{N}}} \nonumber\\
			&= \frac{2 (2\pi)^{\frac{1}{N}-\frac{1}{2}} N^{\frac{a-1}{2}}}{y^{\frac{1}{N}+\frac{a}{2N}}} \sum_{n=1}^\infty \frac{S_a^{(N)}(n)}{n^{\frac{a}{2N}}} \bigg [{}_{\frac{1}{2}} K_{\frac{a}{2N}}^{(N)}\left(\tfrac{4\pi^{N+1}n}{yN^N}, 0\right) - \frac{2^{\frac{1}{2} + \frac{a+1}{N}} \pi^{\frac{(1-N)a}{2N} -N}} {\left(\frac{4\pi^{N+1}n}{yN^N}\right)^{1+\frac{1}{N}+\frac{a}{2N}}}   \tfrac{\sin\(\frac{\pi}{2}(N-a)\)}{2^{N-1}}C_{m,N}\left(\tfrac{1}{2}, \tfrac{a}{2N}, 0, \tfrac{4\pi^{N+1}n}{yN^N}\right)\bigg] \nonumber\\
			&\quad+ \frac{y}{2\pi^{2}} \sum_{k=0}^m \left(-\frac{y^2}{4\pi^2}\right)^k   \zeta(-2kN-N-a)\zeta(2k+2),
		\end{align}
		where
		\begin{align}\label{cmn}
			C_{m, N}(\mu, \nu, w; z) &:= \sum_{k=0}^m \frac{(-1)^{k(N+1)+N}}{k!}  \Gamma \left(\tfrac{1}{2}+\mu+w +k\right) \Gamma\left(1+\mu+\nu +k\right)\prod_{i=1}^{2N-1}\Gamma\left( \tfrac{i}{2N} +\mu+\nu+w+k\right)\left(\frac{z}{2} \right)^{-2k}.
		\end{align}
	\end{theorem}
	\begin{remark}
		Employing \eqref{gmf} below, we see that 
		\begin{align}\label{cmnsimplified}
			C_{m, N}(\mu, \nu, w; z)&=\frac{(2\pi)^{N-\frac{1}{2}}}{(2N)^{\frac{1}{2}+2N\left(\mu+\nu+w\right)}}\sum_{k=0}^m \frac{(-1)^{k(N+1)+N}}{k!}\frac{\Gamma\left(\frac{1}{2}+\mu+w+k \right)\Gamma\left(1+\mu+\nu +k\right)}{\Gamma\left(1+\mu+\nu+w +k\right)}\nonumber\\
			&\qquad\qquad\qquad\qquad\qquad\quad\times\frac{\G(1+2N(\mu+\nu+w+k))}{(2N)^{2Nk}}\left(\frac{z}{2} \right)^{-2k}.
		\end{align}
	\end{remark}

Our next result reduces the Meijer $G$-function occurring in Theorem \ref{In terms of G} to a sum of elementary functions and the generalized hypergeometric function ${}_1F_{2N}$. This will be instrumental in the proofs of the special cases of Theorem \ref{Analytic continuation}. 
To that end, first define $b$ by
	\begin{equation}\label{defb}
		b=b(N):=\begin{cases}
			0,\hspace{1mm}\text{if}\hspace{1mm}N\hspace{1mm}\text{is odd},\\
			1,\hspace{1mm}\text{if}\hspace{1mm}N\hspace{1mm}\text{is even}.
		\end{cases}
	\end{equation}
	For $k\in\N$, define
	\begin{equation}\label{defE}
		\mathscr{E}_{a, N, z}(k, b):=(-1)^k \exp \left(2N e^{\frac{\pi i (2k+b)}{2N}} z^{1\over 2N} + \frac{\pi i (2k+b) (a+1)}{2N} \right) .
	\end{equation}
	
	\begin{theorem}\label{meijergsim}
		Let $a, z\in\C$, $N\in\mathbb{N}$. Then
		\begin{align}\label{meijergsimeqn}
			&G_{1, \, \, 2N+1}^{N+1, \, \, 1} \left(\begin{matrix}
				\frac{1}{2} + \frac{1-a}{2N}\\
				\frac{1}{2} + \frac{1-a}{2N}, \langle \frac{i}{N}\rangle; \langle 1+\frac{3}{2N}-\frac{i}{N} \rangle
			\end{matrix} \Bigg | \, z \right) \nonumber\\
			&= \frac{ N^{N-a - \frac{1}{2}} \Gamma\left(\frac{1-N+a}{2}\right)z^{\frac{1}{2}+\frac{1-a}{2N}} }{\Gamma\left(\frac{N-a}{2}\right)} 
			{}_1F_{2N}\left( \left. \begin{matrix}
				1 \\ \left\langle \frac{1}{2} - \frac{a+1}{2N} + \frac{i}{2N} \right\rangle_{i=1}^{2N}
			\end{matrix}\right| (-1)^{N+1} z \right) + A_{a, N}(z),
		\end{align}
		where $A_{a, N}(z)$ is given by
		\begin{align}\label{aanzodd}
			A_{a, N}(z) :=  
			\frac{(-1)^{\frac{N+1}{2}}z^{1\over N}}{2 \sin \left( \frac{\pi a}{2} \right)}\sqrt{\frac{\pi}{N}} \left [  \sum\limits_{k=0}^{\frac{N-1}{2}} \mathscr{E}_{a, N, z}(k, 0)+  \sum\limits_{k= \frac{N+1}{2}}^{\frac{3N-1}{2}} e^{-\pi i a}\mathscr{E}_{a, N, z}(k, 0)+  \sum\limits_{k= \frac{3N+1}{2}}^{2N-1} e^{-2\pi i a}\mathscr{E}_{a, N, z}(k, 0) \right ], 
		\end{align}
		for $N$ odd, and by
		\begin{align}\label{aanzeven}
			A_{a, N}(z) :=
			\frac{(-1)^{\frac{N+1}{2}}z^{\frac{1}{N}} }{2 \cos \left( \frac{\pi a}{2} \right)}\sqrt{\frac{\pi}{N}} \left [  \sum\limits_{k=0}^{\frac{N}{2}-1} \mathscr{E}_{a, N, z}(k, 1)-  \sum\limits_{k= \frac{N}{2}}^{\frac{3N}{2}-1} e^{-\pi i a}\mathscr{E}_{a, N, z}(k, 1)+  \sum\limits_{k= \frac{3N}{2}}^{2N-1} e^{-2\pi i a}\mathscr{E}_{a, N, z}(k, 1) \right ],
		\end{align}
		for $N$ even.
	\end{theorem}
We now provide some important special cases of Theorems \ref{In terms of G} and \ref{Analytic continuation}. We begin with the special case $a=-2Nm-N$ of Theorem \ref{Analytic continuation} which gives a new generalization of Ramanujan's formula for $\zeta(2m+1)$, that is, \eqref{zetaodd}. 
 \begin{theorem}\label{zetagen3}
 	Let $N$ be an odd positive integer and $\a,\b>0$ such that $\a\b^{N}=\pi^{N+1}$. Then for any non-zero integer $m$,
 	{\allowdisplaybreaks\begin{align}\label{zetageneqn3}
 			&\a^{-\left( \frac{2Nm+N-1}{N+1}\right) }\left(\frac{1}{2}\zeta(2Nm+N)+\sum_{n=1}^{\infty}\frac{n^{-2Nm-N}}{\textup{exp}\left((2n)^{N}\a\right)-1}\right)\nonumber\\
 			&=(-1)^m \b^{-\left( \frac{2Nm+N-1}{N+1}\right) }\frac{2^{(N-1)(2m+1-{1\over  N})}}{N}\Bigg(\frac{1}{2}\frac{\zeta\left(2m+2-{1\over N}\right)}{\sin\left(\pi \over 2N \right) }+\sum_{j=-\frac{N-1}{2}}^{\frac{N-1}{2}}e^{\pi ij\over N}\sum_{n=1}^{\infty}\frac{n^{-2m-2+{1\over N}}}{\textup{exp}\left((2n)^{\frac{1}{N}}\b e^{\frac{i\pi j}{N}}\right)-1}\Bigg)\nonumber\\
 			&\quad+(-1)^{m}2^{2Nm+N-1}\sum_{j=0}^{m+1}\frac{(-1)^jB_{2j}B_{2N(m-j+1)}}{(2j)!(2N(m-j+1))!}\a^{\frac{2j}{N+1}}\b^{\frac{2N^2(m-j+1)}{N+1}}.
 	\end{align}}
 \end{theorem}
The identity \eqref{Zagier} is equivalent to the special case $m=-1$ of the above theorem for $N$ odd. For $N$ even, it can be derived by letting $m=0$ in \eqref{case-2Nm+N} below. See  \S \ref{zagierjrh}.
\begin{theorem}\label{zagiereqvt}
	Let $N\in\mathbb{N}$ and $b$ be defined in \eqref{defb}. Then
	\begin{align}\label{zagier to us}
		\sum_{n=1}^\infty \sigma_{N}^{(N)}(n) e^{-ny}  + \frac{1}{2}\zeta(-N) &=-\frac{1}{2y}-\frac{1}{2N\sin{\left(\pi\over 2N\right)}}\left(\frac{2\pi}{y}\right)^{1+{1\over N}}\zeta\left(-{1\over N}\right) \nonumber\\
		&\quad-{1\over N}\left(\frac{2\pi}{y}\right)^{1+{1\over N}}\sum_{j=-{(N-1+b)\over 2}}^{{N-1-b\over 2}}\sum_{n=1}^\infty    \frac{{n^{{1\over N}}}{e^{i\pi(2j+b)  \over 2N}}}{\exp  \left(2 \pi  e^{i\pi (2j+b) \over 2N}\left(\frac{2\pi n}{y} \right)^{1\over N}\right) -1}.
	\end{align}
\end{theorem}
A plane partition of  an integer $n$ is a two-dimensional array of non-negative integers in the fourth quadrant subject to a non-increasing condition along rows and columns. Let $\mathbb{P}(n)$ denote the number of plane partitions of $n$. Then for $|q|<1$, the generating function of  $\mathbb{P}(n)$ is given by \cite[p.~184, Equation (11.2.15)]{gea}
\begin{equation}\label{pp}
\sum_{n=1}^{\infty}\mathbb{P}(n)q^n=\prod_{n=1}^{\infty}\frac{1}{(1-q^n)^n}.
\end{equation}
As is evident, the above infinite product also generates $\mathscr{P}(n)$, the number of partitions of an integer in which a part $n$ appears with $n$ copies (or with $n$ distinct colors, say). In fact, Chaundy \cite{chaundy} obtained a bijection between the set of plane partitions of an integer and the set of partitions of that integer with ``$n$ copies of $n$''. 

In our work, we encounter the logarithmic derivative of the generalized infinite product
%$\prod_{n=1}^{\infty}\left(1-q^{n^N}\right)^{-n^{2N-1}}$
$\prod\limits_{n=1}^{\infty}\frac{1}{\left(1-q^{n^N}\right)^{n^{2N-1}}}$, where $N\in\mathbb{N}$. Clearly, this product generates the number of power partitions with ``$n^{2N-1}$ copies of $n^{N}$''. Let 
\begin{equation}\label{ppN}
	F_N(q):=\sum_{n=1}^{\infty}\mathscr{P}_{N}(n)q^n=\prod_{n=1}^{\infty}\frac{1}{\left(1-q^{n^N}\right)^{n^{2N-1}}}.
\end{equation}
Observe that for $N>1$, there is a lot of repetition among the coefficients $\mathscr{P}_{N}(n)$ of $F_N(q)$. For example, when $N=3$,  $\mathscr{P}_{N}(n)=1$ for $1\leq n\leq7$; $\mathscr{P}_{N}(n)=33$ for $8\leq n\leq15$; $\mathscr{P}_{N}(n)=561$ for $16\leq n\leq23$, and $\mathscr{P}_{N}(n)=6545$ for $24\leq n\leq26$ and so on. Indeed, this happens because the only parts that can appear in these partitions are $1^3, 2^3, 3^3,\cdots$. However, these parts $1^3, 2^3, 3^3,\cdots$ appear with $1^5, 2^5, 3^5,\cdots$ copies respectively. We evaluate $\mathscr{P}_{3}(24)$, in particular. Note that the only partitions of $24$ that can be enumerated by $\mathscr{P}_3(24)$ are
$\underbrace{1+1+\cdots+1}_{24\hspace{1mm}\text{times}}$,    $8+\underbrace{1+1+\cdots+1}_{16\hspace{1mm}\text{times}}$, $8+8+\underbrace{1+1+\cdots+1}_{8\hspace{1mm}\text{times}}$ and $8+8+8$. But we have $32$ different copies of $8$. Since the order of the copies does not matter, we have
\begin{equation*}
	\mathscr{P}_3(24)=1+\binom{32}{1}+\left\{\binom{32}{2}+\binom{32}{1}\right\}+\left\{\binom{32}{3}+2\binom{32}{2}+\binom{32}{1}\right\}=6545.
\end{equation*}
It is always worthwhile to find the asymptotic behavior of generating functions of restricted partition functions as $q\to1^{-}$ or as $q$ tends to other roots of unity from within the unit disk. For example, Wright \cite[Lemma 1]{wright} found the asymptotic estimate of the generating function of $\mathbb{P}(n)$, occurring in \eqref{pp}, as $q\to1^-$. He derived this result through a long intricate calculation \cite[pp.~180-184]{wright}. A short direct proof of Wright's result was recently obtained in \cite{dk03}. 

Very recently, Bridges, Brindle, Bringmann and Franke \cite{bbbf} studied the asymptotic behavior of partition functions generated by the infinite product $\prod\limits_{n=1}^{\infty}(1-q^n)^{-f(n)}$, where $f(n)$ is an arithmetic function whose Dirichlet series satisfies certain conditions. See Theorems 1.4 and 4.4 of \cite{bbbf}.	 While the plane partition generating function falls under the purview of their setting, the infinite product we consider in \eqref{ppN} does not.

In what follows, we obtain the asymptotic estimate for $F_N(q)$, thereby generalizing Wright's result for odd $N$. To that end, apply the operator $q\frac{d}{dq}$ on the infinite product in \eqref{ppN} to get
\begin{align}\label{gpp3N-1}
q\frac{d}{dq}\log\left(\prod_{n=1}^{\infty}\frac{1}{\left(1-q^{n^N}\right)^{n^{2N-1}}}\right)=-q\frac{d}{dq}\sum_{n=1}^{\infty}n^{2N-1}\log\left(1-q^{n^{N}}\right)=\sum_{n=1}^{\infty}\frac{n^{3N-1}q^{n^N}}{1-q^{n^N}}=	\sum_{n=1}^\infty \sigma_{3N-1}^{(N)}(n) q^n. 
\end{align}
The series on the extreme right-hand side is the special case $m=1$ of the series occurring on the left-hand side of the following transformation, which is a special case of Theorem \ref{In terms of G}. Before we state it though, we need to define the hyperbolic sine and cosine integrals  $\mathrm{Shi}(z)$ and $\mathrm{Chi}(z)$ defined by \cite[p.~150, Equation (6.2.15), (6.2.16)]{NIST}
\begin{align*}
	\mathrm{Shi}(z):=\int_0^z\frac{\sinh(t)}{t}\ dt,\hspace{3mm}
	\mathrm{Chi}(z):=\gamma+\log(z)+\int_0^z\frac{\cosh(t)-1}{t}\ dt.
\end{align*}
For brevity, we use the notation
\begin{align}\label{brevity}
	\left(\sinh\operatorname{Shi}-\cosh\operatorname{Chi}\right)(z):=\sinh(z) \operatorname{Shi} (z)- \cosh(z) \operatorname{Chi} (z)
\end{align}
throughout the sequel.
 	\begin{theorem}\label{thm_resultt}
 	Let $N$ be an odd positive number. For a positive integer $m$,
 	\begin{align}\label{resultt}
 		&\sum_{n=1}^\infty \sigma_{2Nm-1+N}^{(N)}(n) e^{-ny}+ \frac{B_{2Nm}}{2Nmy}	-\frac{(2m)!\z(2m+1)}{Ny^{2m+1}}\nonumber \\
 		& =\frac{2(-1)^m}{\pi}\left({2\pi \over y}\right)^{2m+1}\sum_{n=1}^{\infty} S_{2Nm-1+N}^{(N)}(n)\Bigg[ 
 		\frac{1}{N}\sum_{k=-{\left(N-1 \right)\over 2 }}^{N-1\over 2 }\left( \sinh\operatorname{Shi}-\cosh\operatorname{Chi}\right) \left(  \left({(2\pi)^{N+1}n\over y}\right)^{1\over N}e^{\pi i k \over N}\right)\nonumber\\
 		&\quad\hspace{6.5cm}+\sum_{j=1}^{m} (2Nj-1)!\left({(2\pi)^{N+1}n\over y}\right)^{-2j}\Bigg].
 	\end{align}
 \end{theorem}
Using Theorem \ref{thm_resultt}, we now obtain an intermediate result towards our goal of finding the asymptotic estimate for $F_N(q)$ as $q\to1^{-}$.
\begin{theorem}\label{Asym of Sum Sigma}
	Let $m\in \N$ and let $N$ be a positive odd integer. As $y \to 0$ in $| \mathrm{arg}\, y |<\frac{\pi}{2}$,
	\begin{align*}
		\sum_{n=1}^\infty \sigma_{2Nm-1+N}^{(N)}(n) e^{-ny} &= \frac{(2m)! \zeta(2m+1)}{Ny^{2m+1}} - \frac{B_{2Nm}}{2Nmy} - (-1)^m \frac{4}{(2\pi)^{2Nm}}\nonumber\\
		& \quad\times \sum_{j=1}^{r+1} \frac{\Gamma (2Nm+2Nj)\zeta (2Nm+2Nj) \zeta (2j)}{(2\pi)^{2j(N+1)}}y^{2j-1} + O(y^{2r+3}).
	\end{align*}
\end{theorem}	
Let $q=e^{-y}$, where $\re(y)>0$. This implies 
\begin{corollary}\label{asym F_N(q)}
	Let $N$ be an odd positive integer. As $q \to 1^{-}$,
	\begin{align*}
		F_N(q) = e^c (\log q)^{\frac{B_{2N}}{2N}}\exp\left(\frac{\zeta(3)}{N\log^2 q}\right) \exp\left( - \sum_{j=1}^{r+1} \delta_j (\log q)^{2j}\right)\left(1+ O_r\left((\log q)^{2r+4} \right)\right),
	\end{align*}
	where $c$ is a constant, $B_{2N}$ are Bernoulli numbers, and 
	\begin{equation*}
		\delta_j := \frac{2 \, \Gamma(2N+2Nj) \, \zeta(2N+2Nj) \, \zeta(2j)}{(2\pi)^{2N} j \left( (2\pi)^{N+1} \right)^{2j}}.
	\end{equation*}
\end{corollary}
Letting $N=1$ in the above result, we recover Wright's asymptotic estimate although without the explicit determination of $c$. Wright showed that $c=\displaystyle2\int_{0}^{\infty}\frac{y\log(y)}{e^{2\pi y}-1}\, dy$. 

We have also obtained the special case $a=-2Nm$ of Theorem \ref{Analytic continuation} when $N$ is an even positive integer. This generalizes a formula of Wigert \cite[p.~8-9, Equation (1.5)]{wig}, and was first obtained in \cite[Theorem 1.5]{dixitmaji1}. See Corollary \ref{neven}. Another new result for even $N$ which can be obtained from Theorem \ref{Analytic continuation} is when $a=-2Nm+N$, and is given in \eqref{case-2Nm+N}.  As shown in Section \ref{8.2}, it can also be derived by differentiating a result from \cite[Theorem 2.12]{DGKM}. 

Thus one of the goals of this paper is to not only give new corollaries resulting from Theorems \ref{In terms of G} and \ref{Analytic continuation} but also to provide uniform proofs of several disparate results in the literature.

	\section{Preliminaries}\label{prelim}
	The duplication and reflection formulas for the Gamma function are given by \cite[Chapter 3]{temme}
	\begin{align}
		\G(s)\G\left(s+\frac{1}{2}\right)&=\frac{\sqrt{\pi}}{2^{2s-1}}\G(2s),\label{dupl}\\
		\G(s)\G(1-s)&=\frac{\pi}{\sin(\pi s)}\hspace{5mm}(s\notin\mathbb{Z}).\label{refl}    
	\end{align}
	The Gauss multiplication formula \cite[p.~52]{temme} is given for $m\in\mathbb{N}, m>1$, by
	\begin{equation}\label{gmf}
		\prod_{k=1}^{m}\Gamma\left(z+\frac{k-1}{m}\right)=(2\pi)^{\frac{1}{2}(m-1)}m^{\frac{1}{2}-mz}\Gamma(mz).
	\end{equation}
	The following limit, which can be established using the \eqref{dupl} and \eqref{refl}, will be used frequently in our analysis.
	\begin{equation}\label{usefullim}
		\lim_{a\to -(2\ell+1)} \Gamma(a)\cos\left(\frac{\pi a}{2}\right)= \frac{(-1)^{\ell+1}\pi}{2(2\ell+1)!}\hspace{12mm}(\ell\geq0).
	\end{equation}
The asymmetric form of the functional equation of $\zeta(s)$ reads \cite[p.~73, Equation (4)]{dav}
\begin{equation}\label{zetafe}
	\zeta(1-s)=2^{1-s}\pi^{-s}\Gamma(s)\zeta(s)\cos\left(\frac{\pi s}{2}\right).
\end{equation}
	We need \cite[Formula 7.18.1.1, p.~615]{Prudnikov}
	\begin{align}\label{prudnikovid}
		{}_1F_q&\left(1;{m+1\over q},{m+2\over q},\cdots {m+q\over q};z\right)={m!\over q^{m+1}z^{m/q}} \left[\sum_{k=0}^{q-1}\frac{e^{q\,\theta_k z^{1/q}}}{\theta_k^m}-q^{m+1}\sum_{k=1}^{[m/q]}\frac{z^{m/q-k}}{q^{qk}(m-qk)!} \right],
	\end{align}
	where $\theta_k=e^{2k\pi i/q}$. Here, for $q\in\mathbb{N}$, the ${}_1F_{q}$-hypergeometric function is defined by
	\begin{equation}\label{1fqq} 
	{}_1F_q \left( a; b_1, b_2, \cdots, b_q \big\vert z \right) := \sum_{n=0}^\infty \frac{
		(a)_n}{(b_1)_n (b_2)_n\cdots(b_q)_n }\frac{z^n}{n!}
	\end{equation}
	with $(a)_n:=\Gamma(a+n)/\Gamma(a)$ being the shifted factorial. 
	
	\section{New results on Meijer $G$-function and ${}_{\mu}K_{\nu}^{(N)}(z, w)$}
	The following lemma \cite[Lemma 4.3]{DGKM} will be used in the proof of Theorem \ref{meijergsim}.
	\begin{lemma}\label{cheby}
		Let $z\in\mathbb{C}$ and $N\in\mathbb{N}$. Then
		\begin{align*}
			\frac{\sin( N z)}{\sin (z)} = \sideset{}{''}\sum_{j=-(N-1)}^{N-1} \exp( i j z),
		\end{align*}
	where, here and in the sequel, the notation $\sideset{}{''}\sum\limits_{j=-(N-1)}^{N-1}$ indicates the summation over the values $j = -(N-1), -(N-3), \ldots, N-3, N-1$.Thus,
		\begin{align}
			\frac{ \cos(N z)}{\cos( z )}&=(-1)^{\frac{N-1}{2}}\sideset{}{''}\sum_{j=-(N-1)}^{N-1} i^j\exp( -i j z)\hspace{5mm}\textup{(for $N$ odd)},\label{cc}\\
			\frac{ \sin(N z)}{\cos( z )}&=(-1)^{\frac{N}{2}}\sideset{}{''}\sum_{j=-(N-1)}^{N-1} i^j\exp( i j z)\hspace{5mm}\textup{(for $N$ even)}\nonumber.
		\end{align}
	\end{lemma}
	
We are now ready to prove the reduction formula for the Meijer $G$-function occurring in Theorem \ref{In terms of G}. 

\begin{proof}[Theorem \textup{\ref{meijergsim}}][]
		Using Slater's theorem \cite[Equation 7 p. 145]{Luke},  we first rewrite the Meijer G-function in the form
		\begin{align}\label{Slater on G}
			& G_{1, \, \, 2N+1}^{N+1, \, \, 1} \left(\begin{matrix}
				\frac{1}{2} + \frac{1-a}{2N}\\
				\frac{1}{2} + \frac{1-a}{2N}, \langle \frac{i}{N}\rangle; \langle 1+\frac{3}{2N}-\frac{i}{N} \rangle
			\end{matrix} \Bigg | \, z \right) \nonumber \\
			&=\frac{\prod\limits_{i=1 }^{N}\Gamma\left( -{1 \over 2}+{{2i+a-1} \over {2N}}\right)z^{\frac{1}{2}+\frac{1-a}{2N}} }{\prod\limits_{i=1}^{N}\Gamma \left({1 \over 2}+{{2i-a-2} \over {2N}} \right)}    
			{}_1F_{2N}\left( \left. \begin{matrix}
				1 \\ \left\langle{3 \over 2 }+{{1-a-2i} \over {2N}} \right\rangle_{i=1}^{N}, \left\langle{1 \over 2}+{{2i-a-2} \over {2N}}\right\rangle_{i=1}^{ N}
			\end{matrix}\right| (-1)^{N+1} z \right)\nonumber\\
			&+\sum_{j=1}^{N}\frac{\prod\limits_{i=1 \atop{i\neq j}}^{N}\Gamma\left({i-j \over N}\right) \Gamma\left({1 \over 2}+{{1-a} \over {2N}}-{j \over N} \right)\Gamma\left({1 \over 2}+{j \over N}- {{1-a} \over {2N}} \right) z^{j/N} }{ \prod\limits_{i=1}^{N}\Gamma \left({j+i \over N}-{3 \over 2N} \right) }  
			{}_1F_{2N}\left( \left. \begin{matrix}
				1  \\ \left\langle{1 }-{i-j \over N}\right\rangle_{i=1}^{N},\left\langle{i+j \over N}-{3 \over 2N}\right\rangle_{i=1}^{ N}
			\end{matrix}\right| (-1)^{N+1} z \right)
		\end{align}
	Observe that the ${}_1F_{2N}$ occurring in the finite sum above is actually a ${}0F_{2N-1}$. We, however, write it in that form so as to be able to later use \eqref{prudnikovid}.
	
		The Gauss multiplication formula \eqref{gmf} simplifies the Gamma factors in the first expression of \eqref{Slater on G} as
		\begin{align}\label{Gamma factors in T1}
			&\prod\limits_{i=1 }^{N}\Gamma\left( -{1 \over 2}+{{2i+a-1} \over {2N}}\right) = (2\pi)^{\frac{N-1}{2}}N^{\frac{N-a}{2}}\Gamma\left(\frac{1-N+a}{2} \right),\nonumber\\
			&\prod\limits_{i=1}^{N}\Gamma \left({1 \over 2}+{{2i-a-2} \over {2N}} \right) = (2\pi)^{\frac{N-1}{2}}N^{\frac{1-N+a}{2}}\Gamma\left(\frac{N-a}{2} \right)
		\end{align}
		Similarly, the Gamma factors in the second expression of \eqref{Slater on G} can be simplified as
		\begin{align}\label{Gamma factors in T2}
			&\prod\limits_{i=1 \atop{i\neq j}}^{N}\Gamma\left({i-j \over N}\right) 
			= \prod\limits_{i=1 \atop{i\neq j}}^{N}\frac{\Gamma\left({i-j \over N} +1 \right) }{ {i-j \over N}}
			=\frac{(2\pi)^{\frac{N-1}{2}}N^{-\frac{1}{2}-N+j}\Gamma(N+1-j)}{\prod\limits_{i=1}^{j-1} {i-j \over N} \prod\limits_{i=j+1}^{N}{i-j \over N}}
			= \frac{(-1)^{j-1}(2\pi)^{\frac{N-1}{2}}N^{j-\frac{3}{2}}}{\Gamma(j)}, 
			\nonumber\\
			&\Gamma\left({1 \over 2}+{{1-a} \over {2N}}-{j \over N} \right)\Gamma\left({1 \over 2}+{j \over N}- {{1-a} \over {2N}} \right) = \frac{\pi}{\cos \left(\pi \left(\frac{j}{N}+ \frac{a-1}{2N} \right)\right)},
			\nonumber\\
			&\prod\limits_{i=1}^{N}\Gamma \left({j+i \over N}-{3 \over 2N} \right) = (2\pi)^{\frac{N-1}{2}} N^{1-j} \Gamma \left(j-\frac{1}{2}\right),
		\end{align}
		where we used the functional equation $\Gamma(z+1)=z\Gamma(z)$ of the Gamma function and \eqref{gmf} for the first one, the reflection formula \eqref{refl} for the second one and again \eqref{gmf} for the final one. Inserting the Gamma factors from \eqref{Gamma factors in T1} and \eqref{Gamma factors in T2} into \eqref{Slater on G}, we obtain
		\begin{align*}
			& G_{1, \, \, 2N+1}^{N+1, \, \, 1} \left(\begin{matrix}
				\frac{1}{2} + \frac{1-a}{2N}\\
				\frac{1}{2} + \frac{1-a}{2N}, \langle \frac{i}{N}\rangle; \langle 1+\frac{3}{2N}-\frac{i}{N} \rangle
			\end{matrix} \Bigg | \, z \right) 
			\nonumber \\
			&= N^{N-a - \frac{1}{2}} \frac{\Gamma\left(\frac{1-N+a}{2}\right)z^{\frac{1}{2}+\frac{1-a}{2N}} }{\Gamma\left(\frac{N-a}{2}\right)} 
			{}_1F_{2N}\left( \left. \begin{matrix}
				1 \\ \left\langle{3 \over 2 }+{{1-a-2i} \over {2N}} \right\rangle_{i=1}^{N}, \left\langle{1 \over 2}+{{2i-a-2} \over {2N}}\right\rangle_{i=1}^{ N}
			\end{matrix}\right| (-1)^{N+1} z \right)
			\nonumber \\
			&+\sum_{j=1}^N \frac{(-1)^{j-1} \pi N^{2j-\frac{5}{2}} z^{j/N}}{\Gamma(j)\Gamma(j-\frac{1}{2})\cos \pi \left(\frac{j}{N}+ \frac{a-1}{2N} \right)} 
			{}_1F_{2N}\left( \left. \begin{matrix}
				1  \\ \left\langle{1 }-{i-j \over N}\right\rangle_{i=1}^{N},\left\langle{i+j \over N}-{3 \over 2N}\right\rangle_{i=1}^{ N}
			\end{matrix}\right| (-1)^{N+1} z \right).
		\end{align*}
		Merging the the sequences inside the hypergeometric functions and applying the duplication formula \eqref{dupl} in the last expression, the above equation reduces to
		\begin{align}\label{Reduced Meijer G}
			& G_{1, \, \, 2N+1}^{N+1, \, \, 1} \left(\begin{matrix}
				\frac{1}{2} + \frac{1-a}{2N}\\
				\frac{1}{2} + \frac{1-a}{2N}, \langle \frac{i}{N}\rangle; \langle 1+\frac{3}{2N}-\frac{i}{N} \rangle
			\end{matrix} \Bigg | \, z \right) 
			\nonumber \\
			&= \frac{N^{N-a - \frac{1}{2}} \Gamma\left(\frac{1-N+a}{2}\right)z^{\frac{1}{2}+\frac{1-a}{2N}} }{\Gamma\left(\frac{N-a}{2}\right)} \,
			{}_1F_{2N}\left( \left. \begin{matrix}
				1 \\ \left\langle \frac{1}{2} - \frac{a+1}{2N} + \frac{i}{2N} \right\rangle_{i=1}^{2N}
			\end{matrix}\right| (-1)^{N+1} z \right)
			\nonumber \\
			&+\sum_{j=1}^N \frac{(-1)^{j-1} (2N)^{2j-2} (\pi/N)^{1/2} z^{j/N} }{\Gamma(2j-1)\cos \pi \left(\frac{j}{N}+ \frac{a-1}{2N} \right)} \,
			{}_1F_{2N}\left( \left. \begin{matrix}
				1  \\ \left\langle \frac{j-1}{N}+\frac{i}{2N} \right\rangle_{i=1}^{2N}
			\end{matrix}\right| (-1)^{N+1} z \right).
		\end{align}
		Let $T_1$ denote the finite sum over $j$ in the above equation. We evaluate it by considering two cases depending on the parity of $N$. 
		
		\noindent
		\textbf{Case 1 :} We first consider the case when $N$ is odd. Invoking \eqref{prudnikovid} for the hypergeometric function involved in the term $T_1$, we obtain
		\begin{align*}
			{}_1F_{2N}\left( \left. \begin{matrix}
				1  \\ \left\langle \frac{j-1}{N}+\frac{i}{2N} \right\rangle_{i=1}^{2N}
			\end{matrix}\right|  z \right) = \frac{(2j-2)!}{(2N)^{2j-1}z^{\frac{j-1}{N}}} \sum_{k=0}^{2N-1} \frac{\exp\left(2N e^{\frac{\pi i k}{N}} z^{\frac{1}{2N}} \right)}{\exp\left( \frac{2\pi i k (j-1)}{N} \right)}
		\end{align*}
		so that
		\begin{align}\label{T2 sum}
			T_1&=  \frac{\sqrt{\pi}z^{1/N}}{2N^{3/2}} \sum_{j=1}^N \frac{(-1)^{j-1}}{\cos \pi \left(\frac{j}{N}+ \frac{a-1}{2N} \right)}
			\sum_{k=0}^{2N-1} \frac{\exp\left(2N e^{\frac{\pi i k}{N}} z^{\frac{1}{2N}} \right)}{\exp\left( \frac{2\pi i k (j-1)}{N} \right)} 
			\nonumber\\
			&= \frac{\sqrt{\pi}z^{1/N}}{2N^{3/2}} \sum_{k=0}^{2N-1} \exp\left(2N e^{\frac{\pi i k}{N}} z^{\frac{1}{2N}} \right) 
			\sum_{j=1}^N \frac{(-1)^{j} e^{\frac{2\pi i k j}{N}}}{\cos \left(\pi \left( \frac{a+1}{2N} - \frac{j}{N} \right)\right) },
		\end{align}
		where in the last step, after interchanging the order of summation, we substituted $j$ by $N+1-j$. The next task is to evaluate the sum
		\begin{align*}
			S:= \sum_{j=1}^N \frac{(-1)^{j} e^{\frac{2\pi i k j}{N}}}{\cos \pi \left(\left( \frac{a+1}{2N} - \frac{j}{N} \right)\right) }.
		\end{align*}
		Utilizing the fact that $\cos \left(\pi \left( \frac{a+1}{2} - j \right) \right)= (-1)^{j+1} \sin\left( \frac{\pi a}{2} \right)$, it is seen that
		\begin{align*}
			S =-\frac{1}{\sin\left( \frac{\pi a}{2} \right)} \sum_{j=1}^N e^{\frac{2\pi i k j}{N}} \frac{\cos \left( N\pi \left( \frac{a+1}{2N} - \frac{j}{N} \right) \right)}{\cos \left(\pi \left( \frac{a+1}{2N} - \frac{j}{N} \right)\right) }
		\end{align*}
		We next apply the result Lemma \ref{cheby} for odd $N$ to write the above sum as 
		\begin{align*}
			S =\frac{(-1)^{\frac{N+1}{2}}}{\sin\left( \frac{\pi a}{2} \right)} \sum_{j=1}^N e^{\frac{2\pi i k j}{N}} \sideset{}{''}\sum_{\ell = -(N-1)}^{N-1} i^{\ell} e^{-i\ell \pi \left( \frac{a+1}{2N}-\frac{j}{N} \right)}.
		\end{align*}
		Substituting $\ell$ by $2\ell$ and interchanging the order of summation, we obtain
		\begin{align*}
			%S &=\frac{(-1)^{\frac{N+1}{2}}}{\sin\left( \frac{\pi a}{2} \right)} \sum_{j=1}^N e^{\frac{2\pi i k j}{N}} \sum_{\ell = -\frac{N-1}{2}}^{\frac{N-1}{2}} (-1)^{\ell} e^{-2 i \ell \pi \left( \frac{a+1}{2N}-\frac{j}{N} \right)}
		%	\nonumber\\
			S&= \frac{(-1)^{\frac{N+1}{2}}}{\sin\left( \frac{\pi a}{2} \right)} \sum_{\ell = -\frac{N-1}{2}}^{\frac{N-1}{2}} (-1)^{\ell} e^{-\frac{\pi i \ell (a+1)}{N}}\sum_{j=1}^{N} e^{\frac{2\pi i j(k+\ell)}{N}}
			\nonumber\\
			&= \frac{(-1)^{k+\frac{N+1}{2}}}{\sin\left( \frac{\pi a}{2} \right)}e^{\frac{\pi i k(a+1)}{N}} \sum_{\ell = k-\frac{N-1}{2}}^{k+\frac{N-1}{2}} (-1)^{\ell} e^{-\frac{\pi i \ell (a+1)}{N}}\sum_{j=1}^{N} e^{\frac{2\pi i j \ell}{N}},
		\end{align*}
		where in the last step we replaced $\ell$ by $\ell-k$. Applying \cite[Theorem 8.1, p. 158]{Apostol} in the last sum over $j$, we have
		\begin{align}\label{Evaluation of S}
			S = \frac{(-1)^{k+\frac{N+1}{2}}N}{\sin\left( \frac{\pi a}{2} \right)}e^{\frac{\pi i k(a+1)}{N}} \sum_{\substack{\ell = k-\frac{N-1}{2}\\N \mid \ell}}^{k+\frac{N-1}{2}} (-1)^{\ell} e^{-\frac{\pi i \ell (a+1)}{N}}.
		\end{align}
		We next insert \eqref{Evaluation of S} into \eqref{T2 sum} to obtain
		\begin{align*}
			T_1 = \frac{\sqrt{\pi}z^{1/N}}{2\sqrt{N}} \frac{(-1)^{\frac{N+1}{2}}}{\sin\left( \frac{\pi a}{2} \right)} \sum_{k=0}^{2N-1} (-1)^k \exp\left(2N e^{\frac{\pi i k}{N}} z^{\frac{1}{2N}}+ \frac{\pi i k(a+1)}{N} \right) \sum_{\substack{\ell = k-\frac{N-1}{2}\\N \mid \ell}}^{k+\frac{N-1}{2}} (-1)^{\ell} e^{-\frac{\pi i \ell (a+1)}{N}}.
		\end{align*}
		Since the sum over $k$ ranges from $0$ to $2N-1$, $N \mid \ell$ implies that $\ell$ can take only the values $0$ or $N$ for $N=1$ and $0, N$ or $2N$ for $N >1$. Thus we see that $T_1=A_{a, N}(z)$, where $A_{a, N}(z)$ is defined in \eqref{aanzodd}. Along with \eqref{Reduced Meijer G}, this proves \eqref{meijergsimeqn} for $N$ odd.
		%\begin{align}
		%T_2 = \frac{(-1)^{\frac{N+1}{2}}}{2 \sin \left( \frac{\pi a}{2} \right)}\sqrt{\frac{\pi}{N}} z^{1/N} \left [  \sum\limits_{k=0}^{\frac{N-1}{2}} +  \sum\limits_{k= \frac{N+1}{2}}^{\frac{3N-1}{2}} e^{-\pi i a}+  \sum\limits_{k= \frac{3N+1}{2}}^{2N-1} e^{-2\pi i a} \right ](-1)^k \exp \left(2N e^{\frac{\pi i k}{N}} z^{1/2N} + \frac{\pi i k (a+1)}{N} \right)
		%\end{align}
		
			\noindent
		\textbf{Case 2 :}
		In this case, we evaluate the sum $T_2$ for $N$ even. Proceeding analogously as in the case when $N$ is odd and applying \cite[Formula 7.18.1.1, p. 517]{Prudnikov} for the hypergeometric function involved in the term $T_1$, we arrive at
		\begin{equation*}
			T_1 = \frac{\sqrt{\pi}z^{1/N}}{2N^{3/2}} \sum_{k=0}^{2N-1} \exp\left(2N e^{\frac{\pi i k}{N}} (-z)^{\frac{1}{2N}} \right) 
			\sum_{j=1}^N \frac{(-1)^{j+\frac{j}{N}} e^{\frac{2\pi i k j}{N}}}{\cos \left(\pi \left( \frac{a+1}{2N} - \frac{j}{N} \right)\right) }.
		\end{equation*}
		The sum
		\begin{equation*}
			S := \sum_{j=1}^N \frac{(-1)^{j+\frac{j}{N}} e^{\frac{2\pi i k j}{N}}}{\cos \left(\pi \left( \frac{a+1}{2N} - \frac{j}{N} \right) \right)}
		\end{equation*}
		can be evaluated as before. First we write the sum as 
		\begin{equation*}
			S = \frac{1}{\cos\left(\frac{\pi a}{2} \right)}\sum_{j=1}^N (-1)^{j/N} e^{\frac{2\pi i k j}{N}} \frac{\sin \left( N\pi \left( \frac{a+1}{2N} - \frac{j}{N} \right) \right)}{\cos \left(\pi \left( \frac{a+1}{2N} - \frac{j}{N} \right)\right) }
		\end{equation*}
		by utilizing the fact that $\sin \left( N\pi \left( \frac{a+1}{2N} - \frac{j}{N} \right) \right) = (-1)^j \cos\left(\frac{\pi a}{2} \right)$.  Next we apply Lemma \ref{cheby} for even $N$ on the factor $\frac{\sin \left( N\pi \left( \frac{a+1}{2N} - \frac{j}{N} \right) \right)}{\cos \left(\pi \left( \frac{a+1}{2N} - \frac{j}{N} \right)\right) }$ and proceed similarly as in the previous case to arrive at \eqref{meijergsimeqn} with $A_{a,N}(z)$ defined in \eqref{aanzeven}.
	\end{proof}
	\begin{remark}
		Letting $N=1$ in Theorem \ref{meijergsim} and using the formula \cite[p.~73]{dav}\
		\begin{equation}\label{gammaprop}
			\frac{\Gamma(\frac{s}{2})}{\Gamma\left(\frac{1-s}{2}\right)}=\pi^{-1/2}2^{1-s}\Gamma(s)\cos\left(\frac{\pi s}{2}\right)
		\end{equation}
		results in
		\begin{align}\label{g2113}
			G_{1, \, \, 3}^{2, \, \, 1} \left(\begin{matrix}
				1-\frac{a}{2}\\
				1-\frac{a}{2} , 1; \frac{3}{2}
			\end{matrix} \Bigg | \, z \right) =\pi^{-\frac{1}{2}}2^{1-a}\Gamma(a)\cos\left(\frac{\pi a}{2}\right)z^{1-\frac{a}{2}}{}_1F_{2}\left( \left. \begin{matrix}
				1 \\ \frac{1-a}{2}, 1-\frac{a}{2}
			\end{matrix}\right| z \right)-\frac{z\sqrt{\pi}}{\sin\left(\frac{\pi a}{2}\right)}\cosh(2\sqrt{z}),
		\end{align}
		whereas letting $N=2$ in Theorem \ref{meijergsim} and using \eqref{gammaprop} gives
		\begin{align*}
			G_{1, \, \, 5}^{3, \, \, 1} \left(\begin{matrix}
				\frac{3-a}{4}\\
				\frac{3-a}{4} , \frac{1}{2}, 1; \frac{3}{4}, \frac{5}{4}
			\end{matrix} \Bigg | \, z \right)    
			&=2^{\frac{7}{2}-2a}\pi^{-\frac{1}{2}}\Gamma(a-1)\sin\left(\frac{\pi a}{2}\right)z^{\frac{3-a}{4}}{}_1F_{4}\left( \left. \begin{matrix}
				1 \\ \frac{1}{2}-\frac{a}{4}, \frac{3}{4}-\frac{a}{4}, 1-\frac{a}{4}, \frac{5}{4}-\frac{a}{4}
			\end{matrix}\right| -z \right)\nonumber\\
			&\quad+\frac{(-1)^{3/2}\sqrt{\pi z}}{2\sqrt{2}\cos\left(
				\frac{\pi a}{2}\right)}\bigg[\exp{\left(4e^{\frac{\pi i}{4}}z^{\frac{1}{4}}+\frac{\pi i(a+1)}{4}\right)}+e^{-\pi ia}\bigg\{\exp{\left(4e^{\frac{3\pi i}{4}}z^{\frac{1}{4}}+\frac{3\pi i(a+1)}{4}\right)}\nonumber\\
			&\quad-\exp{\left(4e^{\frac{5\pi i}{4}}z^{\frac{1}{4}}+\frac{5\pi i(a+1)}{4}\right)}\bigg\}-e^{-2\pi ia}\exp{\left(4e^{\frac{7\pi i}{4}}z^{\frac{1}{4}}+\frac{7\pi i(a+1)}{4}\right)}\bigg].
		\end{align*}
		We note that \eqref{g2113} is what essentially occurred in the summand of the series on the right-hand side of \eqref{maineqn}.
	\end{remark}
	
	We now derive Theorem \ref{Asymptotic expansion of muKnuN} which will be useful in proving Theorem \ref{Analytic continuation}.
	
\begin{proof}[Theorem \textup{\ref{Asymptotic expansion of muKnuN}}][]
	The Meijer $G$-function in the definition of ${}_{\mu}K_{\nu}^{(N)}(z, w)$ satisfies	$1\leq n \leq p<q,\ 1\leq m\leq q$, and the conditions
		\begin{align*}
			&a_j-b_h\neq1,2,3,\cdots\hspace{5mm}\text{for}\hspace{1mm}j=1,2,\cdots, n,\hspace{0.5mm}\text{and}\hspace{0.5mm}\hspace{2mm}h=1, 2, \cdots, m;    \\
			&a_j-a_t\neq0,\pm1,\pm2,\cdots\hspace{5mm}\text{for}\hspace{1mm}j, t=1,2,\cdots, n,\hspace{0.5mm}\text{and}\hspace{0.5mm}j\neq t, 
		\end{align*}
		and $|\arg(z)|\leq\rho\pi-\delta$ with  $\rho>0,\ \delta\geq0$. Then from \cite[p.~179, Theorem 2]{Luke}, as $|z|\to\infty$, we have
		\begin{align}\label{mgasym}
			G_{p,q}^{\,m,n} \!\left(  \,\begin{matrix} a_1,\cdots , a_p \\ b_1, \cdots b_m; b_{m+1}, \cdots, b_q \end{matrix} \; \Big| z   \right)  \sim \sum_{j=1}^{n}\exp{(-i\pi(\nu+1)a_j)}\Delta_q^{m,n}(j)E_{p,q}\left( z\exp(i\pi(\nu+1))\|a_j \right) .
		\end{align}
		where $\nu :=q-m-n$, $\Delta_q^{m,n}(j):=(-1)^{\nu+1}\frac{\prod\limits_{\substack{\ell=1\\ \ell\neq j}}^n\Gamma(a_j-a_{\ell})\Gamma(1+a_\ell-a_j)}{\prod\limits_{\ell=m+1}^q\Gamma(a_j-b_\ell)\Gamma(1+b_\ell-a_j)}$ and
		\begin{align*}
			&E_{p,q}(z\|a_j):= z^{a_j-1}\tfrac{\prod\limits_{l=1}^{q}\Gamma(1+b_l-a_j)}{\prod\limits_{l=1}^{p}\Gamma(1+a_l-a_j)}{}_qF_{p-1}\left( \left.\begin{matrix}
				1+b_1-a_j,\cdots,1+b_q-a_j \\1+a_1-a_j,\cdots,1+a_{j-1}-a_j,1+a_{j+1}-a_j,\cdots,1+a_p-a_j \end{matrix} \right|-{1\over z} \right) \nonumber\\
		\end{align*}
		We insert, in particular, $m=N+1, \ n=p=1, \ q=2N+1,$ $a_1=1-\mu-\nu-w+{1\over 2N}$ and
		$\left\langle b_\ell \right\rangle_{\ell=1}^{2N+1} = \left\lbrace \left \langle {j\over N} \right\rangle_{j=1}^N, {1\over 2}+{1\over 2N}-\nu, {1}+{1\over 2N}-w, \left\langle 1+{3\over 2N}-{j\over N} \right\rangle_{j=2}^{N} \right\rbrace$, and replace $z$ by ${z^2\over 4}$ 
		in the above result so as to get
		
		\begin{align}\label{Delta}
		%	&\nu = 2N+1-(N+1)-1=N-1,\\
			\Delta_q^{m,n}(j)&=\frac{(-1)^{N} }{\left(\Gamma(-\mu-\nu)\prod_{i=2}^{N}\Gamma(\tfrac{i-1}{N}-\mu-\nu-w)\right) \left( \Gamma(1+\mu+\nu)\prod_{i=2}^{N}\Gamma(1-{i-1\over N}+\mu+\nu+w) \right)}
			\end{align}
		and
		\begin{align}\label{E}
			&E_{p,q}\left( \frac{z^2}{4}\exp(i\pi N)\|a_1 \right)\nonumber\\
			& =\left( \tfrac{z^2}{4}e^{i\pi N}\right)^{a_1-1} \prod_{i=1}^{N}\Gamma\left( 1+\tfrac{i}{N} -a_1\right) \Gamma\left(1+\tfrac{1}{2}+\tfrac{1}{2N}-\nu-a_1 \right)\Gamma\left(2+\tfrac{1}{2N}-w-a_1 \right)\prod_{i=2}^{N}\Gamma\left( 2+\tfrac{3}{2N}-\tfrac{i}{N}-a_1\right)\nonumber\\  
			&\quad\times {}_{2N+1}F_0\left(\left.  \left\langle 1+\tfrac{i}{N}-a_1 \right\rangle_{i=1}^{N}, 1+\tfrac{1}{2}+\tfrac{1}{2N}-\nu-a_1, 2+\tfrac{1}{2N}-w-a_1 , \left\langle 2+\tfrac{3}{2N}-\tfrac{i}{N }-a_1 \right\rangle_{i=2}^{N};- \right| \tfrac{-4}{z^2}e^{-i\pi N}\right)\nonumber\\
			&=\left( \tfrac{z^2}{4}e^{i\pi N}\right)^{a_1-1} \prod_{i=1}^{N}\Gamma\left( \tfrac{2i-1}{2N }+\mu+\nu+w\right)  \Gamma\left(\tfrac{1}{2}+\mu+w \right)\Gamma\left(1+\mu+\nu \right)\prod_{i=2}^{N}\Gamma\left( 1-\tfrac{i-1}{N}+\mu+\nu+w\right)\sum_{k=0}^{m}\tfrac{(-1)^k}{k!}\nonumber\\ 
			&\quad\times\prod_{i=1}^{N}\left( \tfrac{2i-1}{2N}+\mu+\nu+w\right)_k  \left(\tfrac{1}{2}+\mu+w \right)_k \left(1+\mu+\nu \right)_k\prod_{i=2}^{N}\left( 1-\tfrac{i-1}{N}+\mu+\nu+w\right)_k \left(\tfrac{4}{z^2} \right)^ke^{-i\pi k N}+\mathcal{O}\left(\tfrac{1}{z^{2m+4-2a}} \right).  
		\end{align}
		%Putting $a=1-\mu-\nu-w+{1\over 2N}$, we get
%		\begin{align*}
%			&E_{p,q}\left( (z^2/4)\exp(i\pi N)\|a_j \right)\\
%			&=\left( {z^2\over 4}e^{i\pi N}\right)^{a-1} \left(\prod_{i=1}^{N}\Gamma\left( {2i-1\over 2N }+\mu+\nu+w\right)  \right) \Gamma\left({1\over 2}+\mu+w \right) \Gamma\left(1+\mu+\nu \right)\left(\prod_{i=2}^{N}\Gamma\left( 1-{i-1\over N}+\mu+\nu+w\right) \right)\\ 
%			&\times {}_{2N+1}F_0\left(\left.  \left\langle {2i-1\over 2N }+\mu+\nu+w\right\rangle_{i=1}^{N}, {1\over 2}+\mu+w , 1+\mu+\nu, \left\langle 1-{i-1\over N}+\mu+\nu+w \right\rangle_{i=2}^{N} \right| -{4\over z^2}e^{-i\pi N}\right)\\
%			&=\left( {z^2\over 4}e^{i\pi N}\right)^{a-1} \left(\prod_{i=1}^{N}\Gamma\left( {2i-1\over 2N }+\mu+\nu+w\right)  \right) \Gamma\left({1\over 2}+\mu+w \right)\Gamma\left(1+\mu+\nu \right)\left(\prod_{i=2}^{N}\Gamma\left( 1-{i-1\over N}+\mu+\nu+w\right) \right)\\ 
%			&\times\sum_{k=0}^{m}{(-1)^k \over k!}\left(\prod_{i=1}^{N}\left( {2i-1\over 2N }+\mu+\nu+w\right)_k  \right) \left({1\over 2}+\mu+w \right)_k \left(1+\mu+\nu \right)_k\left(\prod_{i=2}^{N}\left( 1-{i-1\over N}+\mu+\nu+w\right)_k \right)\\ &\times\left({4\over z^2} \right)^ke^{-i\pi k N}+\mathcal{O}\left({1\over z^{2m+4-2a}} \right).  
%		\end{align*}
		Thus, substituting \eqref{Delta} and \eqref{E} in \eqref{mgasym}, we see that as $z \to \infty$,
		\begin{align*}
			&G_{1 \ 2N+1}^{N+1 \ 1}\left(\left.\begin{matrix}
				1+{1\over 2N}-\mu-\nu-w \\  \left\langle {i\over N}\right\rangle_{i=1}^{N}, {1\over 2}+{1\over 2N}-\nu;1+{1\over 2N}-w, \left\langle {1}+{3\over2N}-{i\over N}\right\rangle_{i=2}^{N}	\end{matrix} \right| \frac{z^2}{4} \right) \\
			&=\frac{(-1)^{N} }{\Gamma(-\mu-\nu)\prod_{i=2}^{N}\Gamma({i-1\over N}-\mu-\nu-w) \Gamma(1+\mu+\nu)\prod_{i=2}^{N}\Gamma(1-{i-1\over N}+\mu+\nu+w)}\\
			&\quad\times\sum_{k=0}^{m}{(-1)^k \over k!}\prod_{i=1}^{N}\Gamma\left( {2i-1\over 2N }+\mu+\nu+w+k\right)  \Gamma \left({1\over 2}+\mu+w +k\right) \Gamma\left(1+\mu+\nu +k\right)\\ 
			&\quad\times\prod_{i=2}^{N}\Gamma\left( 1-{i-1\over N}+\mu+\nu+w+k\right) \left({4\over z^2} \right)^{k+\mu+\nu+w-{1 \over 2N}}e^{-i\pi (k+1) N}+\mathcal{O}\left({1\over z^{2m+2-{1 \over N}+2\mu+2\nu+2w}} \right),
		\end{align*} 
	which, with the help of the definition \eqref{muknug}, the reflection formula \eqref{refl} and the fact that
	\begin{equation*}
	\prod_{i=1}^{N}\Gamma\left( {2i-1\over 2N }+\mu+\nu+w+k\right) \prod_{i=2}^{N}\Gamma\left( 1-{i-1\over N}+\mu+\nu+w+k\right)=\prod_{i=1}^{2N-1}\Gamma\left(\frac{i}{2N}+\mu+\nu+w+k\right),
	\end{equation*}
	yields \eqref{estk} upon simplification.
		
%		Using the definition \eqref{muKnu}, we have
%		\begin{align*}
%			{}_\mu K_{\nu}^{(N)}(z, w) &= \frac{2^{\mu+\frac{2}{N}-1+2\mu+2\nu+2w-{1\over N}} \pi^{(1-N)\nu} z^{w+\nu-\frac{2}{N}-2\mu-2\nu-2w+{1\over N}}}{\left(\Gamma(-\mu-\nu)\prod_{i=2}^{N}\Gamma({i-1\over N}-\mu-\nu-w)\right) \left( \Gamma(1+\mu+\nu)\prod_{i=2}^{N}\Gamma(1-{i-1\over N}+\mu+\nu+w) \right)}\\
%			&\times\sum_{k=0}^{m}{(-1)^{k(N+1)} \over k!}\prod_{i=1}^{N}\Gamma\left( {2i-1\over 2N }+\mu+\nu+w+k\right)  \Gamma \left({1\over 2}+\mu+w +k\right) \Gamma\left(1+\mu+\nu +k\right)\\ 
%			&\prod_{i=2}^{N}\Gamma\left( 1-{i-1\over N}+\mu+\nu+w+k\right) \left({z \over 2} \right)^{-2k}+\mathcal{O}\left({1\over z^{2m+2+{1 \over N}+2\mu+2\nu+2w}} \right).
%		\end{align*}
%		Finally we apply reflection formula to arrive at
%		\begin{align}
%			{}_\mu K_{\nu}^{(N)}(z, w) &= \frac{2^{3\mu+2\nu+2w+\frac{1}{N}-1} \pi^{(1-N)\nu-N} }{z^{2\mu+\nu+w+{1\over N}}}\sin (\pi(\mu+\nu))\prod_{i=2}^{N}\sin \left(\pi \left( \mu+\nu+w-{i-1\over N}\right) \right) \nonumber  \\
%			&\times\sum_{k=0}^{m}{(-1)^{k(N+1)+N} \over k!}\prod_{i=1}^{N}\Gamma\left( {2i-1\over 2N }+\mu+\nu+w+k\right)  \Gamma \left({1\over 2}+\mu+w +k\right) \Gamma\left(1+\mu+\nu +k\right)\nonumber \\ 
%			&\prod_{i=2}^{N}\Gamma\left( 1-{i-1\over N}+\mu+\nu+w+k\right) \left({z \over 2} \right)^{-2k}+\mathcal{O}\left({1\over z^{2m+2+{1 \over N}+2\mu+\nu+w}} \right),
%		\end{align}
%		which concludes our lemma.
		
		\section{Proofs of the main transformations}
		
		\subsection{Proof of Theorem \ref{In terms of G}}
		It follows from \cite[Equation (6.16)]{DMV} that for $f(n) = \exp(-ny)$ with $y>0$ and $-1<\re(a)<N$, we have
		\begin{align}\label{ffinal}
			\sum_{n=1}^\infty \sigma_a^{(N)}(n) e^{-ny}  &= -\frac{\zeta(-a)}{2} + \frac{\zeta(N-a)}{y} + \frac{1}{N} \frac{\Gamma \left(\frac{1+a}{N}\right) \zeta \left(\frac{1+a}{N}\right)}{y^{\frac{1+a}{N}}} \nonumber\\
			&+ \frac{(2\pi)^{(N+1)\left(\frac{1+a}{N} \right)-a}}{\pi^2} \sum_{n=1}^\infty S_a^{(N)}(n) \int_0^\infty H_a^{(N)} \left( (2\pi)^{1+1/N} (nt)^{1/N} \right) t^{\frac{1+a}{N}-1} \exp(-ty) \, dt,
		\end{align}
		where $H_a^{(N)}(x)$ is defined in \eqref{HaNx}.
	From \cite[p.~40-41]{DMV}, for $ \max\left\{ 0,  \frac{1 -N +  \re(a)}{N} \right\} < c=\textup{Re}(s)\leq\frac{1+\textup{Re}(a)}{N (N+1)} $,
		\begin{equation*}
			\mathcal{I} := \int_0^\infty H_a^{(N)} \left( \alpha t^{1/N} \right) t^{\frac{1+a}{N}-1} \exp(-ty) \, dt=\frac{{{y^{ - \frac{{(1 + a)}}{N}}}}}{{4i}}\int\limits_{(c)} \frac{{\Gamma (Ns)\cos \left( {\frac{{\pi Ns}}{2}} \right)}}{{\cos \left( {\frac{\pi }{{2N}}(1 + a - Ns)} \right)}} {{{\left( {\frac{{{\alpha ^N}}}{y}} \right)}^{ - s}} ds},
		\end{equation*}
		where $\alpha = (2\pi)^{1+1/N} n^{1/N}$.
		Letting $b=\alpha^N/y$, using the fact that 
	\begin{equation}\label{refl1}
		\frac{\pi}{\cos(\pi s)}=\Gamma\left(\frac{1}{2}+s\right)\Gamma\left(\frac{1}{2}-s\right)
		\end{equation}
		 and substituting $s$ by $2s$, we see that for $ \max\left\{ 0,  \frac{1 -N +  \re(a)}{2N} \right\} < c'=\textup{Re}(s)\leq\frac{1+\textup{Re}(a)}{2N (N+1)} $,
		\begin{align*}
			\mathcal{I}&= \frac{{{y^{ - \frac{{1 + a}}{N}}}}}{{2i}}\int\limits_{(c')} {\frac{{\Gamma (2Ns)\Gamma \left( {\frac{1}{2} + \frac{{1 + a}}{{2N}} - s} \right)\Gamma \left( {\frac{1}{2} - \frac{{1 + a}}{{2N}} + s} \right)}}{{\Gamma \left( {\frac{1}{2} + Ns} \right)\Gamma \left( {\frac{1}{2} - Ns} \right)}}{b^{ - 2s}}ds}  
		\end{align*}
		Employ the Gauss multiplication formula \eqref{gmf} to obtain
		\begin{align*}
			&\Gamma (2Ns) = {(2\pi )^{\frac{1}{2} - N}}{(2N)^{2Ns - \frac{1}{2}}}\prod\limits_{k = 0}^{2N - 1} {\Gamma \left( {s + \frac{k}{{2N}}} \right)}   \nonumber \\ 
			& \Gamma \left( {\frac{1}{2} \pm Ns} \right) = \Gamma \left( {N\left( {\frac{1}{{2N}} \pm s} \right)} \right) = {(2\pi )^{\frac{1}{2}\left( {1 - N} \right)}}{N^{ \pm Ns}}\prod\limits_{k = 0}^{N - 1} {\Gamma \left( {\frac{1}{{2N}} \pm s + \frac{k}{N}} \right)}.
		\end{align*}
		Thus the integral evaluates to
		\begin{align*}
			\mathcal{I}&= \frac{{{y^{ - \frac{{(1 + a)}}{N}}}}}{{4i\sqrt {N\pi } }}\int\limits_{(c')} {\frac{{ \Gamma \left( {\frac{1}{2} - \frac{{1 + a}}{{2N}} + s} \right)\Gamma \left( {\frac{1}{2} + \frac{{1 + a}}{{2N}} - s} \right)\prod\limits_{k = 0}^{2N - 1} {\Gamma \left( {s + \frac{k}{{2N}}} \right)}}}{{\prod\limits_{k = 0}^{N - 1} {\Gamma \left( {\frac{1}{{2N}} + \frac{k}{N} + s} \right)} \Gamma \left( {\frac{1}{{2N}} + \frac{k}{N} - s} \right)}}{{\left( {\frac{{{{(2N)}^{2N}}}}{{{b^2}}}} \right)}^{s}}ds} \nonumber\\
			%&= \frac{{{y^{ - \frac{{1 + a}}{N}}}}\sqrt{\pi}}{{2\sqrt {N} }}{1\over 2\pi i}\int\limits_{(c')} {\frac{{\prod\limits_{k = 0}^{2N - 1} {\Gamma \left( {s + \frac{k}{{2N}}} \right)} \Gamma \left( {\frac{1}{2} - \frac{{1 + a}}{{2N}} + s} \right)\Gamma \left( {\frac{1}{2} + \frac{{1 + a}}{{2N}} - s} \right)}}{{\prod\limits_{k = 0}^{N - 1} {\Gamma \left( {\frac{1}{{2N}} + \frac{k}{N} + s} \right)} \Gamma \left( {\frac{1}{{2N}} + \frac{k}{N} - s} \right)}}{{\left( {\frac{{{{(2N)}^{2N}}}}{{{b^2}}}} \right)}^{s}}ds} \nonumber\\
			&= \frac{{{y^{ - \frac{{(1 + a)}}{N}}}}\sqrt{\pi}}{{2\sqrt {N} }}{1\over 2\pi i}\int\limits_{(c')} {\frac{{ \Gamma \left( {\frac{1}{2} - \frac{{1 + a}}{{2N}} + s} \right)\Gamma \left( {\frac{1}{2} + \frac{{1 + a}}{{2N}} - s} \right)\prod\limits_{k = 0}^{N-1} {\Gamma \left( {s + \frac{k}{{N}}} \right)}}}{\prod\limits_{k=0}^{N-1}{\Gamma \left( {\frac{1}{{2N}} + \frac{k}{N} - s} \right)}}{{\left( {\frac{{{{(2N)}^{2N}}}}{{{b^2}}}} \right)}^{s}}ds},
		\end{align*}
		where in the last step, we use the fact that
		$$\prod_{k = 0}^{2N - 1} \Gamma \left( s + \frac{k}{2N}\right) = \prod_{k = 0}^{N - 1} \Gamma \left( s + \frac{2k+1}{2N}\right)\prod_{k = 0}^{N-1} \Gamma \left( s + \frac{2k}{2N}\right).$$
		Substituting $s=-w+\frac{1}{N}$ and replacing $k$ by $i-1$, we see that for $ \frac{1}{N}-\frac{1+\textup{Re}(a)}{2N (N+1)} \leq c''=\textup{Re}(w)<\min\left\{ \frac{1}{N},  \frac{N + 1- \re(a)}{2N} \right\} $,
		\begin{equation*}
			\mathcal{I}= \frac{{{2N^{3/2}}}\sqrt{\pi}}{y^{\frac{{1 + a}}{N}}{b^{\frac{2}{N}} }}{1\over 2\pi i}\int\limits_{(c'')} {\frac{{ \Gamma \left( {\frac{1}{2} - \frac{{ a-1}}{{2N}} - w} \right)\Gamma \left( {\frac{1}{2} + \frac{{a-1}}{{2N}} + w} \right)\prod\limits_{i = 1}^{N} {\Gamma \left( {-w + \frac{i}{{N}}} \right)}}}{\prod\limits_{i=1}^{N}{\Gamma \left( {-\frac{3}{{2N}} + \frac{i}{N} + w} \right)}}{{\left( {\frac{{{{(2N)}^{2N}}}}{{{b^2}}}} \right)}^{-w}}dw}.
		\end{equation*}
	Now using the above representation for $\mathcal{I}$ in \eqref{ffinal} and then expressing it in terms of Meijer $G$-function using the definition in \eqref{MeijerG}, we arrive at \eqref{In terms of G_eqn} for $-1<\re(a)<N$. With the help of \eqref{muknug}, it is straightforward to see that \eqref{Using muKnuN} is equivalent to \eqref{In terms of G_eqn}.
		
		Next, we extend the validity of \eqref{Using muKnuN} to $\textup{Re}(a)>-1$.  To that end, we first show the uniform convergence of the series $\sum\limits_{n=1}^\infty \frac{S_a^{(N)}(n)}{n^{\frac{a}{2N}}}{}_{\frac{1}{2}} K_{\frac{a}{2N}}^{(N)}\left(\frac{4\pi^{N+1}n}{yN^N}, 0\right)$ in Re$(a)\geq-1+\epsilon$ for any $\epsilon>0$. From \eqref{defbf},
		\begin{equation*}
			|S_a^{(N)}(n)| = \sum_{d_1^N d_2 = n}d_2^{\frac{1+\re(a)}{N}-1} \leq \sum_{d_1d_2 = n}d_2^{\frac{1+\re(a)}{N}-1} = \sigma_{\frac{1+\re(a)}{N}-1}(n)
		\end{equation*}
		Thus it follows from the elementary bound on the divisor function that 
		\begin{equation}\label{Bound of divisor function}
			|S_a^{(N)}(n) n^{-\frac{a}{2N}}| \ll n^{\frac{1}{2}\left\vert \frac{1+\re(a)}{N}-1 \right\vert + \frac{1}{2N} - \frac{1}{2} + \epsilon} = n^{\max \left(-\frac{\re(a)}{2N}, \frac{\re(a)}{2N} + \frac{1}{N} - 1 \right) +\epsilon} ,
		\end{equation}
		for every $\epsilon>0$. On the other hand, Lemma \ref{Asymptotic expansion of muKnuN} provides ${}_{\frac{1}{2}}K_{\frac{a}{2N}}^{(N)}\left( \frac{4 \pi^{N+1} n}{y N^N}, 0 \right) = \mathcal{O}(n^{-1-\frac{1}{N}-\frac{\re(a)}{2N}})$. Hence the series $\sum\limits_{n=1}^\infty \frac{S_a^{(N)}(n)}{n^{\frac{a}{2N}}}{}_{\frac{1}{2}} K_{\frac{a}{2N}}^{(N)}\left(\frac{4\pi^{N+1}n}{yN^N}, 0\right)$ converges absolutely and uniformly as long as 
		%$$
		%\max \left(-\frac{\re(a)}{2N}, \frac{\re(a)}{2N} + \frac{1}{N} - 1 \right) +\epsilon-1-\frac{1}{N}-\frac{\re(a)}{2N}<-1
		%$$
		$\re(a)>-1$. Since the summand of the series is analytic for $\re(a)>-1$, it follows from Weierstrass’ theorem on analytic functions that the series is an analytic function of $a$ in $\re(a)>-1$. The left-hand side of \eqref{Using muKnuN} is also analytic for $\re(a)>-1$, thus it follows from the principle of analytic continuation that \eqref{Using muKnuN} holds for $\re(a)>-1$ and $y>0$. Using similar method as above, both sides of \eqref{Using muKnuN} are seen to be analytic, as a function of $y$, in $\re(y) > 0$. This proves \eqref{Using muKnuN} in its entirety.
		
		\subsection{Proof of Theorem \ref{Analytic continuation}}
		As was shown in the proof of Theorem \ref{In terms of G}, the identity in \eqref{Using muKnuN} holds for Re$(a)>-1$. We first rewrite it in an equivalent form and then use analytic continuation. 
		
		Recall the definition of $C_{m, N}$ from \eqref{cmn}, where $m\geq0$. We begin by adding and subtracting the term 
		\begin{equation*}
		2^{\frac{1}{2} + \frac{a+1}{N}} \pi^{\frac{(1-N)a}{2N} -N} \left(\frac{4\pi^{N+1}n}{yN^N}\right)^{-1-\frac{1}{N}-\frac{a}{2N}} \frac{\sin\(\frac{\pi}{2}(N-a)\)}{2^{N-1}} C_{m,N}\left(\frac{1}{2}, \frac{a}{2N}, 0, \frac{4\pi^{N+1}n}{yN^N}\right)
		\end{equation*}
		 inside the summand of the series on the right-hand side of \eqref{Using muKnuN} and then simplify a bit to arrive at
		\begin{align}\label{additional}
			&\sum_{n=1}^\infty \sigma_a^{(N)}(n) e^{-ny}  + \frac{\zeta(-a)}{2} - \frac{\zeta(N-a)}{y} - \frac{1}{N} \frac{\Gamma \left(\frac{1+a}{N}\right) \zeta \left(\frac{1+a}{N}\right)}{y^{\frac{1+a}{N}}} = \frac{2 (2\pi)^{\frac{1}{N}-\frac{1}{2}} N^{\frac{a-1}{2}}}{y^{\frac{1}{N}+\frac{a}{2N}}} \sum_{n=1}^\infty \frac{S_a^{(N)}(n)}{n^{\frac{a}{2N}}} \bigg [{}_{\frac{1}{2}} K_{\frac{a}{2N}}^{(N)}\left(\tfrac{4\pi^{N+1}n}{yN^N}, 0\right) \nonumber\\
			&- \frac{2^{\frac{1}{2} + \frac{a+1}{N}} \pi^{\frac{(1-N)a}{2N} -N}} {\left(\frac{4\pi^{N+1}n}{yN^N}\right)^{1+\frac{1}{N}+\frac{a}{2N}}}\frac{\sin\(\frac{\pi}{2}(N-a)\)}{2^{N-1}}   C_{m,N}\left(\frac{1}{2}, \frac{a}{2N}, 0; \frac{4\pi^{N+1}n}{yN^N}\right)\bigg] +D_{m, N}(a, y),
		\end{align}
	where
	\begin{align*}
	D_{m, N}(a, y)&:= \frac{y N^{a+N+\frac{1}{2}}}{2 \pi^{2N+a+\frac{5}{2}}} \frac{\sin\(\frac{\pi}{2}(N-a)\)}{2^{N-1}}  \sum_{n=1}^{\infty}\frac{S_a^{(N)}(n)}{n^{1+{a+1\over N}}}C_{m,N}\left(\frac{1}{2}, \frac{a}{2N}, 0; \frac{4\pi^{N+1}n}{yN^N}\right).
\end{align*}
Inserting the definition of $C_{m,N}\left(\tfrac{1}{2}, \tfrac{a}{2N}, 0, \tfrac{4\pi^{N+1}n}{yN^N}\right)$ from \eqref{cmnsimplified} and then using the formula \cite[Equation (1.14)]{DMV} 
\begin{equation}
	\sum_{n=1}^{\infty}\frac{S_a^{(N)}(n)}{n^{s}} = \zeta(Ns)\zeta\left(s+1-\frac{1+a}{N}\right)\hspace{8mm}\left(\re(s)>\max\left\{\frac{1}{N}, \frac{1+\re(a)}{N}\right\}\right)\label{Sdirichlet}
\end{equation}
which is justified since Re$(a)>-1$, we get
\begin{align*}
	D_{m, N}(a, y)
	%&=\frac{y N^{a+N+\frac{1}{2}}}{2^N \pi^{2N+a+\frac{5}{2}}} \sin\(\frac{\pi}{2}(N-a)\) \sum_{k=0}^m\frac{(-1)^{k(N+1)+N}}{\left(\tfrac{2\pi^{N+1}}{yN^N}\right)^{2k}}\prod_{i=1}^{2N} \Gamma \left(\tfrac{1}{2}+\tfrac{a}{2N}+k+\tfrac{i}{2N} \right)\nonumber\\
	%&\qquad\qquad\qquad\qquad\qquad\qquad\qquad\quad\times\zeta(2kN+N+a+1)\zeta(2k+2)\nonumber\\
	&=\frac{y}{\pi^2(2\pi)^{1+N+a}}\sin\(\frac{\pi}{2}(N-a)\) \sum_{k=0}^m\frac{(-1)^{k(N+1)+N}}{\left(\tfrac{(2\pi)^{N+1}}{yN^N}\right)^{2k}} \Gamma \left(2kN+N+a+1 \right)\nonumber\\
	&\qquad\qquad\qquad\qquad\qquad\qquad\qquad\quad\times\zeta(2kN+N+a+1)\zeta(2k+2)\nonumber\\
	&=\frac{y}{2\pi^{2}} \sum_{k=0}^m \left(-\frac{y^2}{4\pi^2}\right)^k   \zeta(-2kN-N-a)\zeta(2k+2), 
	\end{align*}
where in the last step, we employed \eqref{zetafe}.

Substituting this expressions for $D_{m, N}(a, y)$ in \eqref{additional}, we see that \eqref{Eqn:Analytic continuation} holds for Re$(a)>-1$. 
		The validity of \eqref{Eqn:Analytic continuation} is now extended to $\re(a)>-(2m+2)N-1$ using Theorem \ref{Asymptotic expansion of muKnuN}. To that end, let $\mu=1/2, \nu=a/(2N), w=0$ and $z=\frac{4\pi^{N+1}n}{yN^N}$ in Theorem \ref{Asymptotic expansion of muKnuN}. Also, using \eqref{gmf} and \eqref{refl1}, we see that
			\begin{align*}
						\prod_{i=1}^N \cos\left(\pi\left(\frac{a}{2N} - \frac{i-1}{N} \right) \right) &= \prod_{i=1}^N \frac{\pi}{\Gamma\left(\frac{1}{2}+\frac{a}{2N}- \frac{i-1}{N} \right)\Gamma\left(\frac{1}{2}-\frac{a}{2N}+ \frac{i-1}{N} \right)} \nonumber\\
%						&=\pi^N \prod_{i=1}^N \frac{1}{\Gamma\left(-\frac{1}{2}+\frac{a+2}{2N}+\frac{i-1}{N} \right)\Gamma\left(\frac{1}{2}-\frac{a}{2N}+ \frac{i-1}{N} \right)}\nonumber\\
						&= \frac{\pi^N}{\left\lbrace (2\pi)^{\frac{N-1}{2}}N^{\frac{1}{2} - N(-\frac{1}{2}+\frac{a+2}{2N})}\Gamma\left(-\frac{N}{2}+\frac{a}{2}+1 \right)\right\rbrace 
								\left\lbrace (2\pi)^{\frac{N-1}{2}}N^{\frac{1}{2} - N(\frac{1}{2}-\frac{a}{2N})}\Gamma\left(\frac{N}{2}-\frac{a}{2} \right)\right\rbrace}\nonumber\\
%						&=\frac{\pi^N}{(2\pi)^{N-1}N^{1+N/2-a/2-1-N/2+a/2}} \times \frac{\sin\left(\frac{\pi N}{2} - \frac{a\pi}{2} \right)}{\pi}
						& = \frac{\sin\left(\frac{\pi N}{2} - \frac{a\pi}{2} \right)}{2^{N-1}}.
					\end{align*}
		Therefore, as $n\to\infty$, Theorem \ref{Asymptotic expansion of muKnuN} and \eqref{cmn} imply
		\begin{align}\label{additional1}
{}_{\frac{1}{2}} K_{\frac{a}{2N}}^{(N)}\left(\tfrac{4\pi^{N+1}n}{yN^N}, 0\right) - \frac{2^{\frac{1}{2} + \frac{a+1}{N}} \pi^{\frac{(1-N)a}{2N} -N}} {\left(\frac{4\pi^{N+1}n}{yN^N}\right)^{1+\frac{1}{N}+\frac{a}{2N}}}\frac{\sin\(\frac{\pi}{2}(N-a)\)}{2^{N-1}} C_{m,N}\left(\frac{1}{2}, \frac{a}{2N}, 0, \frac{4\pi^{N+1}n}{yN^N}\right)=\mathcal{O}\left(n^{-2m-3-\frac{1}{N}-\frac{\re(a)}{2N}}\right).
\end{align}		
Thus, along with \eqref{Bound of divisor function}, \eqref{additional1} implies that for any $\epsilon>0$, 
the summand of the series on the right hand side of \eqref{Eqn:Analytic continuation} is $\mathcal{O}\left(n^{\max \left(-\frac{\re(a)}{2N}, \frac{\re(a)}{2N} + \frac{1}{N} - 1 \right)-2m-3-\frac{1}{N}-\frac{\re(a)}{2N}+\epsilon} \right)$.
%		\begin{align*}
%			&\frac{S_a^{(N)}(n)}{n^{\frac{a}{2N}}} \bigg [{}_{\frac{1}{2}} K_{\frac{a}{2N}}^{(N)}\left(\tfrac{4\pi^{N+1}n}{yN^N}, 0\right) - \frac{2^{\frac{1}{2} + \frac{a+1}{N}} \pi^{\frac{(1-N)a}{2N} -N}} {\left(\frac{4\pi^{N+1}n}{yN^N}\right)^{1+\frac{1}{N}+\frac{a}{2N}}}\frac{\sin\(\frac{\pi}{2}(N-a)\)}{2^{N-1}} C_{m,N}\left(\frac{1}{2}, \frac{a}{2N}, 0, \frac{4\pi^{N+1}n}{yN^N}\right)\bigg] \nonumber\\
%			&= \mathcal{O}\left(n^{\max \left(-\frac{\re(a)}{2N}, \frac{\re(a)}{2N} + \frac{1}{N} - 1 \right)-2m-3-\frac{1}{N}-\frac{\re(a)}{2N}+\epsilon} \right)
%		\end{align*}
		Thus the above series is uniformly convergent for $\re(a)\geq-(2m+2)N-1+\epsilon$ for any $\epsilon>0$. Since the summand of the above series is analytic in $\re(a)>-(2m+2)N-1$, Weierstrass' theorem on analytic function implies that this series represents an analytic function of $a$ for $\re(a)>-(2m+2)N-1$. 
		Since the left-hand side of \eqref{Eqn:Analytic continuation} as well as the finite sum on its right-hand side are also analytic in $\re(a)>-(2m+2)N-1$, by invoking the principal of analytic continuation we conclude that \eqref{Eqn:Analytic continuation} holds for $\re(a)>-(2m+2)N-1$. This completes the proof of the theorem.
	\end{proof}
	
	\section{New results on derivatives of Mittag-Leffler functions}
	The two-variable Mittag-Leffler function $E_{\alpha, \beta}(z)$ of Wiman \cite{wiman} is defined by
	\begin{align}\label{2varmldef}
		E_{\alpha, \beta}(z):=\sum_{k=0}^{\infty}\frac{z^k}{\G(\alpha k+\beta)}\hspace{10mm}(\textup{Re}(\alpha)>0, \textup{Re}(\beta)>0).
	\end{align}
These functions have found applications in various fields of science and engineering such as random walks, L\'{e}vy flights, fractional Ohm's law, time and space fractional diffusion, nonlinear waves, electric field relaxations, viscoelastic systems, chemical reactions and statistical distributions. In Mathematics, they occur in fractional calculus, fractional order integral or differential equations etc. See \cite{apelblat}, \cite{gkmr}, \cite{mieghem}, and the references therein.

Differentiating both sides  of \eqref{2varmldef} with respect to $\beta$ yields
\begin{equation*}
	\frac{\partial E_{\alpha, \beta}(z)}{\partial\beta}=-\sum_{h=0}^{\infty}\frac{\psi(\alpha h+\beta)}{\Gamma(\alpha h+\beta)}z^h.
\end{equation*}
%In Proposition \ref{}, we studied a generalization of the function $\left.\frac{\partial E_{2, \b}}{\partial\beta}\right|_{\beta=1}$ and obtained its closed-form evaluation in terms of hyperbolic functions and hyperbolic sine and cosine integrals. 

Closed-form evaluations of derivatives of two-variable Mittag-Leffler functions with respect to their parameters were recently given in \cite{apelblat}, where the calculations were performed using \emph{Mathematica}.
 
	Let $\re(z)>0$. From Lemma 3.2 of \cite{DGKM} and Lemma 9.1 from \cite{dkk}, we have, following the notation in \eqref{brevity}, 
	\begin{equation}
			\sum_{h=0}^{\infty} \frac{(\psi(2h+1)-\log(z))}{\Gamma(2h+1)}z^{2h}=\int_{0}^{\infty}\frac{t\cos(t)}{t^2+z^2}\, dt=\left(\sinh\operatorname{Shi}-\cosh\operatorname{Chi}\right)(z)\label{2equiv},
	\end{equation}
which gives the closed form evaluation of $\left.\frac{\partial E_{2, \b}(z^2)}{\partial\beta}\right|_{\beta=1}$ since $\sum_{h=0}^{\infty}\frac{z^{2h}}{\G(2h+1)}=\cosh(z)$.

\begin{remark}\label{extension}
The identity obtained by equating the extreme sides of \eqref{2equiv}, that is,
\begin{equation}
\sum_{h=0}^{\infty} \frac{(\psi(2h+1)-\log(z))}{\Gamma(2h+1)}z^{2h}=\left(\sinh\operatorname{Shi}-\cosh\operatorname{Chi}\right)(z)\label{2equiv1}
\end{equation}
 actually holds for $|\arg(z)|<\pi$ as can be seen by analytic continuation.
\end{remark}

Dixon and Ferrar \cite[Equation (3.12)]{dixfer1} defined the function
\begin{equation*}
K_{/-\nu}(z)=\frac{1}{\pi}\sum_{h=0}^{\infty}\frac{(z/2)^{2h}}{\Gamma(h+1)\Gamma(h+\nu+1)}\left\{2\log\left(\frac{z}{2}\right)-\psi(h+1)-\psi(h+\nu+1)\right\},	
\end{equation*}
which played an important role in their proof of the Vorono\"{\dotlessi} summation formula for $\sum_{n=1}^{\infty}d(n)f(n)$. Logarithmically differentiating \eqref{gmf} so as to get \cite[p.~138, Equation \textbf{5.5.9}]{NIST}
\begin{equation*}
	\sum_{k=1}^{m}\psi\left(z+\frac{k-1}{m}\right)=m\left(\psi(mz)-\log(m)\right),
\end{equation*}
we see that the extreme left-hand side of \eqref{2equiv} is nothing but $-\frac{1}{2}\pi^{3/2}K_{/-\frac{1}{2}}(z)$.
	In the following proposition, we generalize \eqref{2equiv} which gives an \emph{almost} closed-form evaluation of  $\left.\frac{\partial E_{2N, \b}(z^{2N})}{\partial\beta}\right|_{\beta=1}$. The first part of this proposition will be instrumental in proving Corollary \ref{gencompanion}.
	\begin{proposition}\label{kuchnahi}
		Let $\re(z)>0$ and $N\in\mathbb{N}$. Then 
		\begin{itemize}
			\item[{\rm(a)}]
			For $N$ odd, we have 
			\begin{align}\label{Generalized Dixon-Ferror identity odd}
				\sum_{h=0}^{\infty} \frac{\left(\psi(2Nh+1)-\log z\right)}{\Gamma(2Nh+1)}z^{2Nh} = \frac{1}{N}\sum_{k=-\frac{N-1}{2}}^{\frac{N-1}{2}} \left(\sinh\operatorname{Shi}-\cosh\operatorname{Chi}\right) \left(z e^{\frac{\pi i k}{N}} \right)+ \frac{\pi}{2N}\sum_{j=1}^{N-1} \frac{(-1)^j}{\sin \left(\frac{\pi j}{N} \right)} \sum_{h=0}^\infty \frac{z^{2Nh+2j}}{(2Nh+2j)!}.
			\end{align}
			\item[{\rm(b)}]
			For $N$ even,
			\begin{align}\label{Generalized Dixon-Ferror identity even}
				\sum_{h=0}^{\infty} \frac{\left(\psi(2Nh+1)-\log z-\frac{i\pi}{2N}\right)}{\Gamma(2Nh+1)}z^{2Nh} &= \frac{1}{N}\sum_{k=-\frac{N}{2}+1}^{\frac{N}{2}}  \left(\sinh\operatorname{Shi}-\cosh\operatorname{Chi}\right) \left(z e^{\frac{\pi i k}{N}} \right)\nonumber\\
				&+  \frac{\pi}{2N}\sum_{j=1}^{N-1} \frac{(-1)^{j+\frac{j}{N}}}{\sin \left(\frac{\pi j}{N} \right)} \sum_{h=0}^{\infty} \frac{z^{2Nh+2j}}{(2Nh+2j)!}.
			\end{align}
		\end{itemize}
	\end{proposition}
	\begin{proof}
		We first prove (a) where $N$ is odd. Observe that
		\begin{equation}\label{The main series}
			\sum_{h=0}^{\infty} \frac{\psi(2Nh+1)}{\Gamma(2Nh+1)}z^{2Nh} = \sum_{\substack{h=0\\N\mid h}}^{\infty} \frac{\psi(2h+1)}{\Gamma(2h+1)}z^{2h}.
		\end{equation}
	Also,
	\begin{align*}
		\sum_{k=-\frac{(N-1)}{2}}^{\frac{N-1}{2}}e^{\frac{2\pi i kh}{N}}=\sum_{k=0}^{\frac{N-1}{2}}e^{\frac{2\pi i kh}{N}}+\sum_{k=-\frac{(N-1)}{2}}^{-1}e^{\frac{2\pi i kh}{N}}=\sum_{k=0}^{\frac{N-1}{2}}e^{\frac{2\pi i kh}{N}}+\sum_{\ell=\frac{N+1}{2}}^{N-1}e^{\frac{2\pi i (\ell-N)h}{N}}=\sum_{k=0}^{N-1}e^{\frac{2\pi i kh}{N}},
	\end{align*}
		so that from the orthogonality of characters (cf. \cite[p.~158, Theorem 8.1]{Apostol}),
			\begin{equation}\label{Orthogonal character0}
			\sum_{k=-\frac{(N-1)}{2}}^{\frac{N-1}{2}} e^{\frac{2 \pi i k h}{N}} = \begin{cases}
				N & \text{ for } N\mid h\\
				0 &  \text{ for } N\nmid h.
			\end{cases}
		\end{equation}
	Hence from \eqref{The main series} and \eqref{Orthogonal character0},
		\begin{align*}
			\sum_{h=0}^{\infty} \frac{\psi(2Nh+1)}{\Gamma(2Nh+1)}z^{2Nh} = \frac{1}{N} \sum_{h=0}^{\infty} \frac{\psi(2h+1)}{\Gamma(2h+1)}z^{2h} \sum_{k=-\frac{(N-1)}{2}}^{\frac{N-1}{2}} e^{\frac{2\pi i kh}{N}} = \frac{1}{N}\sum_{k=-\frac{(N-1)}{2}}^{\frac{N-1}{2}} \sum_{h=0}^{\infty} \frac{\psi(2h+1)}{\Gamma(2h+1)}\left(ze^{\frac{\pi i k}{N}}\right)^{2h}.
		\end{align*} 
	 Since $|\arg(ze^{\frac{\pi i k}{N}})|<\pi$,	invoking \eqref{2equiv1}, we have
		\begin{align}\label{Generalized Dixon-Ferror LHS}
			\sum_{h=0}^{\infty} \frac{\psi(2Nh+1)}{\Gamma(2Nh+1)}z^{2Nh} &= \frac{1}{N}\sum_{k=-\frac{(N-1)}{2}}^{\frac{N-1}{2}} \left(\sinh\operatorname{Shi}-\cosh\operatorname{Chi}\right) \left(z e^{\frac{\pi i k}{N}} \right)+\frac{1}{N}\sum_{k=-\frac{(N-1)}{2}}^{\frac{N-1}{2}} \log \left(z e^{\frac{\pi i k}{N}} \right)\cosh\left(z e^{\frac{\pi i k}{N}} \right).
		%	&:= \frac{1}{N}(A_1 + A_2) 
		\end{align}
	Let $A$ denote the second finite sum in the above equation. 
		The series expansion of $\cosh\left(z e^{\frac{\pi i k}{N}} \right)$ yields
		\begin{align}\label{Sum T2}
			A &= \sum_{k=-\frac{(N-1)}{2}}^{\frac{N-1}{2}} \log \left(z e^{\frac{\pi i k}{N}} \right) \sum_{\ell=0}^{\infty} \frac{\left(z e^{\frac{\pi i k}{N}} \right)^{2\ell}}{(2\ell)!} \nonumber\\
			&= \sum_{\ell=0}^{\infty} \frac{z^{2\ell}}{(2\ell)!} \sum_{k=-\frac{(N-1)}{2}}^{\frac{N-1}{2}} e^{\frac{2 \pi i k \ell}{N}} \log \left(z e^{\frac{\pi i k}{N}} \right) \nonumber\\
			&= \log z \sum_{\ell=0}^{\infty} \frac{z^{2\ell}}{(2\ell)!} \sum_{k=-\frac{(N-1)}{2}}^{\frac{N-1}{2}} e^{\frac{2 \pi i k \ell}{N}} + \frac{i\pi}{N} \sum_{\ell=0}^{\infty} \frac{z^{2\ell}}{(2\ell)!} \sum_{k=-\frac{(N-1)}{2}}^{\frac{N-1}{2}} k e^{\frac{2 \pi i k \ell}{N}}
		\end{align}
	Similar to \eqref{Orthogonal character0}, one can derive
		\begin{equation}\label{Orthogonal character}
			\sum_{k=-\frac{(N-1)}{2}}^{\frac{N-1}{2}}k e^{\frac{2 \pi i k \ell}{N}} = \begin{cases}
				0 & \text{ for } N\mid \ell\\
				\frac{(-1)^{\ell} N}{e^{\frac{\pi i \ell}{N}}-e^{-\frac{\pi i \ell}{N}}}  &  \text{ for } N\nmid \ell.
			\end{cases}
		\end{equation}
		Therefore, inserting \eqref{Orthogonal character0} and \eqref{Orthogonal character} into \eqref{Sum T2}, we obtain
		\begin{align}\label{Evaluation of T2}
			A&= N \log z \sum_{\substack{\ell=0\\N\mid \ell}}^{\infty} \frac{z^{2\ell}}{(2\ell)!} + i\pi \sum_{\substack{\ell=0\\N\nmid \ell}}^{\infty} \frac{z^{2\ell}}{(2\ell)!} \frac{(-1)^{\ell}}{\left(e^{\frac{\pi i \ell}{N}}-e^{-\frac{\pi i \ell}{N}}\right)}\nonumber\\
			&= N \log z \sum_{h=0}^{\infty} \frac{z^{2Nh}}{(2Nh)!} + i\pi \sum_{j=1}^{N-1}\sum_{h=0}^{\infty} \frac{z^{2Nh+2j}}{(2Nh+2j)!} \frac{(-1)^{Nh+j}}{\left(e^{\frac{\pi i (Nh+j)}{N}}-e^{-\frac{\pi i (Nh+j)}{N}}\right)}\nonumber\\
		%	&= N \log z \sum_{h=0}^{\infty} \frac{z^{2Nh}}{(2Nh)!} + i\pi \sum_{j=1}^{N-1} \frac{(-1)^j}{\left(e^{\frac{\pi i j}{N}}-e^{-\frac{\pi i j}{N}}\right)} \sum_{h=0}^{\infty} \frac{z^{2Nh+2j}}{(2Nh+2j)!} \nonumber\\
			&= N \log z \sum_{h=0}^{\infty} \frac{z^{2Nh}}{(2Nh)!} + \frac{\pi}{2} \sum_{j=1}^{N-1} \frac{(-1)^j}{\sin\left(\frac{\pi j}{N} \right)} \sum_{h=0}^{\infty} \frac{z^{2Nh+2j}}{(2Nh+2j)!},
		\end{align}
		where in the second step , we let $\ell=Nh$ in the first sum and $\ell=Nh+j, $ $1\leq j\leq N-1$ in the second sum . Finally we substitute \eqref{Evaluation of T2} in \eqref{Generalized Dixon-Ferror LHS} to arrive at \eqref{Generalized Dixon-Ferror identity odd}.
		
		We now prove (b) where $N$ is even. Applying \cite[p.~158, Theorem 8.1]{Apostol}, that is,
			\begin{equation*}
			\sum_{k=-\frac{N}{2}+1}^{\frac{N}{2}} e^{\frac{2 \pi i k h}{N}} = \begin{cases}
				N & \text{ for } N\mid h\\
				0 &  \text{ for } N\nmid h,
			\end{cases}
		\end{equation*}
	we have
		\begin{align*}
			\sum_{h=0}^{\infty} \frac{\psi(2Nh+1)}{\Gamma(2Nh+1)}z^{2Nh} = \frac{1}{N} \sum_{h=0}^{\infty} \frac{\psi(2h+1)}{\Gamma(2h+1)}z^{2h} \sum_{k=-\frac{N}{2}+1}^{\frac{N}{2}} e^{\frac{2\pi i kh}{N}} = \frac{1}{N}\sum_{k=-\frac{N}{2}+1}^{\frac{N}{2}}  \sum_{h=0}^{\infty} \frac{\psi(2h+1)}{\Gamma(2h+1)}(ze^{\frac{\pi i k}{N}})^{2h}.
		\end{align*}
		Substitute \eqref{2equiv1} in the above equation to get
		\begin{align}\label{Generalized Dixon-Ferror LHS even}
			\sum_{h=0}^{\infty} \frac{\psi(2Nh+1)}{\Gamma(2Nh+1)}z^{2Nh} &= \frac{1}{N}\sum_{k=-\frac{N}{2}+1}^{\frac{N}{2}} \left(\sinh\operatorname{Shi}-\cosh\operatorname{Chi}\right) \left(z e^{\frac{\pi i k}{N}} \right)+\frac{1}{N}\sum_{k=-\frac{N}{2}+1}^{\frac{N}{2}} \log \left(z e^{\frac{\pi i k}{N}} \right)\cosh\left(z e^{\frac{\pi i k}{N}} \right).
		\end{align}
		Let $B$ denote the second finite sum in the above equation. From the Taylor expansion of $\cosh\left(z e^{\frac{\pi i k}{N}} \right)$, we have
		\begin{align}\label{Sum T2 even}
			B = \sum_{k=-\frac{N}{2}+1}^{\frac{N}{2}} \log \left(z e^{\frac{\pi i k}{N}} \right) \sum_{\ell=0}^{\infty} \frac{\left(z e^{\frac{\pi i k}{N}} \right)^{2\ell}}{(2\ell)!} 
%			&= \sum_{\ell=0}^{\infty} \frac{z^{2\ell}}{(2\ell)!} \sum_{k=-\frac{N}{2}+1}^{\frac{N}{2}} e^{\frac{2 \pi i k \ell}{N}} \log \left(z e^{\frac{\pi i k}{N}} \right) \nonumber\\
			&= \log z \sum_{\ell=0}^{\infty} \frac{z^{2\ell}}{(2\ell)!} \sum_{k=-\frac{N}{2}+1}^{\frac{N}{2}} e^{\frac{2 \pi i k \ell}{N}} + \frac{i\pi}{N} \sum_{\ell=0}^{\infty} \frac{z^{2\ell}}{(2\ell)!} \sum_{k=-\frac{N}{2}+1}^{\frac{N}{2}} k e^{\frac{2 \pi i k \ell}{N}}.
		\end{align}
		Again, it is not difficult to see that
		\begin{equation}\label{Orthogonal character even}
			\sum_{k=-\frac{N}{2}+1}^{\frac{N}{2}}k e^{\frac{2 \pi i k \ell}{N}} = \begin{cases}
				\frac{N}{2} & \text{ for } N\mid \ell\\
				\frac{(-1)^{\ell}N}{1-e^{-\frac{2\pi i \ell}{N}}}  &  \text{ for } N\nmid \ell.
			\end{cases}
		\end{equation}
		Hence from \eqref{Sum T2 even} and \eqref{Orthogonal character even},
		\begin{align}\label{Evaluation of T2 even}
			B &= N \log z \sum_{\substack{\ell=0\\N\mid \ell}}^{\infty} \frac{z^{2\ell}}{(2\ell)!} + i\pi\left[ \frac{1}{2}\sum_{\substack{\ell=0\\N\mid \ell}}^{\infty} \frac{z^{2\ell}}{(2\ell)!}  +\sum_{\substack{\ell=0\\N\nmid \ell}}^{\infty} \frac{z^{2\ell}}{(2\ell)!} \frac{(-1)^{\ell}}{1-e^{-\frac{2\pi i \ell}{N}}}\right]\nonumber\\
			&= \left(N \log z+\frac{i\pi}{2}\right) \sum_{h=0}^{\infty} \frac{z^{2Nh}}{(2Nh)!} + i\pi \sum_{j=1}^{N-1}\sum_{h=0}^{\infty} \frac{z^{2Nh+2j}}{(2Nh+2j)!} \frac{(-1)^j}{1-e^{-\frac{2\pi i j}{N}}}\nonumber\\
		%	&= \left(N \log z+\frac{i\pi}{2}\right) \sum_{h=0}^{\infty} \frac{z^{2Nh}}{(2Nh)!} + i\pi \sum_{j=1}^{N-1}\frac{(-1)^j}{1-e^{-\frac{2\pi i j}{N}}} \sum_{h=0}^{\infty} \frac{z^{2Nh+2j}}{(2Nh+2j)!}\nonumber\\
			&=\left(N \log z+\frac{i\pi}{2}\right) \sum_{h=0}^{\infty} \frac{z^{2Nh}}{(2Nh)!} + \frac{\pi}{2}\sum_{j=1}^{N-1} \frac{(-1)^{j+\frac{j}{N}}}{\sin \left(\frac{\pi j}{N} \right)} \sum_{h=0}^{\infty} \frac{z^{2Nh+2j}}{(2Nh+2j)!}.
		\end{align}
	%	where in the second step in the first and the second sum we made the change of variable $\ell$ by $Nh$ and in the third sum $\ell$ by $Nh+j$ with $0\leq j\leq N-1$.
		 Substituting \eqref{Evaluation of T2 even} in \eqref{Generalized Dixon-Ferror LHS even} yields \eqref{Generalized Dixon-Ferror identity even} upon simplification.
	\end{proof}
	
	\section{Special cases of Theorem \ref{Analytic continuation} for $N$ odd}\label{scaco}
	
	We begin this section with the special case $a=-2Nm-N$ of Theorem \ref{Analytic continuation}.

	\subsection{A new generalization of Ramanujan's formula for $\zeta(2m+1)$}\label{a=-2Nm-N}

	\begin{proof}[Theorem \textup{\ref{zetagen3}}][]
	We use without mention the fact that $N$ is an odd positive integer. Letting $a\to-2Nm-N,$ where $ m\in\mathbb{Z}\backslash\{0\},$ in Theorem \ref{Analytic continuation}, we get
	\begin{align}\label{Eqn:Analytic continuationnew3}
		&\sum_{n=1}^\infty \sigma_{-2Nm-N}^{(N)}(n) e^{-ny}  + \frac{1}{2}\zeta(2Nm+N) - \frac{1}{y}\zeta(2Nm+2N) - \frac{y^{2m+1-{1\over N}}}{N} \Gamma \left(-2m-1+{1\over N}\right) \zeta \left(-2m-1+{1\over N}\right)\nonumber\\
		&= \tfrac{2 (2\pi)^{\frac{1}{N}-\frac{1}{2}} N^{\frac{-2Nm-N-1}{2}}}{y^{\frac{1}{N}-m-\frac{1}{2}}} \sum_{n=1}^\infty \tfrac{S_{-2Nm-N}^{(N)}(n)}{n^{-m-\frac{1}{2}}} \bigg[{}_{\frac{1}{2}} K_{-m-\frac{1}{2}}^{(N)}\left(\substack{\tfrac{4\pi^{N+1}n}{yN^N}, 0}\right) - \tfrac{2^{\frac{1}{2} + \frac{1-2Nm-N}{N}} \pi^{{(1-N)(-m-\frac{1}{2})} -N}} {2^{N-1}\left(\tfrac{4\pi^{N+1}n}{yN^N}\right)^{\frac{1}{N}-m+\frac{1}{2}}} \nonumber \\
		&\quad\times \lim\limits_{a\to -2Nm-N}{\sin\(\frac{\pi}{2}(N-a)\)}C_{m,N}\left(\substack{\tfrac{1}{2}, \tfrac{a}{2N}, 0, \tfrac{4\pi^{N+1}n}{yN^N}}\right)\bigg]+  \frac{y}{2\pi^{2}} \sum_{k=0}^m \left(-\frac{y^2}{4\pi^2}\right)^k   \zeta(2Nm-2kN)\zeta(2k+2).
	\end{align}
	
	Let $Z= \frac{4\pi^{2N+2}n^2}{y^2N^{2N}}=\frac{1}{4}\left(\frac{4\pi^{N+1}n}{yN^N}\right)^2$. We first obtain an elementary expression for ${}_{\frac{1}{2}} K_{-m-\frac{1}{2}}^{(N)}\left(\substack{\tfrac{4\pi^{N+1}n}{yN^N}, 0}\right)$. To that end, we employ \eqref{muknug} in the first step below and then invoke Theorem \ref{meijergsim} in the second step to obtain
	
	\begin{align*}
		{}_{\frac{1}{2}} K_{-m-\frac{1}{2}}^{(N)}\left(\substack{\tfrac{4\pi^{N+1}n}{yN^N}, 0}\right)
		%&- \frac{2^{\frac{1}{2} + \frac{a+1}{N}} \pi^{\frac{(1-N)a}{2N} -N}} {\left(\frac{4\pi^{N+1}n}{yN^N}\right)^{1+\frac{1}{N}+\frac{a}{2N}}}   \frac{\sin\(\frac{\pi}{2}(N-a)\right)}{2^{N-1}}C_{m,N}\(\frac{1}{2}, \frac{a}{2N}, 0, \frac{4\pi^{N+1}n}{yN^N}\right)
		&=\frac{(2\pi)^{-2m-1-2/N}}{\pi^2 \sqrt{2}}\left(\frac{n}{yN^N}\right)^{{-m-\frac{1}{2}}-{2\over N}} G_{1, \, \, 2N+1}^{N+1, \, \, 1} \left(\begin{matrix}
			\frac{1}{2} + \frac{1+2Nm+N}{2N}\\
			\langle \frac{i}{N}\rangle, \frac{1}{2} + \frac{1+2Nm+N}{2N}; \langle 1+ \frac{3}{2N}-{i\over N} \rangle
		\end{matrix} \Bigg | Z \right)\nonumber\\
		&=\frac{(2\pi)^{-2m-1-2/N}}{\pi^2 \sqrt{2}}\left(\frac{n}{yN^N}\right)^{{-m-\frac{1}{2}}-{2\over N}} \Bigg\{\frac{ N^{2Nm+2N - \frac{1}{2}} \Gamma\left(\frac{1-2Nm-2N}{2}\right)Z^{1+m+\frac{1}{2N}} }{\Gamma\left(Nm+N\right)}\nonumber \\
		&\qquad\qquad\qquad\qquad\qquad\qquad \times		{}_1F_{2N}\left( \left. \begin{matrix}
			1  \\ \left\langle{1}+m+{i-1 \over 2N}\right\rangle_{i=1}^{ 2N}
		\end{matrix}\right| Z\right)+ A_{-2Nm-N, N}\left(Z\right)\Bigg\},
	\end{align*}
	where $A_{a, N}(z)$ is defined in \eqref{aanzodd}. 
	
	Employing \eqref{prudnikovid}, we simplify the resulting ${}_1F_{2N}$ in it as
	\begin{align*}
		{}_1F_{2N}\left( \left. \begin{matrix}
			1  \\ \left\langle{1}+m+{i-1 \over 2N}\right\rangle_{i=1}^{ 2N}
		\end{matrix}\right| Z\right)
		&=\frac{(2N(m+1)-1)!}{(2N)^{2N(m+1)}Z^{m+1-{1\over 2N}} }\Bigg\{ \sum_{k=0}^{2N-1}\frac{\exp \left(2 N e^{\pi i k \over N}Z^{1\over 2N}\right) }{e^{\pi i k (2m+2-{1\over N})}}\\
		&\quad -(2N)^{2N(m+1)}\sum_{k=1}^{m}\frac{Z ^{m+1-{1\over 2N}-k}}{(2N)^{2Nk}(2N(m+1)-1-2Nk)!}\Bigg\},
	\end{align*}
	thereby obtaining
	\begin{align}\label{ksimp}
		{}_{\frac{1}{2}} K_{-m-\frac{1}{2}}^{(N)}\left(\tfrac{4\pi^{N+1}n}{yN^N}, 0\right) 
%		&=\frac{\left(\frac{n}{yN^{N}}\right)^{-m-\frac{1}{2}-{2\over N}}}{\pi^2 \sqrt{2}(2\pi)^{2m+1+\frac{2}{N}}}\left[A_{-2Nm-N, N}\left(\frac{4\pi^{2N+2}n^2}{y^2N^{2N}}\right)+\frac{ N^{N+2Nm+N - \frac{1}{2}} \Gamma\left(\frac{1}{2}-Nm-N\right)\left(\frac{4\pi^{2N+2}n^2}{y^2N^{2N}}\right)^{1+\frac{1}{2N}+m} }{\Gamma\left({Nm+N}\right)(2N)^{2N(m+1)}}\right.\nonumber\\
%		&\times\frac{(2N(m+1)-1)!}{ \left(\frac{4\pi^{2N+2}n^2}{y^2N^{2N}} \right)^{m+1-{1\over 2N}} }\left.\Bigg\{ \sum_{k=0}^{2N-1}\frac{\exp \left(2 N e^{\pi i k \over N}\left(\frac{4\pi^{2N+2}n^2}{y^2N^{2N}} \right)^{1\over 2N}\right) }{e^{\pi i k (2m+2-{1\over N})}} -\sum_{k=1}^{m}\frac{(2N)^{2N(m+1)}\left(\frac{4\pi^{2N+2}n^2}{y^2N^{2N}} \right)^{m+1-{1\over 2N}-k}}{(2N)^{2Nk}(2N(m+1)-1-2Nk)!}\Bigg\}\right]\nonumber \\
		&=\frac{\left(\frac{n}{yN^{N}}\right)^{-m-\frac{1}{2}-{2\over N}}}{\pi^2 \sqrt{2}(2\pi)^{2m+1+\frac{2}{N}}}\left[A_{-2Nm-N, N}\left(Z\right)+\frac{(-1)^{m+1}\sqrt{\pi}}{2\sqrt{N}}Z^{{1\over N}}\right.\nonumber \\
		&\quad\times\left.\Bigg\{ \sum_{k=0}^{2N-1}\frac{\exp \left(2 N e^{\pi i k \over N}Z^{1\over 2N}\right) }{e^{\pi i k (2m+2-{1\over N})}} -\sum_{k=1}^{m}\frac{(2N)^{2N(m+1)}Z^{m+1-{1\over 2N}-k}}{(2N)^{2Nk}(2N(m+1)-1-2Nk)!}\Bigg\}\right].
	\end{align}
	%\frac{(2\pi)^{a-2\over N}}{\pi^2 \sqrt{2}}\frac{(n/y)^{{a\over 2N}-{2\over N}}}{N^{{a\over 2}-2}}
	We next consider the first two expressions in the square bracket of \eqref{ksimp}. Using \eqref{aanzodd}, we observe that
	\begin{align*}
		&A_{-2Nm-N, N}\left(Z\right)+\frac{(-1)^{m+1}\sqrt{\pi}}{2\sqrt{N}}\left(Z\right)^{{1\over N}} \sum_{k=0}^{2N-1}\frac{\exp \left(2 N e^{\pi i k \over N}Z^{1\over 2N}\right) }{e^{\pi i k (2m+2-{1\over N})}}\nonumber \\
		&=\frac{(-1)^{m+1}\sqrt{\pi}}{2\sqrt{N}}Z^{{1\over N}}\left\{-\left[\sum_{k=0}^{\frac{N-1}{2}}-\sum_{k=\frac{N+1}{2}}^{\frac{3N-1}{2}}+\sum_{k=\frac{3N-1}{2}}^{2N-1}\right]{e^{\pi i k \over N}}{\exp \left(2 N e^{\pi i k \over N}Z^{1\over 2N}\right)}+\sum_{k=0}^{2N-1}\frac{\exp \left(2 N e^{\pi i k \over N}Z^{1\over 2N}\right) }{e^{\pi i k (2m+2-{1\over N})}}\right\}\nonumber\\
			&=\frac{(-1)^{m+1}\sqrt{\pi}}{\sqrt{N}}Z^{{1\over N}}\sum_{k=\frac{N+1}{2}}^{\frac{3N-1}{2}}{e^{\pi i k \over N}}{\exp \left(2 N e^{\pi i k \over N}Z^{1\over 2N}\right) }.
		\end{align*}

	Substituting the above expression in \eqref{ksimp}, we find that
	\begin{align}\label{beforecancel}
		{}_{\frac{1}{2}} K_{-m-\frac{1}{2}}^{(N)}\left(\tfrac{4\pi^{N+1}n}{yN^N}, 0\right) 
		&=(-1)^{m+1}\frac{N^{Nm+{N\over 2}+{3\over 2}}Z^{1\over N}}{\pi(2\pi)^{2m+{3\over 2}+{2\over N}}(n/y)^{m+{1\over 2}+{2\over N}}}\sum_{k=\frac{N+1}{2}}^{\frac{3N-1}{2}}{e^{\pi i k \over N}}{\exp \left(2 N e^{\pi i k \over N}Z^{1\over 2N}\right) }\nonumber\\
		&\quad+(-1)^{m}\frac{2^{2Nm+2N}N^{3Nm+{5N\over 2}+{3\over 2}}Z^{{1\over N}+m+1-{1\over 2N}}}{(2\pi)^{2m+{5\over 2}+{2\over N}}(n/y)^{m+{1\over 2}+{2\over N}}}\sum_{k=1}^{m}\frac{\left(\frac{y^2}{(2\pi)^{2N+2}n^2}\right)^{k}}{(2Nm-2Nk+2N-1)!}.
	\end{align}
	%\begin{discussion}
	%	[MENTION THAT WE GET LEMMA 10.1 OF DKK FROM THE ABOVE RESULT BY LETTING $N=1$.]
	%\end{discussion}
 Now let $L$ denote the limit in the second expression inside the square brackets on the right-hand side of \eqref{Eqn:Analytic continuationnew3}. 
% upon letting $a\to-2Nm-N$. This gives
%	\begin{align}
%		&\lim_{a\to-2Nm-N}\frac{2^{\frac{1}{2} + \frac{a+1}{N}} \pi^{\frac{(1-N)a}{2N} -N}} {\left(\frac{4\pi^{N+1}n}{yN^N}\right)^{1+\frac{1}{N}+\frac{a}{2N}}}   \frac{\sin\(\frac{\pi}{2}(N-a)\)}{2^{N-1}}C_{m,N}\left(\frac{1}{2}, \frac{a}{2N}, 0, \frac{4\pi^{N+1}n}{yN^N}\right)\nonumber\\
%		&=2^{-2m-N+\frac{1}{2}+{1\over N}}\pi^{\frac{(N-1)(2m+1)}{2}-N}\left(\frac{4\pi^{N+1}n}{yN^N}\right)^{m-{1\over 2}-\frac{1}{N}}\cdot L,
%	\end{align}
%	where
%	\begin{align}
%		L:=\lim_{a\to-2Nm-N}\sin\left(\frac{\pi}{2}(N-a)\right)C_{m,N}\left(\frac{1}{2}, \frac{a}{2N}, 0, \frac{4\pi^{N+1}n}{yN^N}\right).
%	\end{align}
We evaluate this limit. Using the fact $\sin\left(\frac{\pi}{2}(N-a)\right)=(-1)^{\frac{N-1}{2}}\cos\left(\frac{\pi a}{2}\right)$, \eqref{cmn} and \eqref{gmf} in the first step below, and the fact $\cos\left(\frac{\pi a}{2}\right)=(-1)^{k+\frac{N+1}{2}}\cos\left(\frac{\pi }{2}(N+a+1+2Nk)\right)$ in the second step below, we have
	\begin{align*}
		L&=    (-1)^{\frac{N+1}{2}}\left(2\pi\right)^{N-\frac{1}{2}}\sum_{k=0}^{m}{(2N)^{2Nm-2Nk-{1\over 2}}}\left(
		\frac{2\pi^{N+1}n}{yN^{N}}\right)^{-2k}\lim_{a \to -2Nm-N}{\cos\left(\frac{\pi a}{2}\right)\Gamma(N+a+1+2Nk)} \nonumber\\
		&= \left(2\pi\right)^{N-\frac{1}{2}} \sum_{k=0}^{m-1}\frac{(-1)^{k}}{(2N)^{2Nk-2Nm+{1\over 2}}}\left(
		\frac{2\pi^{N+1}n}{yN^{N}}\right)^{-2k}\nonumber\\
		&\quad\times\lim_{N+a+1+2Nk\to-(2\ell+1)}\cos\left(\frac{\pi}{2}(N+a+1+2Nk)\right)\Gamma(N+a+1+2Nk)
	\end{align*}
	since the $k=m$ term vanishes. Here, $\ell=Nm-Nk-1$. Next, we employ \eqref{usefullim} with $a$ replaced by $N+a+1+2Nk$ in the above equation and then replace $k$ by $k-1$ so as to arrive at
	\begin{align*}
		L=  (-1)^{m}\left(2\pi\right)^{N-\frac{1}{2}} \left(\frac{\pi}{2}\right)(2N)^{2Nm+2N-{1\over 2}}\left(\frac{2\pi^{N+1}n}{yN^N}\right)^{2}\sum_{k=1}^{m}\frac{\left(\frac{y^2}{(2\pi)^{2N+2}n^2}\right)^{k}}{(2Nm-2Nk+2N-1)!},
	\end{align*}
	whence, upon simplification, we have
\begin{align}\label{cancel}
	&\lim_{a\to-2Nm-N}\frac{2^{\frac{1}{2} + \frac{a+1}{N}} \pi^{\frac{(1-N)a}{2N} -N}} {\left(\frac{4\pi^{N+1}n}{yN^N}\right)^{1+\frac{1}{N}+\frac{a}{2N}}}   \frac{\sin\(\frac{\pi}{2}(N-a)\)}{2^{N-1}}C_{m,N}\left(\frac{1}{2}, \frac{a}{2N}, 0, \frac{4\pi^{N+1}n}{yN^N}\right)\nonumber\\    
	&=(-1)^{m}\pi(2\pi)^{2Nm+2N-{1\over N}-{1\over 2}}\left({n\over y}\right)^{m-{1\over N}+{3\over 2}}N^{Nm+{N\over 2} +{1\over 2}}\sum_{k=1}^{m}\frac{\left(\frac{y^2}{(2\pi)^{2N+2}n^2}\right)^{k}}{(2Nm-2Nk+2N-1)!}.
\end{align}
	Note that the second expression in \eqref{beforecancel} completely cancels out with the expression in \eqref{cancel}. Therefore, the expression in square brackets on the right-hand side of \eqref{Eqn:Analytic continuationnew3} simplifies in the limit to
	\begin{align*}
		&\lim_{a\to-2Nm-N}\left[{}_{\frac{1}{2}} K_{\frac{a}{2N}}^{(N)}\left(\tfrac{4\pi^{N+1}n}{yN^N}, 0\right) - \frac{2^{\frac{1}{2} + \frac{a+1}{N}} \pi^{\frac{(1-N)a}{2N} -N}} {\left(\frac{4\pi^{N+1}n}{yN^N}\right)^{1+\frac{1}{N}+\frac{a}{2N}}}   \frac{\sin\(\frac{\pi}{2}(N-a)\)}{2^{N-1}}C_{m,N}\left(\frac{1}{2}, \frac{a}{2N}, 0, \frac{4\pi^{N+1}n}{yN^N}\right)\right]\nonumber\\
		&\qquad =(-1)^{m+1}\frac{N^{Nm+{N\over 2}+{3\over 2}}}{\pi(2\pi)^{2m+{3\over 2}+{2\over N}}(n/y)^{m+{1\over 2}+{2\over N}}}\left(\frac{4\pi^{2N+2}n^2}{y^2N^{2N}} \right)^{{1\over N}}\sum_{k=\frac{N+1}{2}}^{\frac{3N-1}{2}}{e^{\pi i k \over N}}{\exp \left(2 N e^{\pi i k \over N}\left(\frac{4\pi^{2N+2}n^2}{y^2N^{2N}} \right)^{1\over 2N}\right) }\nonumber\\
		&\qquad=(-1)^{m+1}\sqrt{{\pi \over 2}}\left(4\pi^2 n\over y \right)^{-m-{1\over 2}} N^{Nm+{N\over 2}-{1\over 2}}\sum_{k=\frac{N+1}{2}}^{\frac{3N-1}{2}}{e^{\pi i k \over N}}{\exp \left(2 \pi e^{\pi i k \over N}\left(\frac{2\pi n}{y} \right)^{1\over N}\right) }.
	\end{align*}
	Using the above equation along with the definition of $S_{a}^{(N)}(n)$ in \eqref{defbf}, and interchanging the order of limit and summation, which is justified since the summand of the series on the right hand side of \eqref{Eqn:Analytic continuation} is $\mathcal{O}\left(n^{\max \left(-\frac{\re(a)}{2N}, \frac{\re(a)}{2N} + \frac{1}{N} - 1 \right)-2m-3-\frac{1}{N}-\frac{\re(a)}{2N}+\epsilon} \right)$ for any $\epsilon>0$, we get
	\begin{align}\label{fourthterm}
		&\tfrac{2 (2\pi)^{\frac{1}{N}-\frac{1}{2}} N^{\frac{-2Nm-N-1}{2}}}{y^{\frac{1}{N}-m-\frac{1}{2}}} \sum_{n=1}^\infty \tfrac{S_{-2Nm-N}^{(N)}(n)}{n^{-m-\frac{1}{2}}} \bigg[{}_{\frac{1}{2}} K_{-m-\frac{1}{2}}^{(N)}\left(\substack{\tfrac{4\pi^{N+1}n}{yN^N}, 0}\right) - \tfrac{2^{\frac{1}{2} + \frac{1-2Nm-N}{2N}} \pi^{{(1-N)(-m-\frac{1}{2})} -N}} {2^{N-1}\left(\tfrac{4\pi^{N+1}n}{yN^N}\right)^{\frac{1}{N}-m+\frac{1}{2}}} \nonumber \\
		&\quad\times \lim\limits_{a\to -2Nm-N}{\sin\(\frac{\pi}{2}(N-a)\)}C_{m,N}\left(\substack{\tfrac{1}{2}, \tfrac{a}{2N}, 0, \tfrac{4\pi^{N+1}n}{yN^N}}\right)\bigg]\nonumber\\
		%&\lim_{a\to-2Nm-1}\frac{2 (2\pi)^{\frac{1}{N}-\frac{1}{2}} N^{\frac{a-1}{2}}}{y^{\frac{1}{N}+\frac{a}{2N}}}\sum_{n=1}^\infty \frac{S_a^{(N)}(n)}{n^{\frac{a}{2N}}} \bigg [{}_{\frac{1}{2}} K_{\frac{a}{2N}}^{(N)}\left(\tfrac{4\pi^{N+1}n}{yN^N}, 0\right) \nonumber\\
		%&\hspace{6.8cm}- \frac{2^{\frac{1}{2} + \frac{a+1}{N}} \pi^{\frac{(1-N)a}{2N} -N}} {\left(\frac{4\pi^{N+1}n}{yN^N}\right)^{1+\frac{1}{N}+\frac{a}{2N}}}   \frac{\sin\(\frac{\pi}{2}(N-a)\)}{2^{N-1}}C_{m,N}\left(\frac{1}{2}, \frac{a}{2N}, 0, \frac{4\pi^{N+1}n}{yN^N}\right)\bigg]\nonumber\\
		&={(-1)^{m+1}\over N}\left( {y\over 2\pi}\right)^{2m-{1\over N}+1}\sum_{d_1, d_2=1}^\infty   {d_2^{-2m-2+{1\over N}}} \sum_{k={N+1\over 2}}^{k={3N-1\over 2}}{e^{\pi i k \over N}} \exp  \left\lbrace 2 \pi d_1 e^{\pi i k \over N}\left(\frac{2\pi d_2}{y} \right)^{1\over N}\right\rbrace \nonumber\\
		&={(-1)^{m}\over N}\left( {y\over 2\pi}\right)^{2m-{1\over N}+1}\sum_{d_2=1}^\infty   {d_2^{-2m-2+{1\over N}}} \sum_{j=-{(N-1)\over 2}}^{{N-1\over 2}}{e^{\pi i j \over N}} \sum_{d_1=1}^{\infty}\left\lbrace \exp \left( - 2 \pi   e^{\pi i j \over N}\left(\frac{2\pi d_2}{y} \right)^{1\over N}\right)\right\rbrace^{d_1} \nonumber\\
		&={(-1)^{m}\over N}\left( {y\over 2\pi}\right)^{2m-{1\over N}+1}\sum_{j=-{(N-1)\over 2}}^{{N-1\over 2}}{e^{\pi i j \over N}}\sum_{n=1}^\infty    \frac{{n^{-2m-2+{1\over N}}}}{\exp  \left\lbrace 2 \pi  e^{\pi i j \over N}\left(\frac{2\pi n}{y} \right)^{1\over N}\right\rbrace -1},
	\end{align}
	where in the penultimate step, we replaced $k$ by $j+N$, and in the last step we summed the inner series using the geometric series formula since ${-(N-1)\over 2}<j<{(N-1)\over 2}$ implies Re$\left( 2 \pi d_1 e^{\pi i j \over N}\left(\frac{2\pi d_2}{y} \right)^{1\over N}\right) =2 \pi d_1 \cos \left( {\pi  j \over N}\right) \left(\frac{2\pi d_2}{y} \right)^{1\over N}>0$. Finally, replacing $k$ by $k-1$ in the last expression on the right-hand side of \eqref{Eqn:Analytic continuation},
	\begin{align}\label{fifthterm}
	%	&\lim_{a\to-2Nm-N}\frac{y}{2\pi^{2}} \sum_{k=0}^m \left(-\frac{y^2}{4\pi^2}\right)^k   \zeta(-2kN-N-a)\zeta(2k+2)\nonumber\\
		\frac{y}{2\pi ^2}\sum_{k=0}^{m} (-1)^k \zeta(2Nm-2Nk)\zeta (2k+2)\left({2\pi \over y}\right)^{-2k}
		=-\frac{2}{y}\sum_{k=1}^{m+1} (-1)^k \zeta(2N(m-k+1))\zeta (2k)\left({2\pi \over y}\right)^{-2k}.
	\end{align}
	From \eqref{Eqn:Analytic continuationnew3}, \eqref{fourthterm} and \eqref{fifthterm},
	\begin{align*}
		&\sum_{n=1}^\infty \sigma_{-2Nm-N}^{(N)}(n) e^{-ny}  + \frac{1}{2}\zeta(2Nm+N) - \frac{1}{y}\zeta(2Nm+2N) - \frac{y^{2m+1-{1\over N}}}{N} \Gamma \left(-2m-1+{1\over N}\right) \zeta \left(-2m-1+{1\over N}\right)\nonumber\\
		&={(-1)^{m+1}\over N}\left( {y\over 2\pi}\right)^{2m-{1\over N}+1}\sum_{j={-(N-1)\over 2}}^{{N-1\over 2}}{e^{\pi i j \over N}}\sum_{n=1}^\infty    \frac{{n^{-2m-2+{1\over N}}}}{\exp  \left\lbrace 2 \pi  e^{\pi i j \over N}\left(\frac{2\pi n}{y} \right)^{1\over N}\right\rbrace -1}\nonumber\\
		&\quad-\frac{2}{y}\sum_{k=1}^{m+1} (-1)^k \zeta(2N(m-k+1))\zeta (2k)\left({2\pi \over y}\right)^{-2k}.
	\end{align*}
	Using \eqref{zetafe} to evaluate the limit on the left-hand side in the above equation and inducting the term $\frac{1}{y}\zeta(2Nm+2N)$ as the $k=0$ term of the finite sum on the right-hand side, we get 
	\begin{align*}
		&\sum_{n=1}^\infty \sigma_{-2Nm-N}^{(N)}(n) e^{-ny}  + \frac{1}{2}\zeta(2Nm+N) \nonumber\\
		&=\frac{(-1)^{m}}{2N\sin{\left(\pi\over 2N\right)}}\left(\frac{y}{2\pi}\right)^{2m+1-{1\over N}}\zeta\left(2m+2-{1\over N}\right) +{(-1)^{m}\over N}\left( {y\over 2\pi}\right)^{2m-{1\over N}+1}\hspace{-5mm}\sum_{j=-{(N-1)\over 2}}^{{N-1\over 2}}{e^{\pi i j \over N}}\sum_{n=1}^\infty    \frac{{n^{-2m-2+{1\over N}}}}{\exp  \left\lbrace 2 \pi  e^{\pi i j \over N}\left(\frac{2\pi n}{y} \right)^{1\over N}\right\rbrace -1}\nonumber\\
		&\quad-\frac{2}{y}\sum_{k=0}^{m+1} (-1)^k \zeta(2N(m-k+1))\zeta (2k)\left({2\pi \over y}\right)^{-2k}.
	\end{align*}
	Finally, let $\alpha=2^{-N}y$ and $\beta = 2\pi^{1+{1\over N}}y^{-\frac{1}{N}}$ so that $\alpha \beta^N =\pi^{N+1}$, multiply both sides of the above equation by $\a^{-\left( \frac{2Nm+N-1}{N+1}\right)}$ and use Euler's formula \eqref{ef} and \eqref{fkny} to arrive at \eqref{zetageneqn3}. This completes the proof.
	\end{proof}

\begin{remark}
	The identity \eqref{zetageneqn} can be obtained from Theorem \textup{\ref{Analytic continuation}} in a similar way by letting $a\to-2Nm-1$.
\end{remark}

%\hspace{3mm}
	\subsection{An identity involving the extended higher Herglotz functions}\label{idtyherglotz}
	
	Recently, the second author, Gupta and Kumar \cite{hhf1} introduced and studied the extended higher Herglotz function
	\begin{align}\label{fknx}
		\mathscr{F}_{k,N}(x):=\sum_{n=1}^{\infty}\frac{\psi(n^{N}x)-\log(n^{N}x)}{n^k}\hspace{5mm}\left(x\in\mathbb{C}\backslash(-\infty,0]\right),
	\end{align}
	where $k$ and $N$ are positive real numbers such that $k+N>1$. When $k=N=1$, this function, known as the Herglotz function \cite{zagier1975} or sometimes the Herglotz-Zagier function, plays an important role in Zagier's explicit version of the Kronecker limit formula for real quadratic fields. The case $N=1$ of \eqref{fknx} was also studied by Vlasenko and Zagier \cite{vz} and is useful in higher Kronecker ``limit'' formulas. 
	
	The identity in Corollary \ref{gencompanion} below essentially involves the extended higher Herglotz functions. It was first obtained in \cite[Theorem 3.11]{hhf1} by differentiating a two-parameter generalization of Ramanujan's formula for $\zeta(2m+1)$ \cite[Theorem 2.4]{DGKM}. Although differentiation gives the result in Corollary \ref{gencompanion} in a fairly straightforward manner, it must be noted that Theorem 2.4 from \cite{DGKM} itself is not easy to derive and requires quite a few technical results towards its proof.
	
	In what follows, we derive the result in Corollary \ref{gencompanion} as a special case of Theorem \ref{Analytic continuation}. This demonstrates the power of Theorem \ref{Analytic continuation} in that results of varied nature such as Theorem \ref{zetagen3}, Corollary \ref{gencompanion}, Theorem \ref{thm_resultt} etc. all can be brought under one roof.
	
	\begin{corollary}\label{gencompanion}
		Let $N$ be an odd positive integer. For $\a \b^N= \pi^{N+1}$ with $\textup{Re}(\a)>0, ~\textup{Re}(\b) >0$, and $m \geq 1$,
		\begin{align*}
			&\a^{-\left(\frac{2Nm}{N+1}-\frac{1}{2}\right)}\left(\frac{1}{2}\zeta(2Nm+1-N)+\sum_{n=1}^{\infty}\frac{n^{-2Nm-1+N}}{e^{(2n)^N\a}-1} \right)-\sum_{j=0}^{m-1}\frac{B_{2j}~\zeta(2Nm+1-2Nj)}{(2j)!}2^{N(2j-1)}\a^{2j-\frac{2Nm}{N+1}-\frac{1}{2}}\nonumber\\
			&=\frac{2^{2m(N-1)}}{N\pi^{\frac{N+1}{2}}}(-1)^{m+1}\b^{-\left(\frac{2Nm}{N+1}-\frac{N}{2}\right)}\Bigg[\frac{N\gamma}{2^{N-1}}\zeta(2m)+\frac{1}{2^N}\sum_{j=-\frac{(N-1)}{2}}^{\frac{(N-1)}{2}}\sum_{n=1}^{\infty}\frac{1}{n^{2m}}\left(\substack{\psi\left( \frac{i \b}{2\pi}(2n)^{1/N}e^{\frac{i\pi j}{N}}\right)+\psi\left( -\frac{i \b}{2\pi}(2n)^{1/N}e^{\frac{i\pi j}{N}}\right) }\right) \Bigg].
		\end{align*}
	\end{corollary}
	\begin{proof}
		Let $a\to-2Nm-1+N$ in Theorem \ref{Analytic continuation} so that
		\begin{align}\label{Eqn:Analytic continuationherglotz}
			&\sum_{n=1}^\infty \sigma_{-2Nm-1+N}^{(N)}(n) e^{-ny}  + \frac{1}{2}\zeta(2Nm+1-N) - \frac{1}{y}\zeta(2Nm+1) \nonumber\\
			&= \lim_{a\to-2Nm-1+N}\tfrac{2 (2\pi)^{\frac{1}{N}-\frac{1}{2}} N^{\frac{a-1}{2}}}{y^{\frac{1}{N}+\frac{a}{2N}}} \sum_{n=1}^\infty \frac{S_a^{(N)}(n)}{n^{\frac{a}{2N}}} \bigg [{}_{\frac{1}{2}} K_{\frac{a}{2N}}^{(N)}\left(\substack{\frac{4\pi^{N+1}n}{yN^N}, 0}\right) - \frac{2^{\frac{1}{2} + \frac{a+1}{N}} \pi^{\frac{(1-N)a}{2N} -N}} {\left(\frac{4\pi^{N+1}n}{yN^N}\right)^{1+\frac{1}{N}+\frac{a}{2N}}}   \frac{\sin\(\frac{\pi}{2}(N-a)\)}{2^{N-1}}\nonumber\\
			&\quad\times C_{m,N}\left(\substack{\tfrac{1}{2}, \tfrac{a}{2N}, 0, \tfrac{4\pi^{N+1}n}{yN^N}}\right)\bigg]+\hspace{-5pt} \lim_{a\to-2Nm-1+N}\left( \tfrac{1}{N}\tfrac{\Gamma \left(\frac{1+a}{N}\right) \zeta \left(\frac{1+a}{N}\right)}{y^{\frac{1+a}{N}}}+\tfrac{y}{2\pi^{2}} \sum_{k=0}^m \left(-\tfrac{y^2}{4\pi^2}\right)^k   \zeta(-2kN-N-a)\zeta(2k+2)\right).
		\end{align}
		We now simplify the expression in square-brackets on the right-hand side of \eqref{Eqn:Analytic continuationherglotz}. Using Theorem \ref{meijergsim} and separating the $k=m$ term from the expression containing $C_{m,N}$, where $C_{m,N}$ is given in \eqref{cmnsimplified}, we have
		\begin{align}\label{ra+otherterm}
			&{}_{\frac{1}{2}} K_{\frac{a}{2N}}^{(N)}\left(\frac{4\pi^{N+1}n}{yN^N}, 0\right) - \frac{2^{\frac{1}{2} + \frac{a+1}{N}} \pi^{\frac{(1-N)a}{2N} -N}} {\left(\frac{4\pi^{N+1}n}{yN^N}\right)^{1+\frac{1}{N}+\frac{a}{2N}}}   \frac{\sin\(\frac{\pi}{2}(N-a)\)}{2^{N-1}}C_{m,N}\left(\tfrac{1}{2}, \tfrac{a}{2N}, 0, \tfrac{4\pi^{N+1}n}{yN^N}\right)\nonumber\\
			&=R(a)+\frac{(y/n)^{1+{1\over N}+{a\over 2N}+2m}N^{1-a\over 2}} {\pi(2\pi)^{{3\over 2}+{1\over N}+a+N+2m(N+1)} }\sin \left(\frac{\pi}{2}(N-a)\right)\Gamma(N+2Nm+a+1),
		\end{align}
		where
		\begin{align}\label{defr}
			R(a)&:=R_{m, N}(a, n, y)\nonumber\\
			&:=\frac{(2\pi)^{a-2\over N}}{\pi^2 \sqrt{2}}\frac{(n/y)^{{a\over 2N}-{2\over N}}}{N^{{a\over 2}-2}} \Bigg\{\frac{ N^{N-a - \frac{1}{2}} \Gamma\left(\frac{1-N+a}{2}\right)\left(\frac{4\pi^{2N+2}n^2}{y^2N^{2N}}\right)^{\frac{1}{2}+\frac{1-a}{2N}} }{\Gamma\left(\frac{N-a}{2}\right)}{}_1F_{2N}\Bigg( \begin{matrix}
				1 \\ \left\langle \frac{1}{2} - \frac{a+1}{2N} + \frac{i}{2N} \right\rangle_{i=1}^{2N}
			\end{matrix}\Bigg| -\frac{4\pi^{2N+2}n^2}{y^2N^{2N}} \Bigg)\nonumber\\
			&\quad+A_{a, N}\left(\frac{4\pi^{2N+2}n^2}{y^2N^{2N}}\right)\Bigg\}+\frac{(y/n)^{1+{1\over N}+{a\over 2N}}N^{1-a\over 2}} {\pi (2\pi)^{{3\over 2}+{1\over N}+a+N}} \sin \left(\frac{\pi}{2}(N-a)\right)\sum_{k=0}^{m-1}\Gamma (N+2Nk+a+1)\left( \frac{(2\pi)^{N+1}n}{y}\right)^{-2k},
		\end{align}
		and $A_{a,N}(z)$ is defined in \eqref{aanzodd}.
		
		Replacing $k$ by $k-1$ in the last expression in the definition of $R(a)$, then using the fact that $\Gamma (-N+2Nk+a+1)=\frac{\pi}{\Gamma(N-a-2Nk)\sin(\pi a)}$ (which follows from \eqref{refl}) and applying
		$\sin\left(\frac{\pi}{2}(N-a)\right)=(-1)^{\frac{N-1}{2}}\cos\left(\frac{\pi a}{2}\right)$, we see that
		\begin{align}\label{V1}
			&\frac{\left(\frac{y}{n}\right)^{1+{1\over N}+{a\over 2N}}N^{1-a\over 2}} {\pi (2\pi)^{{3\over 2}+{1\over N}+a+N}} \sin \left(\frac{\pi}{2}(N-a)\right)\sum_{k=0}^{m-1}\Gamma (N+2Nk+a+1)\left( \frac{(2\pi)^{N+1}n}{y}\right)^{-2k}\nonumber\\
			&=\frac{(-1)^{\frac{N-1}{2}}\left(\frac{y}{n}\right)^{-1+{1\over N}+{a\over 2N}}N^{1-a\over 2}}{2(2\pi)^{{\frac{1}{N}}-{1\over 2}+a-N}\sin\left(\frac{\pi a}{2}\right)}\sum_{k=1}^{m}\frac{\left( (2\pi)^{N+1}n\over y\right)^{-2k} }{\Gamma(N-a-2Nk)}.
		\end{align}
		Thus from \eqref{defr}, \eqref{V1} and the fact 
		\begin{equation}\label{quotient gamma}
		\frac{\Gamma\left(\frac{1-N+a}{2}\right)}{\Gamma\left(\frac{N-a}{2}\right) } =\frac{(-1)^{\frac{N-1}{2}}\sqrt{\pi}2^{N-a-1}}{\Gamma(N-a)\sin\left(\frac{\pi a}{2}\right)},
		\end{equation}
		which results from \eqref{gammaprop}, and using \eqref{aanzodd} and \eqref{defE} with $b=0$, we find that, with $Z=\frac{4\pi^{2N+2}n^2}{y^2N^{2N}}$, 
		\begin{align}\label{rsimplified}
			R(a)&=\frac{\sqrt{N}(-1)^{\frac{N-1}{2}}}{2\sin\left(\frac{\pi a}{2}\right)}\left(\tfrac{4\pi^2n}{yN^{N}}\right)^{\frac{a}{2N}}\hspace{-5pt}\Bigg\{\left(\frac{y}{n}\right)^{-1+{1\over N}}\frac{(2\pi)^{{1\over 2}-{1\over N}+N}}{\G(N-a)}\left(  \frac{y}{(2\pi)^{N+1}n}\right)^{\frac{a}{N}} \hspace{-5pt} {}_1F_{2N}\Bigg( \begin{matrix}
				1 \\ \left\langle \frac{1}{2} - \frac{a+1}{2N} + \frac{i}{2N} \right\rangle_{i=1}^{2N}
			\end{matrix}\Bigg| \tfrac{4\pi^{2N+2}n^2}{y^2N^{2N}} \Bigg)\nonumber\\
			&\quad-\frac{\sqrt{\pi}}{\sqrt{2}N} \left [  \sum\limits_{k=0}^{\frac{N-1}{2}} \mathscr{E}_{a, N, Z}(k, 0)+  \sum\limits_{k= \frac{N+1}{2}}^{\frac{3N-1}{2}} e^{-\pi i a}\mathscr{E}_{a, N, Z}(k, 0)+  \sum\limits_{k= \frac{3N+1}{2}}^{2N-1} e^{-2\pi i a}\mathscr{E}_{a, N, Z}(k, 0)\right ]\nonumber\\
			&\quad+\frac{\left(\frac{y}{n}\right)^{-1+{1\over N}}}{(2\pi)^{{\frac{1}{N}}-{1\over 2}-N}}\left(  \frac{y}{(2\pi)^{N+1}n}\right)^{\frac{a}{N}}\sum_{k=1}^{m}\frac{\left( (2\pi)^{N+1}n\over y\right)^{-2k} }{\Gamma(N-a-2Nk)}\Bigg\}.
		\end{align}
		%\begin{align}\label{rsimplified}
		%R(a)&=\frac{(-1)^{\frac{N-1}{2}}}{2N^{\frac{a-1}{2}}\sin\left(\frac{\pi a}{2}\right)}\Bigg\{\left(\frac{y}{n}\right)^{-1+{1\over N}+{a\over 2N}}\frac{(2\pi)^{{1\over 2}-{1\over N}-a+N}}{\G(N-a)}  {}_1F_{2N}\Bigg( \begin{matrix}
			%1 \\ \left\langle \frac{1}{2} - \frac{a+1}{2N} + \frac{i}{2N} \right\rangle_{i=1}^{2N}
			%\end{matrix}\Bigg| -\frac{4\pi^{2N+2}n^2}{y^2N^{2N}} \Bigg)-\frac{(2\pi)^{\frac{a}{N}+\frac{1}{2}}}{2N\left(\frac{y}{n}\right)^{\frac{a}{2N}}} \nonumber\\
			%&\quad\times\left [\substack{  \sum\limits_{k=0}^{\frac{N-1}{2}} \mathscr{E}_{a, N, z}(k, 0)+  \sum\limits_{k= \frac{N+1}{2}}^{\frac{3N-1}{2}} e^{-\pi i a}\mathscr{E}_{a, N, z}(k, 0)+  \sum\limits_{k= \frac{3N+1}{2}}^{2N-1} e^{-2\pi i a}\mathscr{E}_{a, N, z}(k, 0)} \right ]+\frac{\left(\frac{y}{n}\right)^{-1+{1\over N}+{a\over 2N}}}{(2\pi)^{{\frac{1}{N}}-{1\over 2}+a-N}}\sum_{k=1}^{m}\frac{\left( (2\pi)^{N+1}n\over y\right)^{-2k} }{\Gamma(N-a-2Nk)}\Bigg\}.
			%\end{align}
			We now show that for $a=-2Nm-1+N$, the above expression takes the form $\frac{0}{0}$. Obviously, the sine function in the denominator is zero since $\sin\left(\frac{\pi }{2}(-2Nm-1+N)\right)=0$. The next task is to show that the expression in the curly braces of \eqref{rsimplified} equals zero for $a=-2Nm-1+N$. To that end, using \eqref{prudnikovid}, we see that
			\begin{align*}
				&\frac{(2\pi)^{2Nm+\frac{1}{2}}}{\G(2Nm+1)}\left(\frac{2\pi n}{y}\right)^{2m}  {}_1F_{2N}\Bigg( \begin{matrix}
					1 \\ \left\langle m + \frac{i}{2N} \right\rangle_{i=1}^{2N}
				\end{matrix}\Bigg| \frac{4\pi^{2N+2}n^2}{y^2N^{2N}} \Bigg)\nonumber\\
				&=\frac{(2\pi)^{2Nm+\frac{1}{2}}}{(2Nm)!}\left(\frac{2\pi n}{y}\right)^{2m}\left\lbrace \frac{(2Nm)!}{(2N)^{2Nm+1}\left( \frac{4\pi^{2N+2}n^2}{y^2N^{2N}}\right)^m } \Bigg[\sum_{k=0}^{2N-1}\exp\left\lbrace 2Ne^{\pi i k \over N} \left(\frac{2\pi^{N+1}n}{yN^N} \right)^{1\over N}  \right\rbrace \right. \nonumber \\
				& \quad\left. -(2N)^{2Nm+1}\sum_{k=1}^{m}\frac{\left( \frac{4\pi^{2N+2}n^2}{y^2N^{2N}}\right)^{m-k} }{(2N)^{2Nk}(2Nm-2Nk)!}\Bigg]\right\rbrace.
			\end{align*}
			Thus using the above equation, we see that the first term inside the curly braces on the right-hand side of \eqref{rsimplified} completely cancels with the other two terms in there for $a=-2Nm-1+N$. Therefore, in order to evaluate $\lim_{a\to-2Nm-1+N}R(a)$, we need to use L'Hopital's rule. Before doing this, however, we further simplify the expression in curly braces of \eqref{rsimplified}. The series definition of ${}_1F_{2N}$ along with an application of \eqref{gmf} yields
			\begin{align}
				&\frac{\left(\frac{y}{n}\right)^{-1+{1\over N}}}{(2\pi)^{{\frac{1}{N}}-{1\over 2}-N}}\left(  \frac{y}{(2\pi)^{N+1}n}\right)^{\frac{a}{N}}  \left[\frac{1}{\G(N-a)}{}_1F_{2N}\Bigg( \begin{matrix}
					1 \\ \left\langle \frac{1}{2} - \frac{a+1}{2N} + \frac{i}{2N} \right\rangle_{i=1}^{2N}
				\end{matrix}\Bigg| \tfrac{4\pi^{2N+2}n^2}{y^2N^{2N}} \Bigg)+ \sum_{k=1}^{m}\frac{\left( (2\pi)^{N+1}n\over y\right)^{-2k} }{\Gamma(N-a-2Nk)}\right]\nonumber\\
				&=\frac{\left(\frac{y}{n}\right)^{-1+{1\over N}}}{(2\pi)^{{\frac{1}{N}}-{1\over 2}-N}}\left(  \frac{y}{(2\pi)^{N+1}n}\right)^{\frac{a}{N}} \left[\sum_{h=0}^{\infty}\frac{1}{\Gamma(N-a+2Nh)}\left(\frac{(2\pi)^{N+1}n}{y}\right)^{2h}+\sum_{h=-m}^{-1}\frac{\left( (2\pi)^{N+1}n\over y\right)^{2h} }{\Gamma(N-a+2Nh)}\right]\nonumber\\
				&=\frac{\left(\frac{y}{n}\right)^{-1+{1\over N}}}{(2\pi)^{{\frac{1}{N}}-{1\over 2}-N}}\left(  \frac{y}{(2\pi)^{N+1}n}\right)^{\frac{a}{N}+2m}\sum_{h=0}^{\infty}\frac{\left( (2\pi)^{N+1}n\over y\right)^{2h} }{\Gamma(N-2Nm-a+2Nh)}\nonumber\\
				&=:R_{1}^{*}(a)\label{rstara}.
			\end{align}
			Further, the remaining expression in the curly braces in \eqref{rsimplified} (without the minus sign in front) can be rephrased using \eqref{aanzodd}, \eqref{meijergsimeqn} and \eqref{Reduced Meijer G} by
			\begin{align}\label{r2stara}
				&\frac{
					\sqrt{2}(-1)^{\frac{N+1}{2}}}{\sqrt{N}}\left(\frac{4\pi^{2N+2}n^2}{y^2N^{2N}}\right)^{-\frac{1}{N}}\sin\left(\frac{\pi a}{2}\right)A_{a,N}\left(\frac{4\pi^{2N+2}n^2}{y^2N^{2N}}\right) \nonumber\\
				&=-\frac{
					\sqrt{2}(-1)^{\frac{N+1}{2}}}{\sqrt{N}}\sum_{j=1}^N \frac{ (2N)^{2j-2} \left(\frac{\pi}{N}\right)^{1/2} }{\Gamma(2j-1)}S^{*}(a) \left(\frac{4\pi^{2N+2}n^2}{y^2N^{2N}}\right)^{\frac{j-1}{N}}{}_1F_{2N}\left( \left. \begin{matrix}
					1  \\ \left\langle \frac{j-1}{N}+\frac{i}{2N} \right\rangle_{i=1}^{2N}
				\end{matrix}\right| \frac{4\pi^{2N+2}n^2}{y^2N^{2N}} \right)\nonumber\\
				&=:R_{2}^{*}(a), \text{say},
			\end{align}
			where
			\begin{align}\label{s*a}
				S^{*}(a):=S_{j,N}^{*}(a):=\frac{(-1)^j\sin\left(\frac{\pi a}{2}\right)}{\cos\left(\pi \left(\frac{j}{N}+ \frac{a-1}{2N} \right)\right)}.
			\end{align}
			Therefore, from \eqref{rsimplified}, \eqref{rstara} and \eqref{r2stara},
			\begin{align}\label{rasimplified}
				R(a)&=\frac{1}{2}\sqrt{N}(-1)^{\frac{N-1}{2}}\left(\frac{4\pi^2n}{yN^{N}}\right)^{\frac{a}{2N}}\frac{ R_{1}^{*}(a)-R_{2}^{*}(a)}{\sin\left(\frac{\pi a}{2}\right)}
				%R(a)&=\frac{\sqrt{N}(-1)^{\frac{N-1}{2}}}{2\sin\left(\frac{\pi a}{2}\right)}\left(\frac{4\pi^2n}{yN^{N}}\right)^{\frac{a}{2N}}\left\{ R_{1}^{*}(a)-\frac{
					%\sqrt{2}(-1)^{\frac{N+1}{2}}}{\sqrt{N}}\left(\frac{4\pi^{2N+2}n^2}{y^2N^{2N}}\right)^{-\frac{1}{N}}\sin\left(\frac{\pi a}{2}\right)A_{a,N}\left(\frac{4\pi^{2N+2}n^2}{y^2N^{2N}}\right)\right\}\nonumber\\
				%&=\frac{\sqrt{N}(-1)^{\frac{N-1}{2}}}{2\sin\left(\frac{\pi a}{2}\right)}\left(\frac{4\pi^2n}{yN^{N}}\right)^{\frac{a}{2N}}\Bigg\{ R_{1}^{*}(a)+\frac{
				%\sqrt{2}(-1)^{\frac{N+1}{2}}}{\sqrt{N}}\left(\frac{4\pi^{2N+2}n^2}{y^2N^{2N}}\right)^{-\frac{1}{N}}\nonumber\\
			%&\quad\times\sum_{j=1}^N \frac{ (2N)^{2j-2} \left(\frac{\pi}{N}\right)^{1/2} }{\Gamma(2j-1)}S^{*}(a) \left(\frac{4\pi^{2N+2}n^2}{y^2N^{2N}}\right)^{\frac{j}{N}}{}_1F_{2N}\left( \left. \begin{matrix}
			%	1  \\ \left\langle \frac{j-1}{N}+\frac{i}{2N} \right\rangle_{i=1}^{2N}
			%\end{matrix}\right| \frac{4\pi^{2N+2}n^2}{y^2N^{2N}} \right)\Bigg\},
		\end{align}
		First,
		\begin{align*}
			%\frac{d}{da}R^{*}(a)&=\frac{\left(\frac{y}{n}\right)^{-1+{1\over N}}}{(2\pi)^{{\frac{1}{N}}-{1\over 2}-N}}\left(  \frac{y}{(2\pi)^{N+1}n}\right)^{2m+\frac{a}{N}}\Bigg\{\frac{1}{N}  \log\left(  \frac{y}{(2\pi)^{N+1}n}\right)\sum_{h=0}^{\infty}\frac{\left( (2\pi)^{N+1}n\over y\right)^{2h} }{\Gamma(N-2Nm-a+2Nh)}\nonumber\\
			%&\quad+\sum_{h=0}^{\infty}\frac{\left( (2\pi)^{N+1}n\over y\right)^{2h} \psi(N-2Nm-a+2Nh)}{\Gamma(N-2Nm-a+2Nh)}\Bigg\} 
			\frac{d}{da}R_{1}^{*}(a)&=\frac{\left(\frac{y}{n}\right)^{-1+{1\over N}}}{(2\pi)^{{\frac{1}{N}}-{1\over 2}-N}}\left(  \frac{y}{(2\pi)^{N+1}n}\right)^{2m+\frac{a}{N}}\sum_{h=0}^{\infty}\frac{ \psi(N-2Nm-a+2Nh)-\log\left((2\pi)^{N+1}n\over y\right)^{1/N}}{\Gamma(N-2Nm-a+2Nh)}\left( (2\pi)^{N+1}n\over y\right)^{2h}. 
		\end{align*}
		Applying Proposition \ref{kuchnahi}, we see that
		\begin{align}\label{rstarderivativespl}
			\left.\frac{d}{da}R_{1}^{*}(a)\right|_{a=-2Nm-1+N}&=\frac{\sqrt{2\pi}}{N}\sum_{k=-\frac{(N-1)}{2}}^{\frac{N-1}{2}}\left(\sinh\operatorname{Shi}-\cosh\operatorname{Chi}\right)\left(\substack{2\pi\left(\tfrac{2\pi n}{y}\right)^{\frac{1}{N}}  e^{\frac{\pi i k}{N}} }\right)\nonumber\\
			&\quad+
			\frac{\pi^{\frac{3}{2}}}{N\sqrt{2}}\sum_{j=1}^{N-1} \frac{(-1)^j}{\sin \left(\frac{\pi j}{N} \right)} \sum_{h=0}^\infty \frac{\left(2\pi\left(\frac{2\pi n}{y}\right)^{\frac{1}{N}}\right)^{2Nh+2j}}{(2Nh+2j)!}.
		\end{align}
		We next simplify $S^{*}(a)$ defined in \eqref{s*a}. Using \eqref{cc}, we see that
		\begin{align*}
			S^{*}(a)=\frac{\cos\left(N\pi\left( {j\over N}+{a-1\over 2N}\right)  \right) }{\cos\left(\pi\left( {j\over N}+{a-1\over 2N}\right)  \right) }=(-1)^{\frac{N-1       }{2}}\sum_{\ell=-{(N-1)\over 2}}^{N-1\over 2}(-1)^{\ell}\exp{\left(-{2\pi i \ell j \over N}-{\pi i \ell (a-1)\over N}\right)}
		\end{align*}
		so that
		\begin{align}\label{sstar}
			\left.\frac{d}{da}S^*(a)\right|_{a=-2Nm-1+N}&={\pi i  \over N}(-1)^{\frac{N+1}{2}}\sum_{l=-{(N-1)\over 2}}^{N-1\over 2} l e^{-{2\pi i l (j-1)\over N}}\nonumber\\
			&=\begin{cases}
				0 \hspace{30mm}\text{ for } N\mid (j-1),\\
				i\pi\frac{(-1)^{1-j+\frac{N+1}{2}}}{\left(e^{\frac{\pi i(1-j)}{N}}-e^{-\frac{\pi i(1-j)}{N}}\right)}\text{ for } N\nmid (j-1),
				\end{cases}
		\end{align}
		where in the last step, we used \eqref{Orthogonal character}. The series definition of ${}_1F_{2N}$ along with \eqref{gmf} gives
		\begin{align*}
			{}_1F_{2N}\left( \left. \begin{matrix}
				1  \\ \left\langle{j-1 \over N}+{i\over 2N}\right\rangle_{i=1}^{ 2N}
			\end{matrix}\right| \frac{4\pi^{2N+2}n^2}{y^2N^{2N}}\right)
			=\sum_{h=0}^{\infty}\frac{\Gamma 
				\left({2j-1} \right)}{\Gamma 
				\left({2j-1+ 2Nh} \right)}\left((2\pi)^{N+1} n\over y \right)^{2h}.
		\end{align*}
		Therefore, from \eqref{r2stara}, \eqref{sstar} and the above identity,
		\begin{align}\label{r2staraderivative}
			\left.\frac{d}{da}R_{2}^{*}(a)\right|_{a=-2Nm-1+N}\hspace{-5pt}&=\frac{\pi^{3/2}i}{2\sqrt{2}N^3}\sum_{j=2}^{N}\frac{(-1)^j(2N)^{2j}}{\left(e^{\frac{\pi i(1-j)}{N}}-e^{-\frac{\pi i(1-j)}{N}}\right)}\left(\tfrac{ 4\pi^{2N+2}n^2}{y^2 N^{2N}} \right)^{j-1\over N} \sum_{h=0}^{\infty}\frac{1 }{\Gamma\left({2j-1+ 2Nh} \right)}\left(\tfrac{(2\pi)^{N+1} n}{y }\right)^{2h}\nonumber\\
			&=\frac{\pi^{3\over 2} }{N\sqrt{2}}\sum_{j=1}^{N-1}\frac{(-1)^j}{\sin\left( {\pi j\over N}\right) }{}{} \left(\frac{ (2\pi)^{N+1}n}{y } \right)^{2j\over N}\sum_{h=0}^{\infty}\frac{1}{\Gamma 
				\left({2j+1+ 2Nh} \right)}\left((2\pi)^{N+1} n\over y \right)^{2h}.
		\end{align}
		Inserting \eqref{rstarderivativespl} and \eqref{r2staraderivative} in \eqref{rasimplified}, it can be seen that the sums over $j$ completely cancel out whence
		\begin{align}\label{ralimitfinal}
			\lim\limits_{a\to -2Nm-1+N}R(a)&=\frac{\sqrt{2}(-1)^m}{\sqrt{N\pi}}\left( \frac{4\pi^2n}{yN^N}\right)^{{-m-\frac{1}{2N}+\frac{1}{2}}}\sum_{k=-\frac{(N-1)}{2}}^{\frac{N-1}{2}}\left(\sinh\operatorname{Shi}-\cosh\operatorname{Chi}\right)\left(\substack{2\pi\left(\tfrac{2\pi n}{y}\right)^{\frac{1}{N}}  e^{\frac{\pi i k}{N}} }\right).
		\end{align}
		Thus from \eqref{ra+otherterm}, \eqref{ralimitfinal}, and interchanging the order of limit and summation which is justified since the summand of the series below is $\mathcal{O}\left(n^{\max \left(-\frac{\re(a)}{2N}, \frac{\re(a)}{2N} + \frac{1}{N} - 1 \right)-2m-3-\frac{1}{N}-\frac{\re(a)}{2N}+\epsilon} \right)$ for any $\epsilon>0$, we see that
		\begin{align}\label{unnamed}
			&\lim_{a\to-2Nm-1+N}\tfrac{2 (2\pi)^{\frac{1}{N}-\frac{1}{2}} N^{\frac{a-1}{2}}}{y^{\frac{1}{N}+\frac{a}{2N}}} \sum_{n=1}^\infty \frac{S_a^{(N)}(n)}{n^{\frac{a}{2N}}} \bigg [{}_{\frac{1}{2}} K_{\frac{a}{2N}}^{(N)}\left(\substack{\frac{4\pi^{N+1}n}{yN^N}, 0}\right) - \frac{2^{\frac{1}{2} + \frac{a+1}{N}} \pi^{\frac{(1-N)a}{2N} -N}} {\left(\frac{4\pi^{N+1}n}{yN^N}\right)^{1+\frac{1}{N}+\frac{a}{2N}}}   \frac{\sin\(\frac{\pi}{2}(N-a)\)}{2^{N-1}}\nonumber\\
			&\quad\times C_{m,N}\left(\substack{\tfrac{1}{2}, \tfrac{a}{2N}, 0, \tfrac{4\pi^{N+1}n}{yN^N}}\right)\bigg]\nonumber\\
			&={4(-1)^m\over Ny}\left( y\over 2\pi\right)^{2m} \sum_{k=-{(N-1)\over 2}}^{k={N-1\over 2}}\sum_{n=1}^{\infty} {S_{-2Nm-1+N}^{(N)}(n)}\left(\sinh\operatorname{Shi}-\cosh\operatorname{Chi}\right)\left(\substack{2\pi\left(\tfrac{2\pi n}{y}\right)^{\frac{1}{N}}  e^{\frac{\pi i k}{N}} }\right)	\nonumber \\
			&\qquad +\frac{(-1)^my^{2m+1}\Gamma(2N)}{\pi^2(2\pi)^{2(N+m)}}\sum_{n=1}^{\infty} \frac{S_{-2Nm-1+N}^{(N)}(n)}{n^{2}}.
		\end{align}
		Using the second equality in \eqref{2equiv}
		in the first step below and the definition of $S_{a}^{(N)}(n)$ from \eqref{defbf} in the second step, we see that
		\begin{align}\label{nagoyaapp}
			&\sum_{n=1}^{\infty} {S_{-2Nm-1+N}^{(N)}(n)}\left(\sinh\operatorname{Shi}-\cosh\operatorname{Chi}\right)\left(\substack{2\pi\left(\tfrac{2\pi n}{y}\right)^{\frac{1}{N}}  e^{\frac{\pi i k}{N}} }\right)\nonumber\\
			&=\sum_{n=1}^{\infty} S_{-2Nm-1+N}^{(N)}(n)\int_{0}^{\infty}\frac{t\cos(t)\, dt}{t^2+\left(\substack{2\pi\left(\frac{2\pi n}{y}\right)^{\frac{1}{N}}e^{\frac{i\pi k}{N}}}\right)^2}\nonumber\\
			&=\sum_{n=1}^{\infty}\frac{1}{n^{2m}}\sum_{\ell=1}^{\infty}\int_{0}^{\infty}\frac{t\cos(t)\, dt}{t^2+\ell^2\left(\substack{2\pi\left(\frac{2\pi n}{y}\right)^{\frac{1}{N}}e^{\frac{i\pi k}{N}}}\right)^2}\nonumber\\
			&=\frac{1}{2}\sum_{n=1}^{\infty}\frac{1}{n^{2m}}\left[\log\left(\substack{\left(2\pi n\over y \right)^{1\over N}e^{\pi i k \over N}}\right)-\frac{1}{2}\left\{ \psi\left(\substack{i\left(2\pi n\over y \right)^{1\over N}e^{\pi i k \over N}}  \right)+\psi\left(\substack{-i\left(2\pi n\over y \right)^{1\over N}e^{\pi i k \over N}  }\right)\right\} \right],
		\end{align}
	where in the last step, we used \cite[Theorem 2.2]{DGKM}.
%		Note that from \eqref{saNn}, for Re$(s)>\max\left(\frac{1}{N},\frac{1+\re(a)}{N}\right)$,
%		\begin{equation}\label{saNnds}
%			\sum_{n=1}^{\infty}\frac{S_a^{(N)}(n)}{n^{s}}=\zeta(Ns)\zeta\left(s+1-\frac{1+a}{N}\right).
%		\end{equation}
		Inserting \eqref{nagoyaapp} and \eqref{Sdirichlet} with $s=2$ in \eqref{unnamed} and using the functional equation for $\zeta(s)$ and the fact that $\sum_{n=1}^{\infty}\log(n)n^{-s}=-\zeta'(s)$ for Re$(s)>1$, we deduce that
		\begin{align}\label{mainlimit}
			&\lim_{a\to-2Nm-1+N}\tfrac{2 (2\pi)^{\frac{1}{N}-\frac{1}{2}} N^{\frac{a-1}{2}}}{y^{\frac{1}{N}+\frac{a}{2N}}} \sum_{n=1}^\infty \frac{S_a^{(N)}(n)}{n^{\frac{a}{2N}}} \bigg [{}_{\frac{1}{2}} K_{\frac{a}{2N}}^{(N)}\left(\substack{\frac{4\pi^{N+1}n}{yN^N}, 0}\right) - \frac{2^{\frac{1}{2} + \frac{a+1}{N}} \pi^{\frac{(1-N)a}{2N} -N}} {\left(\frac{4\pi^{N+1}n}{yN^N}\right)^{1+\frac{1}{N}+\frac{a}{2N}}}   \frac{\sin\(\frac{\pi}{2}(N-a)\)}{2^{N-1}}\nonumber\\
			&\quad\times C_{m,N}\left(\substack{\tfrac{1}{2}, \tfrac{a}{2N}, 0, \tfrac{4\pi^{N+1}n}{yN^N}}\right)\bigg]\nonumber\\
			&=\frac{(-1)^my}{\pi^2(2\pi)^{2N}}{\left( y\over 2\pi\right) }^{2m}\Gamma(2N)\zeta(2N)\zeta(2m+2)\nonumber\\
			&\quad+{2(-1)^m\over Ny}\left( y\over 2\pi\right)^{2m} \sum_{k=-{(N-1)\over 2}}^{k={N-1\over 2}}\sum_{n=1}^{\infty}\frac{1}{n^{2m}}\left[\log\left(\substack{\left(2\pi n\over y \right)^{1\over N}e^{\pi i k \over N}}\right)-\frac{1}{2}\left\{ \psi\left(\substack{i\left(2\pi n\over y \right)^{1\over N}e^{\pi i k \over N}}  \right)+\psi\left(\substack{-i\left(2\pi n\over y \right)^{1\over N}e^{\pi i k \over N}  }\right)\right\} \right]\nonumber\\
			&=-\frac{(-1)^my}{2\pi^2}{\left( y\over 2\pi\right) }^{2m}\zeta(1-2N)\zeta(2m+2)+{2(-1)^m\over Ny}\left( y\over 2\pi\right)^{2m} \left(-\zeta'(2m)+\log\left(\frac{2\pi}{y}\right)\zeta(2m)\right)\nonumber\\
			&\quad-
			{(-1)^m\over Ny}\left( y\over 2\pi\right)^{2m} \sum_{j=-{(N-1)\over 2}}^{j={N-1\over 2}}\sum_{n=1}^{\infty}\left\{ \psi\left(\substack{i\left(2\pi n\over y \right)^{1\over N}e^{\pi i j \over N}}  \right)+\psi\left(\substack{-i\left(2\pi n\over y \right)^{1\over N}e^{\pi i j \over N}  }\right)\right\}.
		\end{align}
		We now evaluate the second limit on the right-hand side of \eqref{Eqn:Analytic continuationherglotz}, that is, 
		\begin{align*}
			L_1:=\lim_{a\to-2Nm-1+N}\left( \frac{1}{N}\frac{\Gamma \left(\frac{1+a}{N}\right) \zeta \left(\frac{1+a}{N}\right)}{y^{\frac{1+a}{N}}}+\frac{y}{2\pi^{2}} \sum_{k=0}^m \left(-\frac{y^2}{4\pi^2}\right)^k   \zeta(-2kN-N-a)\zeta(2k+2)\right).
		\end{align*}
		Note that inside the finite sum, the only problematic term for $a=-2Nm-1+N$ is $k=m-1$. Thus, \begin{align}\label{l1}
			L_1&=\frac{y}{2\pi^{2}} \sum_{k=0}^{m-2} \left(-\frac{y^2}{4\pi^2}\right)^k   \zeta(-2kN+2Nm+1-2N)\zeta(2k+2)+(-1)^m\frac{y}{2\pi^{2}}\left(\frac{y}{2\pi}\right)^{2m}\zeta(1-2N)\zeta(2m+2)+L^{*},
		\end{align}
		where
		\begin{align*}
			L^{*}&:=\lim_{a\to-2Nm-1+N}\left( \frac{1}{N}\frac{\Gamma \left(\frac{1+a}{N}\right) \zeta \left(\frac{1+a}{N}\right)}{y^{\frac{1+a}{N}}}+\frac{y}{2\pi^{2}}\left(-\frac{y^2}{4\pi^2}\right)^{m-1}   \zeta(-2Nm+N-a)\zeta(2m+2)\right)\nonumber\\
			&=\lim_{s\to 1-2m}\left(\frac{1}{N}\frac{\Gamma(s)\zeta(s)}{y^s}-\frac{2}{y}\left(\frac{-y^2}{4\pi^2}\right)^{m}\zeta(-2Nm+N+1-Ns)\zeta(2m)\right). 
		\end{align*}
		Using the power series expansions around $s=1-2m$, we observe that
		\begin{align*}
			\zeta(-2Nm+N+1-Ns)&=\frac{1}{-2Nm+N-Ns}+\gamma+O(|-2Nm+N-Ns|)\nonumber\\
			\Gamma(s) &=\dfrac{-1}{\left( 2m-1\right) ! \left( s+2m-1 \right) }-\dfrac{\psi \left( 2m\right) }{\left( 2m-1\right) !}+O\left( \left| s+2m-1\right| \right),\nonumber\\
			\zeta(s)&=\zeta(1-2m)+\zeta'(1-2m)(s+2m-1)+O(|s+2m-1|^2),\nonumber\\
			y^{-s}&=y^{2m-1}\left(1-(s+2m-1)\log y+O_y(|s+2m-1|^2)\right),
		\end{align*}
		as $s\to1-2m$, the above limit can be evaluated to
		\begin{align}\label{lstar}
			L^{*}&=\frac{y^{2m-1}}{N(2m)!}\left\{2m\zeta(1-2m)\left(\log y-\psi(2m)\right)-2m\zeta'(1-2m)+N\gamma B_{2m}\right\}\nonumber\\
			&=\dfrac{y^{2m-1}B_{2m}}{N\left( 2m\right) !}\log \left( \dfrac{2 \pi }{y}\right) +\dfrac{2 \left( -1\right) ^{m}\zeta'\left( 2m\right) y^{2m-1}}{N\left( 2 \pi \right) ^{2m}}+\dfrac{y^{2m-1}\gamma B_{2m}}{\left( 2m\right)!},
		\end{align}
		where we have repeatedly used Euler's formula \eqref{ef} and the fact that $\zeta(1-2m)=-B_{2m}/(2m)$. Thus replacing $k$ by $k-1$  in \eqref{l1} and combining the resultant with \eqref{lstar}, we deduce that
		\begin{align}\label{l1eval}
			L_{1}&= -\frac{2}{y} \sum_{k=1}^{m-1} \left(-\frac{y^2}{4\pi^2}\right)^k   \zeta(-2kN+2Nm+1)\zeta(2k)+(-1)^m\frac{y}{2\pi^{2}}\left(\frac{y}{2\pi}\right)^{2m}\zeta(1-2N)\zeta(2m+2)\nonumber\\
			&\quad+\dfrac{y^{2m-1}B_{2m}}{N\left( 2m\right) !}\log \left( \dfrac{2 \pi }{y}\right) +\dfrac{2 \left( -1\right) ^{m}\zeta'\left( 2m\right) y^{2m-1}}{N\left( 2 \pi \right) ^{2m}}+\dfrac{y^{2m-1}\gamma B_{2m}}{\left( 2m\right)!}.
		\end{align}
		Now substitute \eqref{mainlimit} and \eqref{l1eval} in \eqref{Eqn:Analytic continuationherglotz}, then move the finite sum over $k$ on the left-hand side, induct the term $-\frac{1}{y}\zeta(2Nm+1)$ as its $k=0$ term, and rearrange to arrive at
		\begin{align*}
			&\frac{1}{2}\zeta(2Nm+1-N)+\sum_{n=1}^\infty \sigma_{-2Nm-1+N}^{(N)}(n) e^{-ny}  +\frac{2}{y} \sum_{k=0}^{m-1} \left(-\frac{y^2}{4\pi^2}\right)^k   \zeta(-2kN+2Nm+1)\zeta(2k) \nonumber\\
			&=\dfrac{y^{2m-1}\gamma B_{2m}}{\left( 2m\right)!}-
			{(-1)^m\over Ny}\left( y\over 2\pi\right)^{2m} \sum_{j=-{(N-1)\over 2}}^{j={N-1\over 2}}\sum_{n=1}^{\infty}\left\{ \psi\left(\substack{i\left(2\pi n\over y \right)^{1\over N}e^{\pi i j \over N}}  \right)+\psi\left(\substack{-i\left(2\pi n\over y \right)^{1\over N}e^{\pi i j \over N}  }\right)\right\}.
		\end{align*}
		Finally let $\alpha=2^{-N}y$ and $\beta = 2\pi^{1+{1\over N}}y^{-\frac{1}{N}}$.  Multiplying both sides of the above equation by $\a^{-\left(\frac{2Nm-1}{N+1}-\frac{1}{2}\right)}$, apply Euler's formula \eqref{ef} and \eqref{fkny} and simplify to finally arrive at \eqref{gencompanion}.
				\end{proof}
			\subsection{A transformation for the series $\sum_{n=1}^{\infty}\sigma_{2Nm-1+N}^{(N)}(n)e^{-ny}$}\label{a=2Nm-1+N}
			
		 Here we first obtain the special case $a=2Nm-1+N$ of our general transformation in Theorem \ref{In terms of G}. We then discuss a corollary of this special case itself which is useful in obtaining the asymptotic estimate for the generating function of power partitions with ``$n^{2N-1}$ copies of $n^{N}$'' defined in \eqref{ppN} as $q\to1^{-}$.
			
		\begin{proof}[Theorem \textup{\ref{thm_resultt}}][]
				Suppose $z=\frac{4\pi^{N+1}n}{yN^N}$. Taking limit as $a\to 2Nm-1+N$ in Theorem \ref{In terms of G}, we get
				\begin{align}\label{Gen 2m 1}
					&\sum_{n=1}^\infty \sigma_{2Nm-1+N}^{(N)}(n) e^{-ny}- \frac{1}{y}\zeta(1-2Nm)	-\frac{\Gamma(2m+1)\z(2m+1)}{Ny^{2m+1}}\nonumber \\
					& ={N^{3/2}\over \pi^{5/2}}\left({2\pi \over y}\right)^{2m+1-{2\over N}} \sum_{n=1}^{\infty} \frac{S_{2Nm-1+N}^{(N)}(n)}{n^{2/N}}\lim_{a\to 2Nm-1+N}G_{1, \, \, 2N+1}^{N+1, \, \, 1} \left(\begin{matrix}
						\frac{1}{2} + \frac{1-a}{2N}\\
						\frac{1}{2} + \frac{1-a}{2N}, \langle \frac{i}{N}\rangle; \langle 1+\frac{3}{2N}-\frac{i}{N} \rangle
					\end{matrix} \Bigg | \frac{z^2}{4} \right).
				\end{align}	
			We evaluate the limit of right-hand side of \eqref{Gen 2m 1}. Using \eqref{Reduced Meijer G}, the series definition of ${}_1F_{2N}$ and \eqref{quotient gamma}, we have 
			\begin{align}\label{Gen 2m 2}
				G_{1, \, \, 2N+1}^{N+1, \, \, 1} \left(\begin{matrix}
					\frac{1}{2} + \frac{1-a}{2N}\\
					\frac{1}{2} + \frac{1-a}{2N}, \langle \frac{i}{N}\rangle; \langle 1+\frac{3}{2N}-\frac{i}{N} \rangle
				\end{matrix} \Bigg | \frac{z^2}{4} \right)
			=T_1^{*}+T_2^{*},
			\end{align}
		where 
		\begin{align}
			T_1^{*}&:=T_1^{*}(z, a, N):=\frac{(-1)^{N-1\over 2}\sqrt{\pi N}(2N)^{N-a-1}}{\G(N-a)\sin\left(\pi a\over 2 \right) }\left(z^2\over 4 \right)^{{1\over 2}+{1-a\over 2N}}\sum_{h=0}^{\infty}\frac{\prod_{i=1}^{2N}\G\left( {1\over 2}-{a\over 2N}+{i-1\over 2N}\right) }{\prod_{i=1}^{2N}\G\left( {1\over 2}-{a\over 2N}+h+{i-1\over 2N}\right) } \left(z^2\over 4 \right)^h\nonumber 
			\\
		&\hspace{2.5cm}\quad+\frac{\sqrt{\pi/N}\left(z^2\over 4 \right)^{1/N}}{\cos\left( \pi(a+1)\over 2N\right) }\sum_{h=0}^{\infty}\frac{\prod_{i=1}^{2N}\G\left( {1\over 2N}+{i-1\over 2N}\right) }{\prod_{i=1}^{2N}\G\left( {1\over 2N}+h+{i-1\over 2N}\right) } \left(z^2\over 4 \right)^h,\nonumber\\
			T_2^{*}	&:=T_2^{*}(z, a, N):=\sum_{j=2}^N \frac{(-1)^{j-1} (2N)^{2j-2} (\pi/N)^{1/2} \left(  {z^2\over 4}\right)^{j/N} }{\Gamma(2j-1)\cos \pi \left(\frac{j}{N}+ \frac{a-1}{2N} \right)} \,	\sum_{h=0}^{\infty}\frac{\prod_{i=1}^{2N}\G\left( {2j-1\over 2N}+{i-1\over 2N}\right) }{\prod_{i=1}^{2N}\G\left( {2j-1\over 2N}+h+{i-1\over 2N}\right) } \left(z^2\over 4 \right)^h.\nonumber
		\end{align}
		 Using Gauss multiplication formula \eqref{gmf} and the fact $\cos\left( \pi(a+1)\over 2\right)=-\sin\left( \pi a\over 2\right)$, we arrive at
			\begin{align}\label{T*}
				T_1^{*}&=\frac{(-1)^{N-1\over 2}\sqrt{\pi N}(2N)^{N-a-1}}{\sin\left(\pi a\over 2 \right) }\left(z^2\over 4 \right)^{{1\over 2}+{1-a\over 2N}}\left[ \sum_{h=0}^{m-1}+\sum_{h=m}^{\infty}\right] \frac{(2N)^{2Nh}}{\G\left(N-a+ 2Nh\right) } \left(z^2\over 4 \right)^h\nonumber 
				\\
				&-\frac{\sqrt{\pi/N}\left(z^2\over 4 \right)^{1/N}}{\sin\left( \pi a\over 2\right) }\frac{\cos\left( \pi(a+1)\over 2\right)}{\cos\left( \pi(a+1)\over 2N\right)}\sum_{h=0}^{\infty}\frac{(2N)^{2Nh} }{\G\left( {1+2Nh}\right) } \left(z^2\over 4 \right)^h.\nonumber \\
	&={(-1)^{N-1\over 2}\sqrt{\pi N}(2N)^{N-a-1}}\left(z^2\over 4 \right)^{{1\over 2}+{1-a\over 2N}}\sum_{h=0}^{m-1} \frac{(2N)^{2Nh}\left(z^2\over 4 \right)^h}{{\sin\left(\pi a\over 2 \right) }\G\left(N-a+ 2Nh\right) } \nonumber 
				\\
				&+\frac{(-1)^{N-1\over 2}\sqrt{\pi N}(2N)^{2Nm+N-a-1}}{\sin\left(\pi a\over 2 \right) }\left(z^2\over 4 \right)^{{2Nm+N-a+1\over 2N}} \sum_{h=0}^{\infty} \frac{(2N)^{2Nh}}{\G\left(2Nm+2Nh+N-a\right) } \left(z^2\over 4 \right)^h\nonumber 
				\\
				&-\frac{\sqrt{\pi/N}\left(z^2\over 4 \right)^{1/N}}{\sin\left( \pi a\over 2\right) }(-1)^{N-1\over 2}\sum_{j=-(N-1)}^{N-1}{\vphantom{\sum}}''i^je^{-\pi i j{(a+1)\over 2N}}\sum_{h=0}^{\infty}\frac{(2N)^{2Nh} }{\G\left( {1+2Nh}\right) } \left(z^2\over 4 \right)^h\nonumber\\
				&=:B_1+B_2-B_3,
			\end{align}
		where in the penultimate step we replaced $h$ by $h+m$ in the sum over $h$ from $m$ to $\infty$ and also applied Lemma \ref{cheby} to simplify the last expression.  
			Now 
			\begin{align*}
				\lim_{a\to 2Nm-1+N}B_1&={(-1)^{N-1\over 2}\sqrt{\pi N}(2N)^{-2Nm}}\left(z^2\over 4 \right)^{-m+{1\over N}}\sum_{h=0}^{m-1} {(2N)^{2Nh}\left(z^2\over 4 \right)^h}L_1,
			\end{align*}
			where $L_1=
			\lim\limits_{a\to 2Nm-1+N}{1\over {\sin\left(\pi a\over 2 \right) }\G\left(N-a+ 2Nh\right) }$. Replacing $a$ by $N-a+2Nh$ and using \eqref{usefullim}, we get 
			\begin{align*}
				L_1=\lim\limits_{a\to -(2(Nm-Nh-1)+1)}{1\over (-1)^{{N-1\over 2}+h}\cos\left(\pi a\over 2 \right)\G(a)}= (-1)^{{N-1\over 2}+h}\frac{2(2Nm-2Nh-1)!}{(-1)^{m-h}\pi}
			\end{align*}
			Hence, 
			\begin{align}\label{B1}
				\lim_{a\to 2Nm-1+N}B_1&=\frac{(-1)^m2\sqrt{\pi N}}{\pi (2N)^{2Nm}}\left(z^2\over 4 \right)^{-m+{1\over N}}\sum_{h=0}^{m-1} {(2N)^{2Nh}\left(z^2\over 4 \right)^h}\G(2Nm-2Nh).
			\end{align}
			Also, 
			\begin{align}
				B_2-B_3&=\frac{(-1)^{{N-1\over 2}}\sqrt{\pi N}}{\sin\left(\pi a\over 2 \right) }\left\lbrace  (2N)^{2Nm-1+N}\left(z^2\over 4 \right)^{m+{1\over 2}+{1\over 2N}}\sum_{h=0}^{\infty} \frac{(2N)^{2Nh}\left(z^2\over 4 \right)^h\left( 2N(\frac{z}{2})^{\frac{1}{N}}\right)^{-a} }{\G(2Nh+2Nm+N-a)}\right. \nonumber 
				\\
				&\qquad\qquad\qquad\qquad\quad\left. -{1\over N}\left(z^2\over 4 \right)^{1/N}\sum_{h=0}^{\infty}\frac{(2N)^{2Nh} }{\G\left( {1+2Nh}\right) } \left(z^2\over 4 \right)^h\sum_{j=-\frac{\left(N-1\right)}{2} }^{N-1\over 2 } (-1)^je^{-{\pi i j \over N}(a+1)}\right\rbrace,
			\end{align}
		Note that for $a=2Nm-1+N$, the expression inside the curly braces is zero. Thus $B_2-B_3$ takes the form $0/0$. Consequently, by L'Hopital's rule,
			\begin{align}\label{B2-B3}
				&\lim_{a\to 2Nm-1+N}(B_2-B_3)\nonumber 
				\\
				&=\lim_{a\to 2Nm-1+N}\frac{2(-1)^{{N-1\over 2}}\sqrt{\pi N}}{\pi\cos\left(\pi a\over 2 \right)} \left\lbrace  (2N)^{2Nm-1+N}\left(z^2\over 4 \right)^{m+{1\over 2}+{1\over 2N}}\right.\nonumber 
				\\
				&\quad\times\sum_{h=0}^{\infty}\frac{(2N)^{2Nh}\left(z^2\over 4 \right)^h\left[-\G(2Nh+2Nm+N-a)\log\left( 2N(\frac{z}{2})^{\frac{1}{N}}\right)+\G'(2Nh+2Nm+N-a) \right] }{\left( 2N(\frac{z}{2})^{\frac{1}{N}}\right)^{a}\G^2(2Nh+2Nm+N-a)}  \nonumber 
				\\
				&\quad\left. -{1\over N}\left(z^2\over 4 \right)^{1/N}\sum_{h=0}^{\infty}\frac{(2N)^{2Nh} }{\G\left( {1+2Nh}\right) } \left(z^2\over 4 \right)^h\sum_{j=-\frac{\left(N-1\right)}{2}}^{N-1\over 2 } (-1)^{j+1}{\pi i j \over N}e^{-{\pi i j \over N}(a+1)}\right\rbrace \nonumber 
				\\
				&=	\frac{2(-1)^{{N-1\over 2}}\sqrt{\pi N}}{\pi(-1)^{m+{N-1\over 2}}}\left\lbrace(2N)^{2Nm-1+N}\left(z^2\over 4 \right)^{m+{1\over 2}+{1\over 2N}}\sum_{h=0}^{\infty}\frac{(2N)^{2Nh}\left(z^2\over 4 \right)^h\left[\psi(2Nh+1)-\log\left( 2N(\frac{z}{2})^{\frac{1}{N}}\right)\right]}{\left( 2N(\frac{z}{2})^{\frac{1}{N}}\right)^{2Nm-1+N}\G(2Nh+1)}  \right. \nonumber 
				\\
				&\quad+\left. {\pi i\over N^2}\left(z^2\over 4 \right)^{1/N}\sum_{h=0}^{\infty}\frac{(2N)^{2Nh} }{\G\left( {1+2Nh}\right) } \left(z^2\over 4 \right)^h\sum_{j=-\frac{\left(N-1\right)}{2}}^{N-1\over 2 } (-1)^{j} j e^{-{\pi i j }(2m+1)}\right\rbrace\nonumber\\
%			\end{align}
%			Using the fact $e^{-2\pi i j m}=1$ and $e^{i\pi}=-1$, we find that the last term of above expression is zero. Consequently using \eqref{Generalized Dixon-Ferror identity odd} we deduce 
		%	\begin{align}
				&=2(-1)^m\sqrt{N\over \pi}\left( z\over 2\right)^{2\over N}\sum_{h=0}^{\infty}\frac{\psi(2Nh+1)-\log\left( 2N(\frac{z}{2})^{\frac{1}{N}}\right)}{\G(2Nh+1)}  \left( 2N\left(\frac{z}{2}\right)^{\frac{1}{N}}\right)^{2Nh}\nonumber 
				\\
				&=(-1)^m{2\over \sqrt{\pi N}}\left( z\over 2\right)^{2\over N}\left[ \sum_{k=-{\left(N-1 \right) \over 2} }^{N-1\over 2 }\left( \sinh\operatorname{Shi}-\cosh\operatorname{Chi}\right) \left(  2N\left(\frac{z}{2}\right)^{\frac{1}{N}}e^{\pi i k \over N}\right)\right.\nonumber 
				\\
				& \qquad\qquad\qquad\qquad\qquad\left.  +{\pi \over 2}\sum_{j=1}^{N-1}\frac{(-1)^j}{\sin\left( \pi j \over N\right) }\sum_{h=0}^{\infty} \frac{\left(  2N(\frac{z}{2})^{\frac{1}{N}}\right)^{2Nh+2j} }{(2Nh+2j)!}\right],
			\end{align}
		where in the penultimate step, we used the fact that $\sum\limits_{j=-\frac{\left(N-1\right)}{2}}^{N-1\over 2 }j=0$, and in the last step, \eqref{Generalized Dixon-Ferror identity odd}.
		
		Now letting $a\to2Nm-1+N$ in \eqref{T*} and then inserting \eqref{B1}, \eqref{B2-B3} in it leads to
		\begin{align}\label{limit T*}
\lim_{a\to2Nm-1+N}T_1^{*}&=\frac{(-1)^m2\sqrt{\pi N}}{\pi (2N)^{2Nm}}\left(z^2\over 4 \right)^{-m+{1\over N}}\sum_{h=0}^{m-1} {(2N)^{2Nh}\left(z^2\over 4 \right)^h}\G(2Nm-2Nh)\nonumber\\
&\quad+(-1)^m{2\over \sqrt{\pi N}}\left( z\over 2\right)^{2\over N}\left[ \sum_{k=-{\left(N-1 \right) \over 2} }^{N-1\over 2 }\left( \sinh\operatorname{Shi}-\cosh\operatorname{Chi}\right) \left(  2N\left(\frac{z}{2}\right)^{\frac{1}{N}}e^{\pi i k \over N}\right)\right.\nonumber 
\\
& \qquad\qquad\qquad\qquad\qquad\left.  +{\pi \over 2}\sum_{j=1}^{N-1}\frac{(-1)^j}{\sin\left( \pi j \over N\right) }\sum_{h=0}^{\infty} \frac{\left(  2N\left(\frac{z}{2}\right)^{\frac{1}{N}}\right)^{2Nh+2j} }{(2Nh+2j)!}\right].
		\end{align}
	We now evaluate $\lim_{a\to2Nm-1+N}T_2^{*}$. To that end, use Gauss multiplication formula \eqref{gmf} in the numerator as well as the denominator of $T_2^{*}$, then replace $j$ by $j+1$, and use the fact that $\cos \left(\pi \left(\frac{j}{N}+ \frac{a-1}{2N} \right)\right)=(-1)^{m+1}\sin\left(\frac{\pi j}{n}\right)$ when $a=2Nm-1+N$, to find that
 \begin{align}\label{limit T2*}
 	\lim_{a\to2Nm-1+N}T_2^{*}=(-1)^{m+1}\sqrt{\frac{\pi}{
 	N}}\left(\frac{z}{2}\right)^{2\over N}\sum_{j=1}^{N-1}\frac{(-1)^j}{\sin\left( \pi j \over N\right) }\sum_{h=0}^{\infty} \frac{\left(  2N\left(\frac{z}{2}\right)^{\frac{1}{N}}\right)^{2Nh+2j} }{(2Nh+2j)!}.
 \end{align}
		
		Now letting $a\to2Nm-1+N$ in \eqref{Gen 2m 2}, inserting \eqref{limit T*} and \eqref{limit T2*} into it, and observing that the finite sums over $j$ completely cancel out, we arrive at
		\begin{align*}
			&\lim_{a\to 2Nm-1+N}G_{1, \, \, 2N+1}^{N+1, \, \, 1} \left(\begin{matrix}
				\frac{1}{2} + \frac{1-a}{2N}\\
				\frac{1}{2} + \frac{1-a}{2N}, \langle \frac{i}{N}\rangle; \langle 1+\frac{3}{2N}-\frac{i}{N} \rangle
			\end{matrix} \Bigg | \frac{z^2}{4} \right)\nonumber\\
		&=\frac{(-1)^m2\sqrt{\pi N}}{\pi (2N)^{2Nm}}\left(z^2\over 4 \right)^{-m+{1\over N}}\sum_{h=0}^{m-1} {(2N)^{2Nh}\left(z^2\over 4 \right)^h}\G(2Nm-2Nh)\nonumber\\
		&\quad+(-1)^m{2\over \sqrt{\pi N}}\left( z\over 2\right)^{2\over N}\sum_{k=-{\left(N-1 \right)\over 2 }}^{N-1\over 2 }\left( \sinh\operatorname{Shi}-\cosh\operatorname{Chi}\right) \left(   2N\left(\frac{z}{2}\right)^{\frac{1}{N}}e^{\pi i k \over N}\right)\nonumber\\
				&=\frac{2(-1)^m}{\sqrt{\pi}}\Bigg[\frac{1}{N^{3/2}}\left({2\pi^{N+1}n\over y}\right)^{\frac{2}{N}}\sum_{j=1}^{m} (2Nj-1)!\left({(2\pi)^{N+1}n\over y}\right)^{-2j}\nonumber\\
				&\quad+\frac{1}{N^{5/2}}\left({2\pi^{N+1}n\over y}\right)^{\frac{2}{N}}\sum_{k=-{\left(N-1 \right)\over 2 }}^{N-1\over 2 }\left( \sinh\operatorname{Shi}-\cosh\operatorname{Chi}\right) \left(  \left({(2\pi)^{N+1}n\over y}\right)^{1\over N}e^{\pi i k \over N}\right)\Bigg],
				\end{align*}
		where in the last step, we replaced $h$ by $m-j$ and used the fact that $z={4\pi^{N+1}n\over yN^{N}}$.
		Substituting the above limit evaluation in \eqref{Gen 2m 1}, we arrive at \eqref{resultt}.
			\end{proof}
		Letting $m=1$ in Theorem \ref{thm_resultt} leads to
		\begin{corollary}\label{resulttm=1cor}
			Let $N$ be an odd positive number. For a positive integer $m$,
			\begin{align*}
				&\sum_{n=1}^\infty \sigma_{3N-1}^{(N)}(n) e^{-ny}+ \frac{B_{2N}}{2Ny}	-\frac{2\z(3)}{Ny^{3}}\nonumber \\
				& =\frac{-16\pi^2}{y^3}\sum_{n=1}^{\infty} S_{3N-1}^{(N)}(n)\Bigg[ \frac{(2N-1)!y^2}{n^2(2\pi)^{2N+2}}+
				\frac{1}{N}\sum_{k=-{\left(N-1 \right)\over 2 }}^{N-1\over 2 }\left( \sinh\operatorname{Shi}-\cosh\operatorname{Chi}\right) \left(  \left({(2\pi)^{N+1}n\over y}\right)^{1\over N}e^{\pi i k \over N}\right)\Bigg].
			\end{align*}
		\end{corollary}
%		When $N=1$, the left-hand side of the above result consists of the series $\sum_{n=1}^\infty \sigma_{2}(n)e^{-ny}$ which is connected with the generating function for plane partitions considered by MacMahon \cite[p.~184]{gea}. Indeed, if $y=\log(1/q)$, we have
%		$\displaystyle\sum_{n=1}^\infty \sigma_{2}(n)e^{-ny}=q\frac{d}{dq}\log F(q)$, where $F(q):=\displaystyle\prod_{n=1}^{\infty}\frac{1}{(1-q^n)^n}$, the generating function of plane partitions. The form of the generating function at once shows that the number of plane partitions of a positive integer is equal to the number of partitions of that integer with ``$n$ copies of $n$''.  [ADD RELEVANT REFERENCES.]
%		
%		This now raises an important question - what is the infinite product connected with the more general series $\sum_{n=1}^\infty \sigma_{3N-1}^{(N)}(n) e^{-ny}$, and what partitions does it generate?
%		
%		To that end, note that 

\section{Asymptotics of generalized power partitions with ``$n^{2N-1}$ copies of $n^{N}$''}\label{asymptotics}
		
		\begin{proof}[Theorem \textup{\ref{Asym of Sum Sigma}}][]
	Let $-\frac{N-1}{2}\leq k \leq \frac{N-1}{2}$, where $N$ is an odd natural number, and $w$ be a positive real number to be chosen later. It follows from \cite[Theorem 1.10]{dk03} that
	\begin{align}\label{RNm expr}
		R_{Nm}\left(1, we^{\frac{\pi i k}{N}} \right)=(-1)^{Nm}\left(\left( \sinh\operatorname{Shi}-\cosh\operatorname{Chi}\right)\left(w e^{\frac{\pi i k }{N}} \right) + \sum_{j=1}^{Nm} \frac{(2j-1)!}{w^{2j}e^{\frac{2\pi i j k}{N}}}\right)
	\end{align}
	We also have from \cite[Lemma 2.1]{dk03} that for $w>0$,
	\begin{align}\label{RNm asymp}
		R_{Nm}\left(1, we^{\frac{\pi i k}{N}} \right) &\sim -\frac{\cos (\pi Nm)}{w^{2Nm} e^{2\pi ik m}} \sum_{j=1}^\infty \frac{\Gamma(2Nm+2j)}{w^{2j}e^{\frac{2\pi i j k}{N}}}\nonumber\\
		&= \frac{(-1)^{m+1}}{w^{2Nm}}\sum_{j=1}^\infty \frac{\Gamma(2Nm+2j)}{w^{2j}e^{\frac{2\pi i j k}{N}}}.
	\end{align}
	Thus, \eqref{RNm expr} and \eqref{RNm asymp} together imply
	\begin{align*}
%		&(-1)^{m}\frac{1}{N}\sum_{k=-\frac{N-1}{2}}^{\frac{N-1}{2}}\left(\left( \sinh\operatorname{Shi}-\cosh\operatorname{Chi}\right)\left(w e^{\frac{\pi i k }{N}} \right) + \sum_{j=1}^{Nm} \frac{(2j-1)!}{w^{2j}e^{\frac{2\pi i j k}{N}}}\right)\nonumber \\ 
%		& \sim \frac{1}{N}\sum_{k=-\frac{N-1}{2}}^{\frac{N-1}{2}} \frac{(-1)^{m+1}}{w^{2Nm}}\sum_{j=1}^\infty \frac{\Gamma(2Nm+2j)}{w^{2j}e^{\frac{2\pi i j k}{N}}}\nonumber\\
%		\implies 
		& \frac{1}{N}\sum_{k=-\frac{N-1}{2}}^{\frac{N-1}{2}} \left( \sinh\operatorname{Shi}-\cosh\operatorname{Chi}\right)\left(w e^{\frac{\pi i k }{N}} \right) + \sum_{j=1}^{Nm} \frac{(2j-1)!}{w^{2j}} \left(\frac{1}{N}\sum_{k=-\frac{N-1}{2}}^{\frac{N-1}{2}} e^{-\frac{2\pi i j k}{N}}\right) \nonumber\\
		& \sim -\frac{1}{w^{2Nm}} \sum_{j=1}^\infty \frac{\Gamma(2Nm+2j)}{w^{2j}}\left(\frac{1}{N}\sum_{k=-\frac{N-1}{2}}^{\frac{N-1}{2}} e^{-\frac{2\pi i j k}{N}}\right).
	\end{align*}
	We next employ \eqref{Orthogonal character0} on both sides of the above equation to obtain
	\begin{align*}
		\frac{1}{N}\sum_{k=-\frac{N-1}{2}}^{\frac{N-1}{2}} \left( \sinh\operatorname{Shi}-\cosh\operatorname{Chi}\right)\left(w e^{\frac{\pi i k }{N}} \right) + \sum_{\substack{j=1 \\ N\mid j}}^{Nm} \frac{(2j-1)!}{w^{2j}}  
		\sim -\frac{1}{w^{2Nm}} \sum_{\substack{j=1 \\ N\mid j}}^\infty \frac{\Gamma(2Nm+2j)}{w^{2j}}.
	\end{align*}
	Substitute $j$ by $Nj$ in the above equation to obtain
	\begin{align*}
		\frac{1}{N}\hspace{-2.5mm}\sum_{k=-\frac{N-1}{2}}^{\frac{N-1}{2}} \left( \sinh\operatorname{Shi}-\cosh\operatorname{Chi}\right)\left(w e^{\frac{\pi i k }{N}} \right) + \sum_{j=1}^{m} \frac{(2Nj-1)!}{w^{2N j}}  
		%&\sim -\frac{1}{w^{2Nm}} \sum_{j=1}^\infty \frac{\Gamma(2Nm+2Nj)}{w^{2Nj}}\nonumber\\
		= \frac{-1}{w^{2Nm}} \sum_{j=1}^{r+1} \frac{\Gamma(2Nm+2Nj)}{w^{2Nj}} + O\left(w^{-2Nm-2Nr-4N}\right).
	\end{align*}
	Thus, if we let $w = \left(\frac{(2\pi)^{N+1}n}{y} \right)^{1/N}$, we have
	\begin{align}\label{Asymptotic of sinhshi-coshchi}
		&\frac{1}{N}\sum_{k=-\frac{N-1}{2}}^{\frac{N-1}{2}} \left( \sinh\operatorname{Shi}-\cosh\operatorname{Chi}\right)\left(\left(\frac{(2\pi)^{N+1}n}{y} \right)^{\frac{1}{N}} e^{\frac{\pi i k }{N}} \right) + \sum_{j=1}^{m} \frac{(2Nj-1)! y^{2j}}{(2\pi)^{(N+1)(2j)}n^{2j}} \nonumber\\ 
		&= -\frac{y^{2m}}{(2\pi)^{(N+1)(2m)}n^{2m}} \sum_{j=1}^{r+1} \frac{\Gamma(2Nm+2Nj) }{(2\pi)^{(N+1)(2j)}n^{2j}} \, y^{2j}+ O\left(\frac{y^{2m+2r+4}}{n^{2m+2r+4}}\right).
	\end{align}
	Invoking \eqref{Asymptotic of sinhshi-coshchi} into Theorem \ref{thm_resultt}, we obtain
	\begin{align*}
		\sum_{n=1}^\infty \sigma_{2Nm-1+N}^{(N)}(n) e^{-ny}
		&=\frac{(2m)!\z(2m+1)}{Ny^{2m+1}} - \frac{B_{2Nm}}{2Nmy}	- \frac{2(-1)^m}{\pi}\left({2\pi \over y}\right)^{2m+1} \frac{y^{2m}}{\left((2\pi)^{N+1}\right)^{2m}} \nonumber\\
		&\quad \times \sum_{n=1}^{\infty} \frac{S_{2Nm-1+N}^{(N)}(n)}{n^{2m}}\sum_{j=1}^{r+1} \frac{\Gamma(2Nm+2Nj) }{(2\pi)^{(N+1)(2j)}n^{2j}} \, y^{2j}+ O\left(y^{2m+2r+4}\right)
		\nonumber\\
%		&= \frac{(2m)!\z(2m+1)}{Ny^{2m+1}} - \frac{B_{2Nm}}{2Nmy}	- \frac{4(-1)^m}{(2\pi)^{2Nm}y} \sum_{j=1}^{r+1} \frac{\Gamma(2Nm+2Nj) }{(2\pi)^{(N+1)(2j)}n^{2j}} \, y^{2j} \sum_{n=1}^{\infty} \frac{S_{2Nm-1+N}^{(N)}(n)}{n^{2m+2j}} \nonumber\\
%		&\quad +O\left(y^{2r+3}\right)
&= \frac{(2m)!\z(2m+1)}{Ny^{2m+1}} - \frac{B_{2Nm}}{2Nmy} - \frac{4(-1)^m}{(2\pi)^{2Nm}} \sum_{j=1}^{r+1} \frac{\Gamma(2Nm+2Nj) \zeta(2Nm+2Nj) \zeta(2j)}{(2\pi)^{(N+1)(2j)}}y^{2j-1}\nonumber\\
&\quad+O\left(y^{2r+3}\right),
	\end{align*}
where in the last step, we used \eqref{Sdirichlet}.
%	We next apply the fact that for $\min \{\re (2Nm+2Nj), \re (2j) \} >1$,
%	\begin{equation}
%		\sum_{n=1}^{\infty} \frac{S_{2Nm-1+N}^{(N)}(n)}{n^{2m+2j}} = \zeta(2Nm+2Nj) \zeta(2j),
%	\end{equation}
%	in the above equation to conclude
%	\begin{align*}
%		&\sum_{n=1}^\infty \sigma_{2Nm-1+N}^{(N)}(n) e^{-ny} \nonumber\\
%		&= \frac{(2m)!\z(2m+1)}{Ny^{2m+1}} - \frac{B_{2Nm}}{2Nmy} - \frac{4(-1)^m}{(2\pi)^{2Nm}} \sum_{j=1}^{r+1} \frac{\Gamma(2Nm+2Nj) \zeta(2Nm+2Nj) \zeta(2j)}{(2\pi)^{(N+1)(2j)}} \, y^{2j-1} +O\left(y^{2r+3}\right).
%	\end{align*}
	This completes the proof.
	\end{proof}
We now obtain the asymptotic behavior of the infinite product $F_N(q)$ defined in \eqref{ppN}.
\begin{proof}[Corollary \textup{\ref{asym F_N(q)}}][]
Letting $m=1$, $y = \log \left(\frac{1}{q} \right), |q|<1$ in Theorem \ref{Asym of Sum Sigma} and using \eqref{gpp3N-1}, we obtain
\begin{align*}
	q \frac{d}{dq}\left( \log F(q) \right) = - \frac{2 \zeta(3)}{N(\log q)^3} + \frac{B_{2N}}{2N \log q} 
	&- \frac{4}{(2\pi)^{2N}} \sum_{j=1}^{r+1} \frac{\Gamma(2N + 2Nj) \zeta(2N + 2Nj) \zeta(2j)}{\left((2\pi)^{N+1} \right)^{2j}}(\log q)^{2j-1} 
	\nonumber\\
	&\hspace{5cm}+ O\left(- (\log q)^{2r+3} \right).
\end{align*}
Next divide both sides of the above equation by $q$ and then integrate with respect to $q$ so as to get
\begin{align*}
	\log F(q) &= c + \frac{\zeta(3)}{N(\log q)^2} + \frac{B_{2N}}{2N}\log \log q 
	- \frac{4}{(2\pi)^{2N}} \sum_{j=1}^{r+1} \frac{\Gamma(2N + 2Nj) \zeta(2N + 2Nj) \zeta(2j)}{2j \left((2\pi)^{N+1} \right)^{2j}}(\log q)^{2j}
	\nonumber\\
	&\quad+ O\left((\log q)^{2r+4} \right).
\end{align*}
Finally exponentiating both sides of the above equation, we are led to the desired estimate.
	\end{proof}
		\section{Special cases of Theorem \ref{Analytic continuation} for $N$ even}\label{scace}
		
			\subsection{A generalization of Wigert's formula} In this subsection, we show that the following formula, which was first obtained in \cite[Theorem 1.5]{dixitmaji1}, can be derived from Theorem \ref{Analytic continuation}.
		
		\begin{corollary}\label{neven}
			Let $N$ be an even positive integer and $m$ be \emph{any} integer. For any $\a,\b>0$ satisfying $\a\b^N=\pi^{N+1}$,
			\begin{align}\label{neveneqn}
				&\a^{-\left(\frac{2Nm-1}{N+1}\right)}\left(\frac{1}{2}\zeta(2Nm)+\sum_{n=1}^{\infty}\frac{n^{-2Nm}}{\exp{\left((2n)^{N}\a\right)}-1}\right)\nonumber\\
				&=\b^{-\left(\frac{2Nm-1}{N+1}\right)}\frac{(-1)^m}{N}2^{(N-1)\left(2m-\frac{1}{N}\right)}\Bigg(\frac{\zeta\left(2m+1-\frac{1}{N}\right)}{2\cos\left(\frac{\pi}{2N}\right)}\nonumber\\
				&\quad-2(-1)^{\frac{N}{2}}\sum_{j=0}^{\frac{N}{2}-1}(-1)^j \sum_{n=1}^{\infty}\frac{1}{n^{2m+1-\frac{1}{N}}}\im\Bigg(\frac{e^{\frac{i\pi(2j+1)}{2N}}}{\exp{\left((2n)^{\frac{1}{N}}\b e^{\frac{i\pi(2j+1)}{2N}}\right)}-1}\Bigg)\Bigg)\nonumber\\
				&\quad+(-1)^{\frac{N}{2}+1}2^{2Nm-1}\sum_{j=0}^{m}\frac{B_{2j}B_{(2m+1-2j)N}}{(2j)!((2m+1-2j)N)!}\a^{\frac{2j}{N+1}}\b^{N+\frac{2N^2(m-j)-N}{N+1}}.
			\end{align}
		\end{corollary}
		\begin{proof}
			We will use the fact without mention that $N$ is an even positive integer. Letting $a\to-2Nm, m\in\mathbb{Z}$, in Theorem \ref{Analytic continuation}, we get
			\begin{align}\label{Eqn:Analytic continuationwig}
				&\sum_{n=1}^\infty \sigma_{-2Nm}^{(N)}(n) e^{-ny}  + \frac{1}{2}\zeta(2Nm) - \frac{1}{y}\zeta(N+2Nm) - \frac{y^{2m-1/N}}{N} \Gamma \left(\frac{1}{N}-2m\right) \zeta \left(\frac{1}{N}-2m\right)\nonumber\\
				&= \hspace{-5pt}\lim_{a\to-2Nm}\tfrac{2 (2\pi)^{\frac{1}{N}-\frac{1}{2}} N^{\frac{a-1}{2}}}{y^{\frac{1}{N}+\frac{a}{2N}}} \sum_{n=1}^\infty \tfrac{S_a^{(N)}(n)}{n^{\frac{a}{2N}}} \bigg [{}_{\frac{1}{2}} K_{\frac{a}{2N}}^{(N)}\left(\substack{\tfrac{4\pi^{N+1}n}{yN^N}, 0}\right) - \tfrac{2^{\frac{1}{2} + \frac{a+1}{N}} \pi^{\frac{(1-N)a}{2N} -N}} {\left(\tfrac{4\pi^{N+1}n}{yN^N}\right)^{1+\frac{1}{N}+\frac{a}{2N}}}   \tfrac{\sin\(\frac{\pi}{2}(N-a)\)}{2^{N-1}}C_{m,N}\left(\substack{\tfrac{1}{2}, \tfrac{a}{2N}, 0, \tfrac{4\pi^{N+1}n}{yN^N}}\right)\bigg]\nonumber\\
				&\quad+ \lim_{a\to-2Nm}\frac{y}{2\pi^{2}} \sum_{k=0}^m \left(-\frac{y^2}{4\pi^2}\right)^k   \zeta(-2kN-N-a)\zeta(2k+2).
			\end{align}
			The details are curbed for the part of the proof which is similar to that of Theorem \ref{zetagen3} whereas we provide all of the details for the part involving new analysis. 
			
			We begin with the simplification of the first expression in the square brackets on the right-hand side of \eqref{Eqn:Analytic continuationwig}. To that end, using \eqref{muknug} and Theorem \ref{meijergsim}, we have
			\begin{align}\label{Kwig}
				{}_{\frac{1}{2}} K_{\frac{a}{2N}}^{(N)}\left(\tfrac{4\pi^{N+1}n}{yN^N}, 0\right)& =\frac{(2\pi)^{a-2\over N}}{\pi^2 \sqrt{2}}\frac{(n/y)^{{a\over 2N}-{2\over N}}}{N^{{a\over 2}-2}} \Bigg(\frac{ N^{N-a - \frac{1}{2}} \Gamma\left(\frac{1-N+a}{2}\right)\left(\frac{4\pi^{2N+2}n^2}{y^2N^{2N}}\right)^{\frac{1}{2}+\frac{1-a}{2N}} }{\Gamma\left(\frac{N-a}{2}\right)}\nonumber\\
				&\quad\times{}_1F_{2N}\left( \begin{matrix}
					1 \\ \left\langle \frac{1}{2} - \frac{a+1}{2N} + \frac{i}{2N} \right\rangle_{i=1}^{2N}
				\end{matrix}\right| -\frac{4\pi^{2N+2}n^2}{y^2N^{2N}} \Bigg)+ A_{a, N}\left(\frac{4\pi^{2N+2}n^2}{y^2N^{2N}}\right)\Bigg),
			\end{align}
			where $A_{a, N}(z)$ is defined in \eqref{aanzeven}. 
			
			Letting $a\to-2Nm$ and employing \eqref{prudnikovid}, we see that
			\begin{align*}
				&{}_1F_{2N}\left(\left. \begin{matrix}
					1 \\ \left\langle \frac{1}{2}+m + \frac{i-1}{2N} \right\rangle_{i=1}^{2N}
				\end{matrix}\right| -\frac{4\pi^{2N+2}n^2}{y^2N^{2N}}\right)\nonumber\\
				&=
				\frac{(2N)^{-2Nm-N}(2Nm+N-1)!}{\left(- \frac{4\pi^{2N+2}n^2}{y^2N^{2N}}\right)^{{1\over 2}+m-{1 \over 2N}} }\left\{\sum_{k=0}^{2N-1}{\frac{\exp \left\lbrace 2 \pi e^{\pi i k \over N}\left(\frac{-4\pi^2 n^2}{y^2} \right)^{1\over 2N}\right\rbrace}{\exp \left\lbrace  {\pi i k \over N}\left(2Nm+N-1 \right)\right\rbrace} }- \sum_{k=1}^{m}\frac{(2N)^{2Nm+N}\left(-\frac{4\pi^{2N+2}n^2}{y^2N^{2N}} \right)^{m+{1\over 2}-{1\over 2N}-k}}{(2N)^{2Nk}(2Nm-2Nk+N-1)!} \right\}.  
			\end{align*}
			Substituting the above expression in \eqref{Kwig}, we get
			\begin{align*}
				{}_{\frac{1}{2}} K_{-m}^{(N)}\left(\tfrac{4\pi^{N+1}n}{yN^N}, 0\right)& =   \frac{N^{Nm+2}}{\pi^2\sqrt{2}(2\pi)^{m}\left(\frac{2\pi n}{y}\right)^{m+\frac{2}{N}}} \Bigg[A_{-2Nm, N}\left(\frac{4\pi^{2N+2}n^2}{y^2N^{2N}}\right)+{1\over 2}\left( \pi \over N \right)^{5\over 2}  (-1)^{{1\over 2N}-m-{1\over 2}-{N\over 2}}\left( 2\pi n \over y \right)^{2\over N}\nonumber\\
				&\times\left\{\sum_{k=0}^{2N-1}(-1)^k\exp \left\lbrace \substack{2 \pi e^{\pi i k \over N}\left(\frac{-4\pi^2 n^2}{y^2} \right)^{1\over 2N}+{i\pi k\over N}}\right\rbrace  - \sum_{k=1}^{m}\frac{(2N)^{2Nm+N}\left(-\frac{4\pi^{2N+2}n^2}{y^2N^{2N}} \right)^{m+{1\over 2}-{1\over 2N}-k}}{(2N)^{2Nk}(2Nm-2Nk+N-1)!} \right\} \Bigg].
			\end{align*}
			Therefore, using \eqref{aanzeven} and the fact that $(-1)^{\frac{1}{2N}}=e^{\frac{i\pi}{2N}}$, we find that
			\begin{align*}
				&A_{-2Nm, N}\left(\frac{4\pi^{2N+2}n^2}{y^2N^{2N}}\right)+{1\over 2}\left( \pi \over N \right)^{5\over 2}  (-1)^{{1\over 2N}-m-{1\over 2}-{N\over 2}}\left( 2\pi n \over y \right)^{2\over N}
				\sum_{k=0}^{2N-1}(-1)^k\exp \left\lbrace 2 \pi e^{\pi i k \over N}\left(\frac{-4\pi^2 n^2}{y^2} \right)^{1\over 2N}+{i\pi k\over N}\right\rbrace\nonumber\\
				&=\frac{1}{2}(-1)^{m+\frac{N+1}{2}}\left(\frac{\pi}{N}\right)^{5/2}\left(\frac{2\pi n}{y}\right)^{2/N}\left\{\left[\sum_{k=0}^{\frac{N}{2}-1}-\sum_{k=\frac{N}{2}}^{\frac{3N}{2}-1}+\sum_{k=\frac{3N}{2}}^{2N-1}\right](-1)^k{\exp \left\lbrace 2 \pi e^{\pi i (2k+1) \over 2N}\left(\frac{2\pi n}{y} \right)^{1\over N}+\frac{\pi i(2k+1)}{2N}\right\rbrace }\right.\nonumber\\
				&\quad\left.-\left[\sum_{k=0}^{\frac{N}{2}-1}+\sum_{k=\frac{N}{2}}^{\frac{3N}{2}-1}+\sum_{k=\frac{3N}{2}}^{2N-1}\right](-1)^k{\exp \left\lbrace 2 \pi e^{\pi i (2k+1) \over 2N}\left(\frac{2\pi n}{y} \right)^{1\over N}+\frac{\pi i(2k+1)}{2N}\right\rbrace }\right\}\nonumber\\
				&=-(-1)^{m+\frac{N+1}{2}}\left(\frac{\pi}{N}\right)^{5/2}\left(\frac{2\pi n}{y}\right)^{2/N}\sum_{k=\frac{N}{2}}^{\frac{3N}{2}-1}(-1)^k{\exp \left\lbrace 2 \pi e^{\pi i (2k+1) \over 2N}\left(\frac{2\pi n}{y} \right)^{1\over N}+\frac{\pi i(2k+1)}{2N}\right\rbrace}\nonumber\\
				&=(-1)^{m+\frac{N+1}{2}}\left(\frac{\pi}{N}\right)^{5/2}\left(\frac{2\pi n}{y}\right)^{2/N}\sum_{k=-\frac{N}{2}}^{\frac{N}{2}-1}(-1)^k{\exp \left\lbrace -2 \pi e^{\pi i (2k+1) \over 2N}\left(\frac{2\pi n}{y} \right)^{1\over N}+\frac{\pi i(2k+1)}{2N}\right\rbrace},
			\end{align*}
			where in the last step, we replaced $k$ by $k+N$.
			\begin{align*}
				{}_{\frac{1}{2}} K_{-m}^{(N)}\left(\tfrac{4\pi^{N+1}n}{yN^N}, 0\right) 
				&=(-1)^{m+\frac{N+1}{2}}\sqrt{\frac{\pi}{2}}N^{mN-\frac{1}{2}}\left(\frac{4\pi^2n}{y}\right)^{-m}\sum_{k=-\frac{N}{2}}^{\frac{N}{2}-1}(-1)^k{\exp \left\lbrace -2 \pi e^{\pi i (2k+1) \over 2N}\left(\tfrac{2\pi n}{y} \right)^{1\over N}+\tfrac{\pi i(2k+1)}{2N}\right\rbrace}\nonumber\\
				&\quad+(-1)^{\frac{N}{2}}\pi N^{mN+\frac{1}{2}}(2\pi)^{2Nm+N-\frac{1}{N}-\frac{1}{2}}\left(\frac{n}{y}\right)^{m+1-\frac{1}{N}}\sum_{k=1}^{m}\frac{\left(-\frac{y^2}{(2\pi)^{2N+2}n^2}\right)^{k}}{(2Nm-2Nk+N-1)!}.
			\end{align*}
			Next, proceeding similarly as in \eqref{cancel}, we see that
			\begin{align*}
				&\lim_{a\to-2Nm}\frac{2^{\frac{1}{2} + \frac{a+1}{N}} \pi^{\frac{(1-N)a}{2N} -N}} {\left(\frac{4\pi^{N+1}n}{yN^N}\right)^{1+\frac{1}{N}+\frac{a}{2N}}}   \frac{\sin\(\frac{\pi}{2}(N-a)\)}{2^{N-1}}C_{m,N}\left(\frac{1}{2}, \frac{a}{2N}, 0, \frac{4\pi^{N+1}n}{yN^N}\right)\nonumber\\    
				&=-(-1)^{\frac{N}{2}}\pi N^{mN+\frac{1}{2}}(2\pi)^{2Nm+N-\frac{1}{N}-\frac{1}{2}}\left(\frac{n}{y}\right)^{m+1-\frac{1}{N}}\sum_{k=1}^{m}\frac{\left(-\frac{y^2}{(2\pi)^{2N+2}n^2}\right)^{k}}{(2Nm-2Nk+N-1)!}.
			\end{align*}
			Therefore, the expression inside the square brackets on the right-hand side of \eqref{Eqn:Analytic continuationwig} simplifies in the limit to
			\begin{align}\label{limitfourtheven}
				&\lim_{a\to-2Nm}\left[{}_{\frac{1}{2}} K_{\frac{a}{2N}}^{(N)}\left(\tfrac{4\pi^{N+1}n}{yN^N}, 0\right) - \frac{2^{\frac{1}{2} + \frac{a+1}{N}} \pi^{\frac{(1-N)a}{2N} -N}} {\left(\frac{4\pi^{N+1}n}{yN^N}\right)^{1+\frac{1}{N}+\frac{a}{2N}}}   \frac{\sin\(\frac{\pi}{2}(N-a)\)}{2^{N-1}}C_{m,N}\left(\frac{1}{2}, \frac{a}{2N}, 0, \frac{4\pi^{N+1}n}{yN^N}\right)\right]\nonumber\\
				&=(-1)^{m+\frac{N+1}{2}}\sqrt{\frac{\pi}{2}}N^{mN-\frac{1}{2}}\left(\frac{4\pi^2n}{y}\right)^{-m}\sum_{k=-\frac{N}{2}}^{\frac{N}{2}-1}(-1)^k{\exp \left( -2 \pi e^{\pi i (2k+1) \over 2N}\left(\frac{2\pi n}{y} \right)^{1\over N}+\frac{\pi i(2k+1)}{2N}\right)}.
			\end{align}
			From \eqref{limitfourtheven}, the definition of $S_{a}^{(N)}(n)$ in \eqref{defbf}, and interchanging the order of limit and summation (since the summand of the series below is $\mathcal{O}\left(n^{\max \left(-\frac{\re(a)}{2N}, \frac{\re(a)}{2N} + \frac{1}{N} - 1 \right)-2m-3-\frac{1}{N}-\frac{\re(a)}{2N}+\epsilon} \right)$ for any $\epsilon>0$), we see that
			\begin{align}\label{fourthtermeven}
				&\lim_{a\to-2Nm}\hspace{-5pt}\tfrac{2 (2\pi)^{\frac{1}{N}-\frac{1}{2}} N^{\frac{a-1}{2}}}{y^{\frac{1}{N}+\frac{a}{2N}}} \sum_{n=1}^\infty \tfrac{S_a^{(N)}(n)}{n^{\frac{a}{2N}}} \bigg [{}_{\frac{1}{2}} K_{\frac{a}{2N}}^{(N)}\left(\substack{\tfrac{4\pi^{N+1}n}{yN^N}, 0}\right) - \tfrac{2^{\frac{1}{2} + \frac{a+1}{N}} \pi^{\frac{(1-N)a}{2N} -N}} {\left(\tfrac{4\pi^{N+1}n}{yN^N}\right)^{1+\frac{1}{N}+\frac{a}{2N}}}   \tfrac{\sin\(\frac{\pi}{2}(N-a)\)}{2^{N-1}}C_{m,N}\left(\substack{\tfrac{1}{2}, \tfrac{a}{2N}, 0, \tfrac{4\pi^{N+1}n}{yN^N}}\right)\bigg]\nonumber\\
				&=\frac{1}{N}\left( y\over  2\pi\right)^{2m-{1\over N}}(-1)^{m+\frac{N+1}{2}}\sum_{d_2=1}^{\infty} {d_2}^{{1\over N}-1-2m}\sum_{k=-\frac{N}{2}}^{\frac{N}{2}-1}(-1)^ke^{\frac{\pi i(2k+1)}{2N}}\sum_{d_1=1}^{\infty }{\exp \left( -2 \pi d_1e^{\pi i (2k+1) \over 2N}\left(\frac{2\pi d_2}{y} \right)^{1\over N}\right)}\nonumber\\
				&=\frac{1}{N}\left( y\over  2\pi\right)^{2m-{1\over N}}(-1)^{m+\frac{N+1}{2}}\sum_{d_2=1}^{\infty} {d_2}^{{1\over N}-1-2m}\left[ \sum_{k=-{N\over 2}}^{-1}+\sum_{k=0}^{{N\over 2}-1}\right]\frac{(-1)^ke^{\frac{\pi i(2k+1)}{2N}}}{\exp \left( 2 \pi e^{\pi i (2k+1) \over 2N}\left(\frac{2\pi d_2}{y} \right)^{1\over N}\right)-1}\nonumber\\
				&=\frac{1}{N}\left( y\over  2\pi\right)^{2m-{1\over N}}(-1)^{m+\frac{N+1}{2}}\sum_{n=1}^{\infty} {n}^{{1\over N}-1-2m}\sum_{k=0}^{{N\over 2}-1} (-1)^k \textup{Im} \left( e^{i\pi (2k+1) \over 2N}\over \exp \left( 2 \pi e^{\pi i (2k+1) \over 2N}\left(\frac{2\pi {n}}{y} \right)^{1\over N}\right)-1 \right),
			\end{align}
			where, in the penultimate step, we evaluated the sum over $d_1$ using the geometric series formula, since $-{N\over 2} \leq k \leq {N\over 2} -{1} $ implies $-{\pi \over 2}<\arg \left(e^{\pi i (2k+1) \over 2N}\left(\frac{2\pi {n}}{y} \right)^{1\over N} \right) <{\pi \over 2} $, and in the ultimate step, we replaced $k$ by $-k-1$ in the first finite sum. Finally, we consider the  finite sum on the right-hand side of \eqref{Eqn:Analytic continuationwig}. Upon letting $a\to-2Nm$ in it, the $k=m$ term vanishes since for even $N$, $\zeta(-N)=0$. Therefore, substituting $k$ by $k-1$, we obtain \begin{align}\label{fifthtermeven}
				\lim_{a\to-2Nm}\frac{y}{2\pi^{2}} \sum_{k=0}^m \left(-\frac{y^2}{4\pi^2}\right)^k   \zeta(-2kN-N-a)\zeta(2k+2)=- \frac{2}{y}\sum_{k=1}^{m} (-1)^k \zeta(2N(m-k)+N)\zeta (2k)\left({2\pi \over y}\right)^{-2k}.
			\end{align}
			Thus, from \eqref{Eqn:Analytic continuationwig}, \eqref{fourthtermeven} and \eqref{fifthtermeven}, we see that
			\begin{align*}
				&\sum_{n=1}^\infty \sigma_{-2Nm}^{(N)}(n) e^{-ny}  + \frac{1}{2}\zeta(2Nm) - \frac{1}{y}\zeta(N+2Nm) - \frac{y^{2m-1/N}}{N} \Gamma \left(\frac{1}{N}-2m\right) \zeta \left(\frac{1}{N}-2m\right)\nonumber\\
				&=\frac{1}{N}\left( y\over  2\pi\right)^{2m-{1\over N}}(-1)^{m+\frac{N+1}{2}}\sum_{n=1}^{\infty} {n}^{{1\over N}-1-2m}\sum_{k=0}^{{N\over 2}-1} (-1)^k \textup{Im} \left( e^{i\pi (2k+1) \over 2N}\over \exp \left( 2 \pi e^{\pi i (2k+1) \over 2N}\left(\frac{2\pi {n}}{y} \right)^{1\over N}\right)-1 \right)
				\nonumber\\
				&\quad- \frac{2}{y}\sum_{k=1}^{m} (-1)^k \zeta(2N(m-k)+N)\zeta (2k)\left({2\pi \over y}\right)^{-2k}.
			\end{align*}
			Using the functional equation \eqref{zetafe} to rewrite $\Gamma \left(\frac{1}{N}-2m\right) \zeta \left(\frac{1}{N}-2m\right)$, and inducting $\frac{1}{y}\zeta(N+2Nm)$ into the second finite sum over $k$ as its $k=0$ term, we arrive at
			\begin{align*}
				&\sum_{n=1}^\infty \sigma_{-2Nm}^{(N)}(n) e^{-ny}  + \frac{1}{2}\zeta(2Nm)\nonumber\\
				&=\frac{(-1)^m}{N}\left(\frac{y}{2\pi}\right)^{2m-\frac{1}{N}}\Bigg[\frac{\zeta\left(2m+1-\frac{1}{N}\right)}{2\cos\left(\frac{\pi}{2N}\right)}
				(-1)^{\frac{N+1}{2}}\sum_{n=1}^{\infty} {n}^{{1\over N}-1-2m}\sum_{k=0}^{{N\over 2}-1} (-1)^k\textup{Im} \left( \substack{e^{i\pi (2k+1) \over 2N}\over \exp \left( 2 \pi e^{\pi i (2k+1) \over 2N}\left(\frac{2\pi {n}}{y} \right)^{1\over N}\right)-1 }\right)\Bigg]\nonumber\\
				&\quad- \frac{2}{y}\sum_{k=0}^{m} (-1)^k \zeta(2N(m-k)+N)\zeta (2k)\left(\tfrac{2\pi}{y}\right)^{-2k}.
			\end{align*}
			Letting $\alpha={y\over 2^N }$ and $\beta = \frac{2\pi^{1+{1\over N}}}{y^{1\over N}}$ and multiplying both sides of the above equation by $\a^{-\left(\frac{2Nm-1}{N+1}\right)}$ and using Euler's formula \eqref{ef} and \eqref{fkny}, and simplifying, we conclude that \eqref{neveneqn} holds.
		\end{proof}

		\subsection{A transformation for the series $\sum_{n=1}^{\infty}\sigma_{-2Nm+N}^{(N)}(n)e^{-ny}$}\label{8.2}
		Let $N$ be an even positive integer and $m$ be any integer.
		Differentiating both sides of the identity in \cite[Theorem 2.12]{DGKM} with respect to $a$, then letting $a=1$, and using the facts $\left.\frac{\partial}{\partial a}\zeta(s, a)\right|_{a=1}=-s\zeta(s+1)$ and $\a \b^N=\pi ^{N+1}$, we get 
		\begin{align}\label{case-2Nm+N}
			{\a}^{-\left( \frac{2Nm-1}{N+1}\right) }&\left( 
			\zeta \left( 2Nm\right) +\sum ^{m}_{j=1}
			\frac{B_{2j}}{\left( 2j\right) !}
			\zeta \left( 2N\left( m-j\right) \right) 
			\left( 2^{N}\alpha \right) ^{2j}-2^{N}\alpha \sum ^{\infty }_{n=1}
			\frac{n^{-2Nm+N}}{e^{\left( 2n\right) ^{N}\alpha }-1} \right) \nonumber \\
			&= \b^{-\left( \frac{2Nm-1}{N+1}\right) }
			\frac{2^{2Nm-1}}{N} 
			\left\{ -\left( \frac{1-2Nm}{N}\right)  \pi^{-\left( \frac{1-2Nm}{N}\right)} \Gamma \left( \frac{1-2Nm}{N}\right) \zeta \left( \frac{1-2Nm}{N}+1\right) \right. 
			\nonumber \\
			& \qquad \left.-4 \pi  \left( -1\right) ^{m+1} 
			2^\frac{1-2Nm}{N}\sum ^{\frac{N}{2}-1}_{j=0}
			\sum ^{\infty }_{n=1}\frac{1}{n^{2m-\frac{1}{N}}}
			\re \left( \frac{e^\frac{i\pi  \left( 2j+1\right) }{2N}}{\exp\left((2n)^{1\over N}\b e^\frac{i\pi  \left( 2j+1\right) }{2N} \right)-1 } \right)\right\}  \nonumber \\
			& \qquad+\frac{\left(-1\right)^{\frac{N}{2}+1}}{2}
			2^{2Nm-1} \frac{B_{\left( 2m-1\right) N}}{\left( \left( 2m-1\right) N\right) !}
			\alpha ^{\frac{2}{N+1}}\beta ^{N+\frac{2N^{2}\left( m-1\right) -N}{N+1}}. 
		\end{align}
%		It gives a special case $a=N$ for an even $N$. Let $m=0$ in \eqref{case-2Nm+N}, we have 
%		\begin{align}
%			{\a}^{ \frac{1}{N+1} }\left( 
%			{-1\over 2} +2^N\a\sum ^{\infty }_{n=1}
%			\frac{n^{N}}{e^{\left( 2n\right) ^{N}\alpha }-1} \right) &= 
%			\frac{\b^{ \frac{1}{N+1} }}{2N} 
%			\left\{ - \frac{1}{N}  \pi^{-\left( \frac{1}{N}\right)} \Gamma \left( \frac{1}{N}\right) \zeta \left( \frac{1}{N}+1\right) \right. 
%			\nonumber \\
%			& \qquad \left. +\pi   
%			2^{\frac{1}{N}+ 2}\sum ^{\frac{N}{2}-1}_{j=0}
%			\sum ^{\infty }_{n=1}{n^{\frac{1}{N}}}
%			\re \left( \frac{e^\frac{i\pi  \left( 2j+1\right) }{2N}}{\exp\left((2n)^{1\over N}\b e^\frac{i\pi  \left( 2j+1\right) }{2N} \right)-1 } \right)\right\} 
%		\end{align}
One may also prove this identity by letting $a=-2Nm+N$ ($N$ even) in Theorem \ref{Analytic continuation}. We leave this as an exercise.
	\section{The generalized Dedekind eta-transformation}\label{zagierjrh}
	
As mentioned in the introduction, Tenenbaum, Wu and Li \cite{tenenbaum} obtained a new proof of the asymptotic expansion of $p_k(n)$, the number of power partitions of $n$, as $n\to\infty$ using the saddle point method. In this proof, they encountered the series $\sum_{n=1}^{\infty}\sigma_{N}^{(N)}(n)e^{-ny}$. We now prove the exact transformation for this series given in Theorem \ref{zagiereqvt} and then show that it is equivalent to the generalized Dedekind eta-transformation given in \eqref{Zagier}.

\begin{proof}[Theorem \textup{\ref{zagiereqvt}}][]
For $N$ odd, let $m=-1$ in Theorem \ref{zetagen3}. For $N$ even, let $m=0$ in \eqref{case-2Nm+N} above, use the fact Re$(z)=\frac{1}{2}(z+\overline{z})$, separate the two sums over $j$ and replace $j$ by $-1-j$ in the second sum. Then, in both of the resulting identities, that is, the ones for $N$ odd and $N$ even, substitute $y=2^{N}\alpha$ to arrive at \eqref{zagier to us}. 
\end{proof}

%\subsection{Equivalence of Theorem \ref{zagiereqvt} and \eqref{Zagier}}

	We now derive Theorem \ref{zagiereqvt} from \eqref{Zagier}.
	\begin{proof}
		We prove the result only for odd $N$. The case for $N$ even can be similarly proved.
		
		Letting $z=iy$, Re$(y)>0$, in \eqref{Zagier}, we have
		\begin{equation}\label{Zagier 1}
			\eta_N\left( {i\over y}\right) =(2\pi)^{(N-1)/2}\sqrt{y}\prod_{w\in\mathbb{H}\atop{w^N=\pm iy}}\eta_{1/N}(w),
		\end{equation}
		whereas taking logarithm of both sides of \eqref{Zagier eta} with $s=N$ and $z=iy$ with Re$(y)>0$ results in
		\begin{equation}\label{log eta}
			\log \eta_N(iy):=\pi \zeta(-N)y-\sum_{n=1}^{\infty}\sigma_N^{(N)}(n){e^{-2\pi n y}\over n}.
		\end{equation}
		Taking logarithm on both sides of \eqref{Zagier 1} leads to
		\begin{align}\label{expn Zag}
			\log \left( \eta_N\left( {i\over y}\right)\right) ={1\over 2}(N-1)\log (2\pi)+{1\over 2}\log( y)+\sum_{w\in\mathbb{H}\atop{w^N=iy}}\log\eta_{1/N}(w)+\sum_{w\in\mathbb{H}\atop{w^N=-iy}}\log\eta_{1/N}(w).
		\end{align}
		Now consider the third expression on the right-hand side of \eqref{expn Zag}. %Let $w=r_1e^{i\theta}, \  
		For odd $N$, $w^N=iy$ and $w\in \mathbb{H}$ implies $w=y^{1\over N}e^{(4k+1)\pi i \over 2N}$, where $0\leq k\leq{N-1\over 2}$. Similarly, $w$ in the fourth expression is given by  $w=y^{1\over N}e^{(4k-1)\pi i \over 2N}$, where $1\leq k\leq{N-1\over 2}$. Hence using \eqref{log eta} and the definition of $\sigma_{N}^{(N)}(n)$, we observe that
		\begin{align*}
			\sum_{w\in\mathbb{H}\atop{w^N=iy}}&\eta_{1/N}(w)+\sum_{w\in\mathbb{H}\atop{w^N=-iy}}\eta_{1/N}(w)\nonumber
			\\
			&=\sum_{k=0}^{N-1\over 2}\left[-i \pi \zeta\left(-\frac{1}{N}\right) y^{1\over N}e^{(4k+1)\pi i \over 2N}-\sum_{n=1}^{\infty}\sum_{m=1}^{\infty}\frac{1}{m}\exp\left( 2\pi i n^{1\over N}my^{1\over N}e^{(4k+1)\pi i \over 2N}\right) \right]\nonumber
			\\
			&\quad+\sum_{k=1}^{N-1\over 2}\left[-i \pi \zeta\left(-\frac{1}{N}\right) y^{1\over N}e^{(4k-1)\pi i \over 2N}-\sum_{n=1}^{\infty}\sum_{m=1}^{\infty}\frac{1}{m} \exp\left( 2\pi i n^{1\over N}my^{1\over N}e^{(4k-1)\pi i \over 2N}\right)\right] \nonumber
			\\
			&=-\sum_{k=0}^{N-1} i\pi \zeta\left(-\frac{1}{N}\right)  y^{1\over N}e^{(2k+1)\pi i \over 2N}-\sum_{k=0}^{N-1}\sum_{n=1}^{\infty}\sum_{m=1}^{\infty}\frac{1}{m}\exp\left( 2\pi i m n^{1\over N}y^{1\over N}e^{(2k+1)\pi i \over 2N}\right). 
		\end{align*}
		Inserting \eqref{log eta} with $y$ replaced by $1/y$, and the above equation in \eqref{expn Zag}, we deduce an equivalent form of Zagier's identity, that is,
		\begin{align*}
			{\pi \zeta(-N)\over y}-\sum_{n=1}^{\infty}\sigma_N^{(N)}(n){e^{-2\pi n / y}\over n}&={1\over 2}(N-1)\log (2\pi)+{1\over 2}\log( y)- i\pi \zeta\left(\frac{-1}{N}\right) y^{1\over N}\sum_{k=0}^{N-1}e^{(2k+1)\pi i \over 2N}\nonumber
			\\
			&\qquad-\sum_{k=0}^{N-1}\sum_{n=1}^{\infty}\sum_{m=1}^{\infty}\frac{1}{m}\exp\left( 2\pi i m n^{1\over N}y^{1\over N}e^{(2k+1)\pi i \over 2N}\right).
		\end{align*} 
		Differentiating both sides with respect to $y$, we obtain 
		\begin{align}\label{us to zagier}
			-{\pi \zeta(-N)\over y^2}-{2\pi \over y^2}\sum_{n=1}^{\infty}\sigma_N^{(N)}(n){e^{-2\pi n \over y}}&={1\over 2y}-{i\pi\over N }\zeta\left(-\frac{1}{N}\right) y^{{1\over N}-1}\sum_{k=0}^{N-1} e^{(2k+1)\pi i \over 2N}\nonumber
			\\
			&\quad-{2 \pi i\over N }e^{\pi i \over 2N}y^{{1\over N}-1}\sum_{k=0}^{N-1}e^{\pi  ik \over N}\sum_{n=1}^{\infty}n^{1\over N}\sum_{m=1}^{\infty}{\left( \exp\left( 2\pi i  n^{1\over N}y^{1\over N}e^{(2k+1)\pi i \over 2N}\right)\right)^m }.\nonumber
			\\
			&={1\over 2y}- {i\pi\over N }\zeta\left(-\frac{1}{N}\right) y^{{1\over N}-1}\left(i\over \sin\left(\pi\over 2N \right)  \right)  \nonumber
			\\
			&\quad+{2 \pi \over N } y^{{1\over N}-1}\sum_{j=-{(N-1)\over 2}}^{N-1\over 2} e^{\pi  ij \over N}\sum_{n=1}^{\infty}\frac{n^{1\over N} \exp\left( -2\pi   (ny)^{1\over N}e^{\pi ij \over N}\right) }{1-\exp\left( -2\pi   (ny)^{1\over N}e^{\pi ij \over N}\right)},
		\end{align}
		where in the second sum over $k$ in the last step, we replaced $k$ by $j+{N-1\over 2}$. Now multiply both sides by $-{y^2\over 2\pi}$, and then replace $y$ by $2\pi/y$ in the resulting identity to get \eqref{zagier to us} upon simplification.
	\end{proof}
	\begin{remark}\label{remark}
		If we start from our \eqref{zagier to us} by retracing the steps until \eqref{us to zagier} and then integrating both sides with respect to $y$ and doing further simplification, we obtain \eqref{Zagier} without the explicit evaluation of the constant $(2\pi)^{\frac{N-1}{2}}$. 
	\end{remark}
	
	\section{Concluding remarks}\label{cr}
We obtained explicit general transformations for the series $\sum_{n=1}^\infty \sigma_a^{(N)}(n) e^{-ny}$ in Theorems \ref{In terms of G} and \ref{Analytic continuation}. The special case $a=3N-1$ of Theorem \ref{In terms of G}, which is given in Corollary \ref{resulttm=1cor} allows, in view of \eqref{gpp3N-1}, a straightforward derivation of the asymptotic behavior of the generating function of the generalized power partitions with ``$n^{2N-1}$ copies of $n^{N}$'' as $q\to1^{-}$, that is, of $\prod\limits_{n=1}^{\infty}\left(1-q^{n^N}\right)^{-n^{2N-1}}$.

This generating function, which is introduced for the first time in this paper, raises important questions from the point of view of both analysis and combinatorics. We list them below.\\

(1) Determine the asymptotic behavior of $\prod\limits_{n=1}^{\infty}\left(1-q^{n^N}\right)^{-n^{2N-1}}$ as $q$ approaches other roots of unity.

(2) Determine the constant $c$ in Corollary \ref{asym F_N(q)}. Note that when $N=1$, Wright has shown that $c=\displaystyle2\int_{0}^{\infty}\frac{y\log(y)}{e^{2\pi y}-1}\, dy$. But for odd $N>1$, this constant could very well depend on $N$. Our methods do not allow for its explicit determination.

(3) Extend the results in \cite{bbbf} to encompass the infinite product $\prod\limits_{n=1}^{\infty}\left(1-q^{n^N}\right)^{-f_N(n)}$, where $f_N(n)$ is an arithmetic functions whose Dirichlet series satisfy properties similar to those in \cite[p.~4]{bbbf}.

(4) Find the description of the generalized plane partitions generated by the  $\prod\limits_{n=1}^{\infty}\left(1-q^{n^N}\right)^{-n^{2N-1}}$. In this paper, we have viewed the product as generating partitions with ``$n^{2N-1}$ copies of $n^{N}$'', which for $N=1$, are partitions with ``$n$ copies of $n$''. In view of the latter being equinumerous with the number of plane partitions of the associated positive integer, it certainly seems worthwhile to study the generalized plane partitions generated by the above product as in what they are, how they could be put in a one-to-one correspondence with partitions with ``$n^{2N-1}$ copies of $n^{N}$''.

(5) Find the asymptotic behavior of $F_N(q)$ for \emph{even} $N$. In this paper, we have only considered odd $N$. Moreover, the same questions in (1)-(4) above can be asked in the case when $N$ is even.\\

As a special case $a=N$ of Theorem \ref{In terms of G}, we obtained Corollary \ref{zagiereqvt} which, as mentioned in Remark \ref{remark}, is equivalent to Zagier's identity \eqref{Zagier} up to the determination of the integration constant $\frac{(N-1)}{2}\log(2\pi)$. To directly get Zagier's transformation, one way might be to obtain a transformation for $\sum_{n=1}^\infty S_{-1}^{(N)}(n)e^{-ny}$, where $S_{a}^{(N)}(n)$ is defined in \eqref{defbf}. More generally, obtaining a transformation for $\sum_{n=1}^\infty S_{a}^{(N)}(n)e^{-ny}$ might prove to be of importance.

\section*{Acknowledgements}
The first author is an INSPIRE faculty at IISER Kolkata supported by the DST grant DST/INSPIRE/04/\newline 2021/002753. The major part of this work was done when the first author was a postdoctoral fellow at IIT Gandhinagar and was funded by the SERB NPDF grant PDF/2021/001224. The second author is supported by the Swarnajayanti Fellowship grant SB/SJF/2021-22/08 of SERB (Govt. of India). The third author is supported by CSIR SPM Fellowship under the grant number SPM-06/1031(0281)/2018-EMR-I. All of the authors sincerely thank the respective funding agencies for their support. The first author also thanks IIT Gandhinagar for its support.
	

\begin{thebibliography}{00}
		
		\bibitem{gea}
		G.E.~Andrews, The theory of partitions, Addison-Wesley Pub. Co., NY, 300 pp. (1976). Reissued, Cambridge University Press, New York, 1998.
		
		\bibitem{apelblat}
		A.~Apelblat, \emph{Differentiation of the Mittag-Leffler functions with respect to parameters in the Laplace transform approach}, Mathematics, 8 (2020), 657.
		
		\bibitem{Apostol} T.~Apostol, {\em Introduction to analytic number theory}, Springer Science+Business Media, New York, 1976.
		
%		\bibitem{berndthurwitzzeta}
%		B.~C.~Berndt, \emph{On the Hurwitz zeta function}, Rocky~Mountain~J.~Math.~\textbf{2} (1972), 151--157.
		
%		\bibitem{RN_1}
%		B.~C.~Berndt, \emph{Ramanujan's notebooks, Part I}, Springer-Verlag, New York, 1985.
		
		\bibitem{bcbramsecnote}
		B.~C.~Berndt, \emph{Ramanujan's Notebooks, Part \textup{II}}, Springer-Verlag, New York, 1989.
		
		\bibitem{V}
		B.~C.~Berndt, \emph{Ramanujan's Notebooks}, Part \textup{V}, Springer--Verlag, New
		York, 1998.
		
			\bibitem{RLNII}
		George E.~Andrews and Bruce C.~Berndt, \emph{Ramanujan's Lost Notebook Part \textup{II}}, Springer, New York, 2009, 418 pp.
		
		\bibitem{bdrz1}
		B.~C.~Berndt, A.~Dixit, A.~Roy and A.~Zaharescu, \emph{New pathways and connections in number theory and analysis motivated by two incorrect claims of Ramanujan}, Adv. Math.~\textbf{304} (2017), 809--929.
		
		\bibitem{berndtstraubzeta}
		B.~C.~Berndt and A.~Straub, \emph{Ramanujan's formula for $\zeta(2n+1)$}, Exploring the Riemann zeta function, Eds. H. Montgomery, A. Nikeghbali, and M. Rassias, pp.~13--34, Springer, 2017.
		
		\bibitem{bbbf}
		W.~Bridges, B.~Brindle, K.~Bringmann and J.~Franke, \emph{Asymptotic expansions for partitions generated by infinite products}, submitted for publication, arXiv:2303:11864v1, March 21, 2023.
		
		\bibitem{chaundy}
		T.~W.~Chaundy, \emph{The unrestricted plane partition}, Q.~J.~Math.~\textbf{Ser. 3} (1932), 76--80.
		
		\bibitem{cohen}
		E.~Cohen, \emph{An extension of Ramanujan's sum}, Duke Math.~J.~\textbf{16} (1949), 85--90.
		
		\bibitem{crum}
		M.~M.~Crum, \emph{On some Dirichlet series}, J.~London Math.~Soc.~\textbf{15} (1940), 10--15.
		
		\bibitem{dav}
		H.~Davenport, \emph{Multiplicative Number Theory}, 3rd ed., Springer--Verlag, New York, 2000.
		
%		\bibitem{deninger}
%		C.~Deninger, \emph{On the analogue of the formula of Chowla and Selberg for real quadratic fields}, J.~Reine Angew. Math.~\textbf{351} (1984), 171--191.
%		
%		\bibitem{dilcher}
%		K.~Dilcher, \emph{On generalized gamma functions related to the Laurent coefficients of the Riemann zeta function}, Aequationes Math.~\textbf{48} (1994), 55--85.
		
		\bibitem{hhf1}
		A.~Dixit, R.~Gupta and R.~Kumar, \emph{Extended higher Herglotz functions \textup{I}. Functional equations}, submitted for publication.
		
		\bibitem{DGKM}
		A.~Dixit, R.~Gupta, R.~Kumar and B.~Maji, \emph{Generalized Lambert series, Raabe's cosine transform and a two-parameter generalization of Ramanujan's formula for $\zeta(2m + 1)$}, Nagoya Math.~J.~\textbf{239} (2020), 232--293.
		
	\bibitem{dkk}
	A.~Dixit, A.~Kesarwani and R.~Kumar, \emph{Explicit transformations of certain Lambert series}, Res.~ Math.~Sci.~\textbf{9}, 34 (2022) (54 pages). 
		
		\bibitem{dk03}
		A.~Dixit and R.~Kumar, \emph{Applications of the Lipschitz summation formula and a generalization of Raabe's cosine transform}, submitted for publication.
		
		\bibitem{dixitmaji1}
		A.~Dixit and B.~Maji, \emph{Generalized Lambert series and arithmetic nature of odd zeta values}, Proc.~Royal~Soc.~Edinburgh, Sect. A: Mathematics, \textbf{150} Issue 2 (2020), 741--769.
		
		\bibitem{DMV} A.~Dixit, B.~Maji and A.~Vatwani, \emph{Vorono\"{\dotlessi} summation formula for the generalized divisor function $\sigma_z^{(k)}(n)$}, submitted for publication.
		
		\bibitem{dixfer1}
		A.~L.~Dixon and W.~L.~Ferrar, \emph{Lattice-point summation formulae}, Quart.~J.~Math.~\textbf{2} (1931), 31--54.
		
		\bibitem{dorigoni}
		D.~Dorigoni and A.~Kleinschmidt, \emph{Resurgent expansion of Lambert series and iterated Eisenstein integrals}, Commun. Number Theory Phys.~\textbf{15} No. 1 (2021), 1--57.
		
%		\bibitem{dzhrbashyan}
%		M.~M.~Dzhrbashyan, \emph{Integral Transforms and Representations of Functions in the Complex Domain}, Izdat. ``Nauka", Moscow, 1966, 671 pp.
		
%		\bibitem{transseries}
%		G.~A.~Edgar, \emph{Transseries for beginners}, Real Anal.~Exchange~\textbf{35} no. 2 (2010), 253--309.
		\bibitem{gkmr}
R.~Gorenflo, A.~A.~Kilbas, F.~Mainardi and S.~V.~Rogosin, \emph{Mittag-Leffler Functions, Related Topics and Applications}, Springer Monographs in Mathematics, Springer-Verlag 2014, 443 pp.
		
		\bibitem{guinand} 
		A.~P.~Guinand, \emph{Some rapidly convergent series for the Riemann $\Xi$-function}, Quart. J. Math.(Oxford) \textbf{6} (1955), 156-160.
		
%		\bibitem{ishibashi}
%		M.~Ishibashi, \emph{Laurent coefficients of the zeta function of an indefinite quadratic form}, Acta Arith.~\textbf{106} No. 1 (2003), 59--71.
		%\bibitem{kanemitsuclosed}
		%S.~Kanemitsu, \emph{On evaluation of certain limits in closed form}, Th\'{e}orie des nombres (Quebec), 459--474, de Gruyter, Berlin, 1989.
		
		\bibitem{ramhar}
		G.H. Hardy and S. Ramanujan, Asymptotic formulæ in combinatory analysis, Proc. London Math. Soc., (2) 17, (1918), 75–115. Reprinted in Collected Papers of Srinivasa Ramanujan, Chelsea Publishing Company, New York (1962), pp. 276–309
		
		\bibitem{ktyhr}
		S.~Kanemitsu, Y.~Tanigawa and M.~Yoshimoto, \emph{On the values of the Riemann zeta-function at rational arguments}, Hardy-Ramanujan J.~\textbf{24} (2001), 11--19.
		
		\bibitem{ktyacta}
		S.~Kanemitsu, Y.~Tanigawa, and M.~Yoshimoto, \emph{On multiple Hurwitz zeta-function values at rational arguments}, Acta Arith.~\textbf{107}, No. 1 (2003), 45--67.
		
		\bibitem{koshwigleningrad}
		N.~S.~Koshliakov, \emph{Sur une int\'{e}grale d\'{e}finie et son application $\grave{a}$ la th\'{e}orie des formules sommatoires}, J.~Soc.~Phys.-Math.~Leningrad~\textbf{2}, No. 2 (1929), 123--130.
		
		\bibitem{kosh1938}
		N.~S.~Koshliakov, \emph{Note on certain integrals involving Bessel functions}, Bull. Acad. Sci. URSS Ser. Math.~\textbf{2} No. 4, 417--420; English text 421--425 (1938).
		
%		\bibitem{kratzel1}
%		E.~Kr\"{a}tzel, \emph{Dedekindsche Funktionen und Summen, I} (German), Period.~Math.~ Hungar.~\textbf{12} no. 2, (1981), 113--123.
		 
		 \bibitem{kratzel2}
		 E.~Kr\"{a}tzel, \emph{Dedekindsche Funktionen und Summen, II} (German), Period.~Math.~ Hungar.~\textbf{12} no. 3, (1981), 163--179.
		
		\bibitem {Luke} Y.~L.~Luke, \emph{The Special Functions and Their Approximations}, Vol. 1, UK Edition, Academic Press, INC. 1969.
		
	\bibitem{mccarthy}
		P.~J.~McCarthy, \emph{Arithmetische Funktionen}, Springer Spektrum, 2017.
		
		\bibitem{mieghem}
		P.~ Van Mieghem, \emph{The Mittag-Leffler function}, arXiv:2005.13330v4, September 26, 2021 (71 pages).
		
		\bibitem{NIST} F. W. J. Olver, D. W. Lozier, R. F. Boisvert and C. W. Clark, eds., {\it NIST Handbook of Mathematical Functions}, Cambridge University Press, Cambridge, 2010.	
		
		\bibitem{Prudnikov} A. P. Prudnikov, Yu. A. Brychkov and O. I. Marichev, {\em Integrals and series}, Vol. 3: More Special Functions, Gordon and Breach, New York, 1990.
		
		\bibitem{ram1918}
		S.~Ramanujan, \emph{On certain trigonometric sums and their applications in the theory of numbers}, Trans.~Cambridge~Philos.~Soc.~\textbf{22} (1918), 179-199.
		
		\bibitem{lnb}
		S.~Ramanujan, \emph{The Lost Notebook and Other Unpublished
			Papers}, Narosa, New Delhi, 1988.
		
		\bibitem{ramnote}
		S.~Ramanujan, Notebooks (2 volumes), Tata Institute of Fundamental Research, Bombay, 1957; second ed., 2012.
		
%		\bibitem{roblesroy}
%		N.~Robles and A.~Roy, \emph{Moments of averages of generalized Ramanujan sums}, Monatsh.~Math.~\textbf{182} no. 2 (2017), 433-461.
%		
%		\bibitem{shirasaka}
%		S.~Shirasaka, \emph{On the Laurent coefficients of a class of Dirichlet series}, Result.~Math.~\textbf{42} (2002), 128--138.
		
		\bibitem{temme}
		N.~M.~Temme, \emph{Special functions: An introduction to the classical functions of mathematical physics}, Wiley-Interscience Publication, New York, 1996.
		
		\bibitem{tenenbaum}
		G.~Tenenbaum, J.~Wu and Y.-L.~Li, \emph{Power partitions and saddle-point method}, J.~Number Theory~\textbf{204} (2019), 435--445.
		
%		\bibitem{watsonself}
%		G.~N.~Watson, \emph{Some self-reciprocal functions}, Quart.~J.~Math.~(Oxford)\textbf{2} (1931) 298--309.
		% for the remark that the special case $N=1$ of the kernel $H_{z}^{(N)}(x)$ is the Watson kernel
		 
		 \bibitem{vz}
		 M.~Vlasenko and D.~Zagier, \emph{Higher Kronecker ``limit" formulas for real quadratic fields}, J. reine angew. Math. 679, pp. 23--64 (2013).
		 
		 \bibitem{ccvw}
 P.~Chavan, S.~Chavan, C.~Vignat and T.~Wakhare, \emph{Dirichlet series under standard convolutions: variations on Ramanujan's identity for odd zeta values},  Ramanujan J.~\textbf{59} no. 4 (2022), 1245--1285.
 
		\bibitem{wig}
		S.~Wigert, \emph{Sur une extension de la s\'{e}rie de Lambert}, Arkiv Mat.~Astron.~Fys.~\textbf{19} (1925), 13 pp.
		
		\bibitem{wigannalen}
		S.~Wigert, \emph{Sur une nouvelle fonction enti$\grave{e}$re et son application $\grave{a}$ la th\'eorie des nombres}, Math.~Ann.~\textbf{96} No. 1 (1927), 420--429.
		
		\bibitem{Wig}
		S.~Wigert, \emph{Note sur la s\'{e}rie $\displaystyle{\sum_{n=1}^{\infty}e^{-n^kx}}$}, Tohoku Math.~J.~\textbf{38} (1933), 451--457.
		
			\bibitem{wiman}
		A.~Wiman, \emph{Uber den fundamental satz in der theorie der funcktionen $E_{a}(x)$}, Acta Math.~\textbf{29} (1905), 191--201.
		
		\bibitem{wright}
		E.~M.~Wright, \emph{Asymptotic partition formulae I. Plane partitions}, Q.~J.~Math.~\textbf{1} (1931), 177--189.
		
%		\bibitem{wright3}
%	E.~M.~Wright, \emph{Asymptotic partition formulae, III. Partitions into kth powers}, Acta~Math.~\textbf{63} (1934), 143--191.
		
		\bibitem{zagier1975}
		D.~Zagier, \emph{A Kronecker limit formula for real quadratic fields}, Math.~Ann.~\textbf{213} (1975), 153--184.
		
		\bibitem{zagierhrj}
		D.~Zagier, \emph{Power partitions and a generalized eta transformation property}, Hardy-Ramanujan J.~\textbf{44} (2021), 1-18.
		
	\end{thebibliography}
\end{document}